\documentclass[11pt, a4paper, reqno]{amsart}

\usepackage{amsrefs, amsmath, amsthm}
\usepackage{comment}
\usepackage{amsfonts, mathrsfs}
\usepackage{amssymb}
\usepackage{amscd}
\usepackage{fullpage}
\usepackage[labelformat=empty]{subfig}
\usepackage{booktabs}
\usepackage{xypic}
\usepackage{enumerate}
\usepackage[pdftex]{graphicx,epsfig}
\usepackage{cleveref}

\newcommand{\X}{\mathcal{X}}
\newcommand{\Z}{\mathbb{Z}}
\newcommand{\C}{\mathbb{C}}

\newcommand{\so}{\mathcal{O}}
\newcommand{\fc}{\mathbf{1}}

\BibSpec{unpublished}{%
+{}{\sc \PrintAuthors} {author}
+{,}{ \textrm} {title}
+{,}{ \textit} {journal}
+{,}{ } {volume}
+{}{ \IfEmptyBibField{journal}{}{\PrintDate{date}}} {transition}
+{,}{ } {pages}
+{,}{ } {status}
+{.}{} {transition}
+{}{ Preprint\IfEmptyBibField{date}{}{ \PrintDate{date}}, available at \tt} {eprint}
+{.}{} {transition}
}

\BibSpec{article}{%
+{}{\sc \PrintAuthors} {author}
+{,}{ \textrm} {title}
+{,}{ \textit} {journal}
+{,}{ } {volume}
+{}{ \IfEmptyBibField{journal}{}{\PrintDate{date}}} {transition}
+{,}{ } {pages}
+{,}{ } {status}
+{.}{} {transition}
+{}{ Preprint available at \tt} {eprint}
+{.}{} {transition}
}

\BibSpec{book}{
+{}{\sc \PrintAuthors} {author}
+{,}{ \textit} {title}
+{.}{ } {series}
+{,}{ } {volume}
+{,}{ \PrintPages } {pages}
+{.}{ } {publisher}
+{,}{ } {place}
+{,}{ } {date}
+{.}{} {transition}
}

\BibSpec{incollection}{
+{,} { in \it } {conference}
+{} {\sc \PrintAuthors} {author}
+{,} { \textrm } {title}
+{,} {in } {conference}
+{,} { \textit } {booktitle}
+{}{ } {volume}
+{,} { pp.~} {pages}
+{.}{ } {publisher}
+{,}{ } {place}
+{,} { \PrintDateB} {date}
+{.} {} {transition}
}

\BibSpec{innerbook}{
+{.} { \emph} {title}
+{.} { } {part}
+{:} { \emph} {subtitle}
+{.} { } {series}
+{,} { } {volume}
+{.} { Edited by \PrintNameList} {editor}
+{.} { Translated by \PrintNameList}{translator}
+{.} { \PrintContributions} {contribution}
+{.} { } {publisher}
+{.} { } {organization}
+{,} { } {address}
+{,} { \PrintEdition} {edition}
+{,} { \PrintDateB} {date}
+{.} { } {note}
}

\allowdisplaybreaks[3]

\newcommand{\de}{{\partial}}

\newcommand{\rd}{\mathrm{d}}
\newcommand{\ri}{\mathrm{i}}
\newcommand{\re}{\mathrm{e}}

\newcommand{\fX}{\mathfrak{X}}

\newcommand{\U}{\mathbb{U}_{\rho}^{\X,Y}}

\newcommand{\bbN}{\mathbb{N}}

\newcommand{\bbZ}{\mathbb{Z}}

\newcommand{\bbC}{\mathbb{C}}
\newcommand{\bbO}{\mathbb{O}}
\newcommand{\bbP}{\mathbb{P}}
\newcommand{\bbF}{\mathbb{F}}
\newcommand{\bbQ}{\mathbb{Q}}

\newcommand{\bbA}{\mathbb{A}}
\newcommand{\bbU}{\mathbb{U}}

\def\bary{\begin{array}} 
\def\eary{\end{array}} 
\def\ben{\begin{enumerate}} 
\def\een{\end{enumerate}}
\def\bit{\begin{itemize}} 
\def\eit{\end{itemize}}
\def\nn{\nonumber} 

\newcommand{\cY}{\mathcal{Y}}
\newcommand{\cZ}{\mathcal{Z}}
\newcommand{\cO}{\mathcal{O}}

\newcommand{\DD}{\mathcal{D}}
\newcommand{\LL}{\mathcal{L}}
\newcommand{\cS}{\mathcal{S}}

\newcommand{\cN}{\mathcal{N}}

\newcommand{\cA}{\mathcal{A}}
\newcommand{\HH}{\mathcal{H}}
\newcommand{\cB}{\mathcal{B}}
\newcommand{\cF}{\mathcal{F}}
\newcommand{\cI}{\mathcal{I}}

\newcommand{\cX}{\mathcal{X}}

\newcommand{\cM}{\mathcal M}
\newcommand{\cQ}{\mathcal Q}

\newcommand{\bfz}{\mathbf{z}}

\newcommand{\eq}[1]{\begin{equation}#1\end{equation}}
\newcommand{\ea}[1]{\begin{align}#1\end{align}}
\def\bary{\begin{array}} 
\def\eary{\end{array}} 
\def\ben{\begin{enumerate}} 
\def\een{\end{enumerate}}
\def\bit{\begin{itemize}} 
\def\eit{\end{itemize}}
\def\nn{\nonumber} 
\def\de {\partial}

\def\a{\alpha}
\def\b{\beta}
\def\g{\gamma}
\def\d{\delta}

\DeclareMathOperator*{\Res}{Res}

\theoremstyle{plain}

\newtheorem{thm}{Theorem}[section]
\newtheorem*{mt1}{Main Theorem}
\newtheorem*{thmmir}{\cref{thm:mirror}}
\newtheorem*{cocrceff}{\cref{thm:dcrccoheff}}
\newtheorem*{cocrcgerby}{\cref{thm:dcrccoh}}

\newtheorem{lem}[thm]{Lemma}

\newtheorem{prop}[thm]{Proposition}
\newtheorem{conj}[thm]{Conjecture}

\newtheorem*{conj*}{Conjecture}

\newtheorem{proposal}{Proposal}
\newtheorem{cor}[thm]{Corollary}
\newtheorem*{cor*}{Corollary}

\newtheorem{defn}{Definition}[section]

\crefname{equation}{Eq.}{Eqs.}
\crefname{eqnarray}{Eq.}{Eqs.}
\crefname{conj}{Conjecture}{Conjectures}
\crefname{lem}{Lemma}{Lemmas}
\crefname{thm}{Theorem}{Theorems}
\crefname{rmk}{Remark}{Remarks}
\crefname{prop}{Proposition}{Propositions}
\crefname{section}{Section}{Sections}
\crefname{appendix}{Appendix}{Appendices}
\crefname{cor}{Corollary}{Corollaries}
\crefname{figure}{Figure}{Figures}

\theoremstyle{definition}

\newtheorem{rmk}[thm]{Remark}

\newcommand{\GIT}[1]{/\!\!/_{\kern-.2em #1 \kern0.1em}}

\newcommand{\ch}{\mathrm{ch}}

\renewcommand{\Re}{\mathfrak{Re}}
\renewcommand{\Im}{\mathfrak{Im}}

\renewcommand{\l}{\left}
\renewcommand{\r}{\right}
\newcommand{\bra}{\left\langle}
\newcommand{\ket}{\right\rangle}

\newcommand{\rank}{\operatorname{rank}}

\newcommand{\ev}{\operatorname{ev}}
\usepackage{color}

\def\bes{\begin{subequations}}
\def\ees{\end{subequations}}

\begin{document}

\begin{abstract}

We formulate a Crepant Resolution Correspondence for open Gromov--Witten
invariants (OCRC) of toric Lagrangian branes inside Calabi--Yau 3-orbifolds by encoding the open theories
into sections of Givental's symplectic vector space. The
correspondence can be phrased as the identification of these sections via a linear morphism of
Givental spaces. We deduce from this a 
Bryan--Graber-type statement for disk invariants, which we extend
 to arbitrary topologies in the Hard Lefschetz case. Motivated by
 ideas of Iritani, Coates--Corti--Iritani--Tseng and Ruan, we furthermore propose 1) a general form
 of the morphism entering the OCRC, which arises from a geometric
correspondence between equivariant $K$-groups, and 2) an all-genus
version of the OCRC
for Hard Lefschetz targets.  We provide a complete proof of both statements in
the case of minimal resolutions of threefold $A_n$-singularities; as a
necessary step of the proof we establish the all-genus closed Crepant Resolution Conjecture with
descendents in its strongest form
for this class of examples. Our methods 
rely on a new description of the quantum $D$-modules
underlying the equivariant Gromov--Witten theory of this family of targets.

\end{abstract}

\title{Crepant Resolutions and Open Strings}

\author{Andrea Brini}
\address{A. B.: Institut Montpelli\'erain Alexander Grothendieck,
 UMR 5149 du CNRS, Universit\'e de Montpellier, 
Place Eug\`ene Bataillon,
Montpellier Cedex 5, France\\
Department of Mathematics, Imperial College London, 180 Queen's Gate London
SW7 2AZ, United Kingdom
}
\email{a.brini \_at\_ imperial ac uk}

\email{andrea.brini \_at\_ umontpellier fr}

\author{Renzo Cavalieri}
\address{R. C.: Department of Mathematics, Colorado State University, 101 Weber Building, Fort
Collins, CO 80523-1874}
\email{renzo \_at\_ math colostate edu}

\author{Dustin Ross}
\address{D. R.: Department of Mathematics, University of Michigan, 2074 East Hall, 530 Church Street, Ann Arbor, MI  48109-1043}
\email{dustyr \_at\_ umich edu}

\maketitle

\tableofcontents

\section{Introduction}
\subsection{Summary of Results}
 
This paper proposes an approach to the Crepant Resolution Conjecture for open Gromov--Witten invariants, and supports it with a series of results and verifications about threefold $A_n$-singularities and their resolutions.\\

Let $\cZ$ be a smooth toric Calabi--Yau Deligne--Mumford stack  of dimension three with
generically trivial stabilizers and semi-projective coarse moduli space, and let $L$ be an Aganagic--Vafa brane
(\cref{sec:ogw}). Fix a Calabi--Yau torus action $T$ on $\cZ$ and
denote by $\Delta_\cZ$ the free module over $H(BT)$ spanned by
the $T$-equivariant lifts of orbifold cohomology classes
of Chen--Ruan degree at most two. We define
(\cref{ssec:dcrc}) a family of elements of Givental space,
\eq{
\mathbb{F}_{L,\cZ}^{\rm disk}: H_T(\cZ) \to \HH_\cZ =  H_T(\cZ)((z^{-1})),
}
which we call the \textit{winding neutral disk potential}.  Upon appropriate
specializations of the variable $z$, $\mathbb{F}^{\rm disk}_{L,\cZ}$ encodes disk
invariants of $(\cZ,L)$ at any winding $d$. \\

Consider a \textit{crepant
  resolution diagram} $\cX \to X \leftarrow Y$, where $X$ is the coarse moduli
space of $\cX$ and $Y$ is a crepant resolution of the singularities of $X$. A
Lagrangian boundary condition $L$ is chosen on $\cX$ and  we denote by $L'$
its transform in $Y$. 
Our
version of the open crepant resolution conjecture is a comparison of the (restricted) winding neutral disk potentials.

\begin{proposal}[The OCRC]
There exists a $\bbC((z^{-1}))$-linear map of Givental spaces
$\bbO: \HH_\X \to \HH_Y$ and analytic functions $\mathfrak{h}_\cX: \Delta_\cX
\to \bbC$, $\mathfrak{h}_Y: \Delta_Y \to \bbC$ such that
\eq{
\mathfrak{h}_Y^{1/z}{\mathbb{F}_{L,Y}^{\rm disk}}\big|_{\Delta_Y}= \mathfrak{h}_\cX^{1/z} \bbO\circ \mathbb{F}_{L,\X}^{\rm disk}\big|_{\Delta_\cX}
}
upon analytic continuation of quantum cohomology parameters.
\end{proposal}

Further, we conjecture (\cref{conj:iri}) that both $\bbO$ and
$\mathfrak{h}_\bullet$ are completely determined
by the classical toric geometry of $\cX$ and $Y$. In particular, we give a 
prediction for the transformation $\bbO$ depending on a choice of identification of the
$K$-theory lattices of $\cX$ and $Y$. \\

When $\cX$ is a Hard Lefschetz
Calabi--Yau orbifold, the OCRC comparison extends to  all of
$H_T(\cZ)$. This, together with ideas of Coates--Iritani--Tseng and Ruan (\cref{conj:scr}),
motivates
a comparison for potentials encoding invariants of maps with
arbitrary topology.

\begin{proposal}[The quantum OCRC]

Let $\X \rightarrow X \leftarrow Y$ be a Hard Lefschetz diagram for which the OCRC  holds. Defining  $\bbO^{\otimes \ell}= \bbO(z_1)\otimes\ldots\otimes \bbO(z_\ell)$, we have:
\eq{
{\mathbb{F}_{L',Y}^{g,\ell}}= \bbO^{\otimes \ell}\circ \mathbb{F}_{L,\X}^{g,\ell},
}
where the winding neutral open potential $\mathbb{F}^{g,\ell}$ is the genus-$g$, $\ell$-boundary components analog of $
\bbF^{\rm disk}$  defined in \cref{sec:hg}. \\
\end{proposal}

Consider now the family of threefold $A_n$ singularities, where $\cX=[\bbC^2/\bbZ_{n+1}]\times \bbC$ and $Y$ is its canonical minimal
resolution. 
\begin{mt1}
The OCRC, \cref{conj:iri} and the quantum OCRC hold for the $A_n$-singularities for any choice of Aganagic--Vafa brane on $\cX$.
\end{mt1}
Our verification of the OCRC and \cref{conj:iri} in this family of examples follow from
\cref{prop:wncrc,thm:sympl}. The quantum OCRC is a consequence of the closed
string quantum CRC  in its strongest version (\cref{conj:scr}, \cref{thm:cqcrc}), which we establish in \cref{sec:quantum}. From this, we deduce a series of comparisons of
more classical generating functions for open invariants, in the spirit of Bryan--Graber's formulation of the CRC. \\

In \cref{ssec:ocrc} we define the \textit{cohomological disk potential}
$\cF_{L}^{\rm disk}$ - a cohomology valued generating function for disk
invariants  that ``remembers" the twisting and the attaching fixed point of
an orbimap. We also consider the coarser \textit{scalar potential}
(see \cref{sec:ogw}), which keeps track of the winding of the orbimaps but
forgets the twisting and attaching point.
There are essentially two different choices for the Lagrangian boundary condition on $\X$; the simpler case occurs when $L$ intersects one of the effective legs of the orbifold. In this case we have the following result.
\begin{cocrceff}[effective leg]
{ Identifying} the winding parameters and setting $\bbO_\Z(\mathbf{1^k})=P^{n+1}$ for every $k$, we have:
\eq{
\cF_{L',Y}^{\rm disk}(t,y,\vec{w}) = \bbO_\bbZ \circ \cF_{L,\X}^{\rm disk}(t,y,\vec{w}). 
  }
\end{cocrceff}
It is immediate to observe that the scalar potentials coincide (\cref{cor:esc}).\\

The case when $L$ intersects the ineffective leg of the orbifold is more subtle.
\begin{cocrcgerby}[ineffective leg]
We exhibit a matrix  $\bbO_\bbZ$ of roots of unity and  a specialization of the winding parameters depending on the equivariant weights such that
\eq{
\cF_{L',Y}^{\rm disk}(t,y,\vec{w}) = \bbO_\bbZ \circ \cF_{L,\X}^{\rm disk}(t,y,\vec{w}) .
}
\end{cocrcgerby}
The comparison of scalar potentials in this case does not hold
anymore. Because of the special form of the matrix $\bbO_\bbZ$ we  deduce in
\cref{cor:sc} that the scalar disk potential for $Y$ corresponds to
the contribution to the potential for $\cX$ by the untwisted disk maps. 
Our proof of the quantum CRC makes it 
an exercise in book-keeping to extend the statements of
\cref{thm:dcrccoh,thm:dcrccoheff} to  compare generating functions for  open
invariants 
with arbitrary genus and number of boundary components, even treating all boundary Lagrangian conditions at the
same time. The main tool used in the proof of our main theorem 
is a new global
description of the gravitational quantum cohomology of the $A_n$ geometries,
which enjoys a number of remarkable features, and may have an independent interest {\it per se}.

\begin{thmmir}
By identifying the $A$-model moduli space with a genus zero double Hurwitz space, we construct a global quantum $D$-module $(\cF_{\lambda,\phi}, T\cF_{\lambda,\phi}, \nabla^{(g,z)},H(,)_{g})$ which is locally isomorphic to $\mathrm{QDM}(\cX)$ and $\mathrm{QDM}(Y)$ in appropriate neighborhoods of the orbifold and large complex structure points.
\end{thmmir}

\subsection{Context, motivation and further discussion}

Open Gromov--Witten (GW) theory 
intends to study holomorphic maps from bordered Riemann surfaces, where the image of the boundary is constrained
 to lie in a Lagrangian submanifold of the target. While some general foundational
 work has been done \cite{Solomon:2006dx, MR2425184}, at this point most
 of the results in the theory rely on additional structure. In \cite{hht1, hht2}  Lagrangian Floer theory is employed to study the case when the boundary condition is a fiber of the moment map.
In the toric context, a mathematical approach
\cite{Katz:2001vm, Diaconescu:2003qa, MR2861610,r:lgoa} to construct operatively
a virtual counting theory of open maps is via the use of
localization\footnote{Alternatively, open string invariants in the manifold
  case can be defined using relative stable morphisms \cite{Li:2001sg}.}. 
A variety of striking relations have been verified connecting open GW theory and several other types of invariants,
including open $B$-model invariants and matrix models \cite{Aganagic:2000gs,
  Aganagic:2001nx, Lerche:2001cw, Bouchard:2007ys, fang2012open}, quantum knot invariants
\cite{Gopakumar:1998ki, Marino:2001re}, and ordinary
Gromov--Witten and Donaldson--Thomas theory via ``gluing along the boundary''  \cite{Aganagic:2003db,
  Li:2004uf, moop}.\\

Since Ruan's influential conjecture \cite{MR2234886}, an intensely studied
problem in Gromov--Witten theory has been to determine the relation between GW invariants of target spaces
related by a crepant birational transformation (CRC). The most general
formulation of the CRC is framed in terms of Givental formalism
(\cite{MR2529944}; see also \cite{coates2007quantum} for an
expository account); the conjecture has been proved in
a  number of examples \cite{MR2510741, MR2529944, MR2486673} and has by now gained folklore status, with
a general proof in the toric setting announced for some time \cite{ccit2,cij}. A natural question one can ask is whether
similar relations exist in the context of open Gromov--Witten theory. Within
the toric realm, physics arguments based on open mirror symmetry
\cite{Bouchard:2007ys, Bouchard:2008gu, Brini:2008rh} have given strong indications that
some version of the Bryan--Graber \cite{MR2483931} statement of the crepant
resolution conjecture should hold at the level of open invariants. This was
proven explicitly for the crepant resolution of the Calabi--Yau orbifold $[\bbC^3/\bbZ_2]$
in \cite{cavalieri2011open}. 
Around the same time, it was suggested 
\cite{Brini:2011ij, talk-banff} that a general statement of a Crepant Resolution
Conjecture for open invariants  should have a natural formulation within
Givental's formalism, as in \cite{MR2510741, MR2529944}. Some implications of this
philosophy were verified in
\cite{Brini:2011ij} for the crepant resolution $\cO_{\bbP^2}(-3)$ of the orbifold
$[\bbC^3/\bbZ_3]$. \\

The OCRC we propose here is a natural
extension to open Gromov--Witten theory of the Coates--Iritani--Tseng
approach \cite{MR2529944} to Ruan's conjecture.
The  observation that the disk function of \cite{MR2861610,r:lgoa} can be interpreted as an endomorphism of Givental space makes the OCRC statement follow almost tautologically from  the Coates--Iritani--Tseng/Ruan picture of the ordinary 
CRC via toric mirror symmetry \cite{MR2529944}. 
The more striking aspect of our conjecture is then that the linear function
$\bbO$ comparing the winding neutral disk potentials is considerably simpler
than the symplectomorphism $\U$  in the closed CRC and it is characterized in
terms of {\it purely classical data}: essentially,  the equivariant Chern
characters of $\cX$ and $Y$. This is intimately related to Iritani's {\it \cref{eq:iritanisymp}}
 that the analytic continuation for the flat sections of the global quantum $D$-module is realized via the composition of $K$-theoretic central charges.
While Iritani works non-equivariantly on proper targets, his  constructions carry through to the equivariant setting, and inspire us to make  Conjecture \ref{conj:iri}. We point out that our results do not rely on the validity of  Iritani's proposal, but rather support the fact that an equivariant version of his proposal should hold.\\

Iritani's theory is inspired and consistent with the idea of  global mirror
symmetry, i.e. that there 
exists a global quantum $D$-module on the
$A$-model moduli space which locally agrees with the Frobenius structure given
by quantum cohomology. In order to verify his proposal in the
equivariant setting relevant for this paper, we give a new construction
of this global structure: motivated by the connection of the Gromov--Witten
theory of $A_n$-surface singularities to certain integrable systems of
Toda-type \cite{Brini:2014mha}, we
realize the global $A$-model quantum $D$-module as a system of one-dimensional Euler--Pochhammer
hypergeometric periods. This mirror picture possesses several 
remarkable properties which  enable us to verify in detail our proposals for the
open CRC, as well as proving along the way various results of independent
interest on Ruan's conjecture as well as its refinements (Iritani, Coates--Iritani--Tseng) and
extensions (the higher genus CRC).
 First off, the computation of the analytic
continuation of flat sections is significantly simplified with respect to the standard toric mirror
symmetry methods based on the Mellin--Barnes integral: in particular, the
Hurwitz space picture gives closed-form expressions for the analytic
continuation of flat sections upon crossing a parametrically large number of walls.
A useful consequence
for us is \cref{thm:sympl}, which furnishes an explicit form for the morphism $\U$
of Givental's spaces of \cite{MR2510741}, as well as a verification of
Iritani's proposal \cite{MR2553377} in the fully equivariant
setting. 
Furthermore, the monodromy action on branes (and therefore equivariant $K$-theory) gets identified with the
Deligne--Mostow monodromy of hypergeometric integrals, thereby giving
a natural action of the pure braid group with $n+2$ strands on the equivariant $K$-groups of $A_n$-resolutions.
Finally, proving the strong version of the quantized Crepant Resolution Conjecture
(\cref{conj:scr}) is reduced by
\cref{thm:mirror} to a calculation in Laplace-type asymptotics. To our
knowledge, this provides the first example where a full-descendent version
of Ruan's conjecture is established to all genera\footnote{For non-descendent
  invariants, an all-genus Bryan--Graber-type statement for $A_n$-surface resolutions
  was proved by Zhou \cite{zhou2008crepant}. In an allied context,
  Krawitz--Shen \cite{krawitz2011landau} have established an all-genus LG/CY correspondence for
  elliptic orbifold lines.}.

\subsection{Relation to ongoing and other work} 

A proof of the all-genus Ruan's conjecture and of the quantum OCRC for
the other case of a Hard Lefschetz Crepant Resolution of toric Calabi--Yau
3-folds -- the $G$-Hilb resolution of $[\bbC^3/G]$ with $G=\bbZ_2 \times
\bbZ_2$ -- will be offered in the companion paper \cite{Brini:2014fea}, where the OCRC will also be proven for
a family of non-Hard Lefschetz targets. 

In the current form, the {\it winding neutral disc potential}, which encodes information about disc invariants, depends on the choice of an Aganagic-Vafa brane incident to one of the torus invariant lines. In other words, we have a different object corresponding to each phase of the open moduli. It would be desirable to have a construction  of the winding neutral disc potential that is independent of the choice of Lagrangian, and to obtain the various  boundary conditions as specializations, so as to witness more explicitly the phase transitions in the open moduli. We are currently investigating this proposal and have some positive results in the case of target $\bbC^3$.

We have also been made aware of the existence of a number of projects related in various ways to
the subject of this paper. In a forthcoming paper, Coates--Iritani--Jiang will
establish Iritani's proposal on the relation between the $K$-group McKay
correspondence and the CRC in the fully equivariant setting for
general semi-projective toric varieties. Ke--Zhou \cite{kezhou} have announced a proof of
the quantum McKay correspondence for disk invariants on effective outer legs for semi-projective toric
Calabi--Yau $3$-orbifolds using results of \cite{fang2012open}; this is the
case where the comparison of cohomological disk potentials of the OCRC is simplified to an
identification of the scalar disk potentials. Very recently, a similar 
statement for scalar potentials was obtained by Chan--Cho--Lau--Tseng in
\cite{2013arXiv1306.0437C} as an application of their construction of a class
of non-toric Lagrangian branes inside toric
Calabi--Yau 3-orbifolds. 
This opens up the suggestive hypothesis that our setup for the OCRC may be
generalized beyond the toric setting considered here.

\subsection{Organization of the paper} This paper is organized as follows. 
\cref{sec:crc} is a presentation of various versions of the ordinary (closed
string) Crepant Resolution Conjecture that are addressed in this paper. In
\cref{sec:ocrc} we present our proposal for the Open Crepant Resolution
Conjecture, whose consequences we analyze in detail in \cref{sec:An} for the
case of $A_n$-resolutions. Proofs of the statements contained here are offered in
\cref{sec:j,sec:quantum}: \cref{sec:j} is devoted to the construction of the
Hurwitz-space mirror, which is used to verify our prediction on the form of
the morphism
$\bbO$, while in \cref{sec:quantum} the quantum CRC and OCRC are established
by combining the tools of \cref{sec:j} with Givental's quantization
formalism. Relevant background material on Gromov--Witten theory and quantum
$D$-modules is reviewed in \cref{sec:back}, while \cref{sec:an} collects
mostly notational material on the toric geometry concerning our examples. A
technical result on the analytic continuation of hypergeometric integrals
required in the proof of \cref{thm:sympl} is discussed in \cref{sec:anFD}.

\subsection*{Acknowledgements} We would like to thank Hiroshi
Iritani, Yunfeng Jiang, \'Etienne Mann, Stefano Romano, Ed Segal, Mark
Shoemaker, and in particular Tom Coates for useful discussions,
correspondence and/or explanation of their work. We are also grateful to Bohan
Fang, Melissa Liu and Zhenyu Zong for correspondence after the appearance of
their work \cite{Fang:2016svw, zz}, which led to an improved version of our manuscript in the
discussion of \cref{rmk:existR}. This project originated from discussions at the Banff Workshop on ``New recursion
formulae and integrability for Calabi--Yau manifolds'', October 2011; we are
grateful to the organizers for the kind invitation and the great scientific
atmosphere at BIRS.   A.~B.~has been supported by a Marie Curie Intra-European Fellowship
under Project n. 274345 (GROWINT).  R.~C.~ has been supported by NSF grant DMS-1101549. D.~R.~ has been supported by NSF RTG grants DMS-0943832 and DMS-1045119.  Partial support from the GNFM-INdAM under the
Project ``Geometria e fisica dei sistemi integrabili'' is also acknowledged. \\

\section{Gromov--Witten theory and Crepant Resolutions: setup and conjectures}

\label{sec:crc}

This section collects and ties together various incarnations of the CRC that
we wish to focus on. We assume here familiarity with Gromov--Witten theory and
Givental's formalism; relevant background material is collected in \cref{sec:back}.\\

Consider a toric Gorenstein orbifold $\cX$, and let $X\leftarrow \nolinebreak
Y$ be a crepant resolution of its coarse moduli space. For $\cZ=\cX,Y$, fix an algebraic $T\simeq\bbC^*$ action with
zero-dimensional fixed loci such that the resolution morphism is
$T$-equivariant. The equivariant Chen--Ruan
  cohomology ring $H(\cZ) \triangleq H^{\rm orb}_{T}(\cZ)$ of
  $\cZ$ is a rank-$N_\cZ\triangleq \rank_{\bbC[\nu]} H(\cZ)$ free module over
  $H(BT)\simeq \bbC[\nu]$, where $\nu=c_1(\cO_{BT}(1))$ is the equivariant
  parameter. The genus zero Gromov--Witten theory of $\cZ$ defines a deformation of
  the ring structure on $H(\cZ)$, and equivalently, the existence of a
  distinguished family of flat structures on its tangent bundle. We shall fix notation as follows:

\begin{center}
\begin{tabular}{cp{9.5cm}lc}
$\eta$ & the flat pairing on the space of vector fields $\cX(H(\cZ))$ induced
  by the 
Poincar\'e pairing on $\cZ$ & & \cref{eq:pair} \\
$\circ_\tau$ & the quantum product at $\tau\in H(\cZ)$ & & \cref{eq:qprod1,eq:qprod2}\\
$\nabla^{(\eta,z)}$ & the Dubrovin connection on $\cX(H(\cZ))$ & & \cref{eq:defconn1}\\
$\mathrm{QDM}(\cZ)$ & the quantum $D$-module structure on $\cX(\cZ)$ induced
by $(\eta, \circ_\tau)$ & & \cref{eq:QDE} \\
$\cS_\cZ$ & the vector space of horizontal sections of $\nabla^{(\eta,z)}$\\
$H(,)_{\cZ}$ & the canonical pairing on $\cS_\cZ$ induced by $\eta$ & & \cref{eq:pairDmod} \\
$J^\cZ$ & the big $J$-function of $\cZ$ & & \cref{eq:Jfun1} \\
$S_\cZ$ & the fundamental solution ({\it $S$-calibration}) of
$\mathrm{QDM}(\cZ)$  & & \cref{eq:fundsol} \\
$\HH_\cZ$ & Givental's symplectic vector space of $\cZ$ & & \crefrange{eq:givsp}{eq:sympform}\\
$\LL_\cZ$ & the Lagrangian cone associated to $QH(\cZ)$ & & \cref{eq:lcone}\\
$\mathrm{Sp}_\pm(\HH_\cZ)$ & the positive/negative symplectic loop group of
$\HH_\cZ$ & & \cref{sec:secRq}\\
$\Delta_\cZ$ & the free module over $\bbC[\nu]$ spanned by
$T$-equivariant lifts of orbifold cohomology classes
with $\mathrm{deg}^{\rm CR} \leq 2$ 
\end{tabular}
\end{center}

\subsection{Quantum $D$-modules and the CRC}

Ruan's Crepant Resolution Conjecture  can be phrased as the existence of a {\it global quantum $D$-module}
underlying the quantum $D$-modules of $\cX$ and $Y$. This is a 4-tuple
$(\cM_A, F, \nabla, H(,)_F)$ given by 
%
a connected complex analytic space $\cM_A$ and
a holomorphic vector bundle $F\to \cM_A$, endowed with a flat
$\cO_{\cM_A}$-connection $\nabla$ and a non-degenerate $\nabla$-flat inner
product $H(,)_F \in \mathrm{End}(F)$.
%
\begin{conj}[The Crepant Resolution Conjecture]
There exist a global quantum $D$-module $(\cM_A, F, \nabla, H(,)_F)$ such that
for open subsets $V_\cX$, $V_Y \subset \cM_A$ we locally have
\begin{align}
(\cM_A, F, \nabla, H(,)_F)\big|_{V_\cX} &\simeq \mathrm{QDM}(\cX), \\
(\cM_A, F, \nabla, H(,)_F)\big|_{V_Y} &\simeq \mathrm{QDM}(Y).
\label{eq:gqdm}
\end{align}
\end{conj}

In particular, any 1-chain $\rho$ in $\cM_A$ with ends in $V_\cX$ and $V_Y$ gives an analytic continuation map
of $\nabla$-flat sections 
$\bbU^{\cX, Y}_{\cS,\rho}:\
\cS_\cX\to \cS_Y
$,
which is an isometry of $H(,)_F$ 
and identifies the quantum $D$-modules of
$\cX$ and $Y$. \\

Even when \cref{eq:gqdm} holds, there may be an obstruction to extend the isomorphism of small quantum
products to big quantum cohomology which
is relevant in our formulation of the OCRC. Locally
around the large radius limit point of $\cX$ and $Y$,
canonical trivializations of the global flat connection $\nabla$ are given by 
a system of flat coordinates for the small Dubrovin connection. Generically
they are not {\it mutually flat}: on the overlap 
$V_\cX\cap V_Y$, the relation between the two coordinate systems is
typically not affine over $\bbC(\nu)$, and as a result the induced Frobenius structures on
$H(\cX)$ and $H(Y)$ may be inequivalent. In favorable situations, for example
when the coarse moduli space $Z$ is semi-projective, the two charts are related by a conformal factor
$\mathfrak{h}_Y \mathfrak{h}_\cX^{-1}$ for local functions $\mathfrak{h}_\cX
\in \cO_{V_\cX}, \mathfrak{h}_Y \in \cO_{V_Y}$  which are
in turn completely determined by the toric combinatorics defining $X$ and $Y$
as  GIT quotients (\cref{eq:hz}). A sufficient condition \cite{MR2529944}
for the two Frobenius
structures to coincide (i.e. $\mathfrak{h}_\cX=\mathfrak{h}_Y$) is given by the Hard Lefschetz criterion for $\cX$, 
\eq{
\mathrm{age}(\theta) - \mathrm{age}(\mathrm{inv}^*\theta) = 0,
}
for any class $\theta\in H(\cX)$.

\subsection{Integral structures and the CRC}
\label{sec:intstr}

In \cite{MR2553377}, Iritani uses $K$-groups to define an integral structure
in the quantum D-module associated to the Gromov--Witten theory of a smooth and proper
Deligne--Mumford stack $\cZ$; we recall 
 the discussion in \cite{MR2553377, MR2683208}. Write $K(\cZ)$ for the Grothendieck group of topological vector bundles
$V\to\cZ$ and consider the map ${\mathscr{F}}:K(\cZ)\to H(\cZ)\otimes\C((z^{-1}))$ given by 
\eq{\label{eq:stackymukai}
{\mathscr{F}}(V)\triangleq (2\pi)^{-\frac{\dim\cZ}{2}}z^{-\mu} \overline{\Gamma}_\cZ\cup(2\pi\ri)^{\deg/2}\mathrm{inv}^*\ch(V),
}
where $\ch(V)$ is the orbifold Chern character, $\cup$ is the topological cup
product on $I\cZ$, and 
\ea{
\label{eq:gammaT}
\overline{\Gamma}_\cZ &\triangleq  \bigoplus_v\prod_f\prod_\delta\Gamma(1-f+\delta), \\
\mu & \triangleq \left(\frac{1}{2}\deg(\phi)-\frac{3}{2}\right)\phi;
}
 the sum in \cref{eq:gammaT} is over all connected components of the inertia stack, the left
product is over the eigenbundles in a decomposition of the tangent bundle $T\cZ$
with respect to the stabilizer action (with $f$ the rational weight of the action on the eigenspace; note that $1-f$ is always strictly positive and hence $\overline{\Gamma}_\cZ$ is an invertible function in a neighborhood of $0$), and the
right product is over all of the Chern roots $\delta$ of the
eigenbundle. Via the fundamental solution \cref{eq:fundsol} this induces a map
to the space of flat sections of $\mathrm{QDM}(\cZ)$; its image is a lattice \cite{MR2553377}
in $\cS_\cZ$, which Iritani dubs the {\it $K$-theory integral structure} of
  $QH(\cZ)$.
\\

Iritani's theory has important implications for the Crepant Resolution
Conjecture. At the level of integral structures, the analytic continuation map
$\bbU_{\cS,\rho}^{\cX, Y}$ of flat sections should be induced by an isomorphism
$\bbU_{K,\rho}^{\cX,Y}: K(Y) \to K(\cX)$ at the $K$-group level. The Crepant Resolution Conjecture can then be phrased in terms of the
existence of an identification of the integral local systems underlying
quantum cohomology, which, according to \cite{MR2553377}, 
should take the shape of a natural geometric
correspondence between $K$-groups.

\subsection{The symplectic formalism and the CRC}
The symplectic geometry of Frobenius manifolds gives the Crepant Resolution Conjecture  a natural formulation in terms of
morphisms of Givental spaces, as pointed out by
Coates--Corti--Iritani--Tseng \cite{MR2510741, MR2529944} (see also
\cite{coates2007quantum} for a review). 
\begin{conj}[\cite{MR2529944}]\label{conj:ccit}
There exists a $\bbC((z^{-1}))$-linear symplectic
isomorphism of Givental spaces $\bbU_\rho^{\cX,Y}:\HH_\cX\rightarrow \HH_Y,$
matching the Lagrangian cones of $\cX$ and $Y$ upon a suitable analytic
continuation of small quantum cohomology parameters:
\eq{
\U(\mathcal{L}_\cX)=\mathcal{L}_Y.
\label{eq:Uc}
}
\end{conj}
This version of the CRC is equivalent to the quantum $D$-module approach via the
fundamental solutions, which give a canonical $z$-linear identification
\eq{\label{eq:givetosect}
S_\cZ(\tau,z):\HH_\cZ\stackrel{\cong}{\longrightarrow}\mathcal{S}_\cZ.
}
translating the analytic continuation map $\bbU_{\cS,\rho}^{\cX,Y}$ to a symplectic
isomorphism of Givental spaces $\bbU_\rho^{\cX,Y}$. 
Iritani's theory of integral structures proposes that the symplectic
isomorphism  $\bbU_\rho^{\cX,Y}$ should be induced from a natural equivalence at the level of $K$ lattices of $\cX$ and $Y$, as illustrated in \cref{eq:intstructure}:
\begin{proposal}[\cite{MR2553377}]\label{eq:iritanisymp}
Inverting the central charge $\mathscr{F}_{\cX}$, one obtains:
\eq{
\bbU_\rho^{\cX,Y}={\mathscr{F}}_Y\circ\bbU_{K,\rho}^{\cX,Y}\circ{\mathscr{F}}_\cX^{-1}.
}
\end{proposal}

\begin{figure}[bt]
\eq{\begin{xy}
(0,40)*+{K(\cX)}="a"; (40,40)*+{K(Y)}="b";
(0,20)*+{\HH_\cX}="e"; (40,20)*+{\HH_Y}="f";
(0,0)*+{\cS_\cX}="c"; (40,0)*+{\cS_Y}="d";
{\ar^{\bbU_{K,\rho}^{\cX,Y}} "a";"b"};
{\ar^{\bbU_{\rho}^{\cX,Y}} "e";"f"};
{\ar_{{\mathscr{F}}_\cX} "a";"e"};{\ar^{{\mathscr{F}}_Y} "b";"f"};
{\ar_{S_\cX(x,z)} "e";"c"};{\ar^{S_Y(t,z)} "f";"d"};
{\ar^{  \bbU_{\cS,\rho}^{\cX,Y} } "c";"d"}; \nonumber
\end{xy}
}
\caption{Analytic continuation of flat sections, symplectomorphism of Givental spaces and comparison of integral structures.}
\label{eq:intstructure}
\end{figure}

A case of particular interest for us is the following. Suppose that $c_1(\cX)=0$, $\mathrm{dim}_\bbC\cX=3$ and assume further that
the $J$-functions $J^\cZ$, for $\cZ$ either $\cX$ or $Y$, and $\U$ admit well-defined non-equivariant limits,
\eq{
J_{\rm n-eq}^\cZ(\tau,z) \triangleq \lim_{\nu\to 0}J^\cZ(\tau,z), \qquad \bbU^{\cX,Y}_{\rho,0} \triangleq \lim_{\nu\to 0} \U. 
}
By the string equation and dimensional constraints, $\re^{-\tau^0/z} J_{\rm n-eq}^\cZ(\tau,z)$ is a Laurent
polynomial of the form \cite[\S10.3.2]{MR1677117}
\eq{
J_{\rm n-eq}^\cZ(\tau,z) = \re^{-\tau^0/z}\l(z + \sum_{i=1}^{N_\cZ-1}\l(\tau^i  + \frac{\mathfrak{f_i}^\cZ(\tau)}{z}\r)\phi_i + \frac{\mathfrak{g}^\cZ(\tau)}{z^2}\mathbf{1}_\cZ\r),
}
where $\mathfrak{f}^\cZ(\tau)$ and $\mathfrak{g}^\cZ(\tau)$ are
analytic functions around the large radius limit point of $\cZ$. Restricting $J_{\rm n-eq}^\cZ(\tau,z)$
to $\Delta_\cZ$ and picking up a branch $\rho$ of analytic continuation of the
quantum parameters, the vector-valued analytic function $\cI_\rho^{\cX,Y}$
defined by
\eq{
\begin{xy}
(0,0)*+{\Delta_\cX}="a"; (40,0)*+{\Delta_Y}="b";
(0,20)*+{\HH_\cX}="c"; (40,20)*+{\HH_Y}="d";
{\ar^{\cI_\rho^{\cX,Y}} "a";"b"};
{\ar_{J_{\rm n-eq}^\cX\big|_{\Delta_\cX}} "a";"c"};{\ar^{J_{\rm n-eq}^Y\big|_{\Delta_Y}} "b";"d"};
{\ar^{ \mathfrak{h}_\X^{1/z}\bbU^{\cX,Y}_{\rho,0}   \mathfrak{h}_Y^{-1/z}} "c";"d"};
\end{xy}
\label{eq:iddelta}
}
gives an analytic
isomorphism\footnote{Explicitly,  matrix entries $(\bbU^{\cX,Y}_{\rho,0})_{ij}$ of
$\bbU^{\cX,Y}_{\rho,0}$ are monomials in $z$; call $\mathfrak{u}_{ij}$ the
coefficient of such monomial. Then \cref{eq:iddelta} boils down to the
statement that quantum cohomology parameters 
$\tau^\bullet_i$ in $\Delta_\bullet$ for $i=1, \dots, l_Y$ are identified as 
\eq{
\tau^Y_i = (\cI^{\cX,Y}_\rho \tau^\cX)_i \triangleq
\mathfrak{u}_{i0}+
\sum_{j=1}^{l_Y}\mathfrak{u}_{ij} (\tau^\cX)^j+
\sum_{k=l_Y+1}^{N_Y-1}\mathfrak{u}_{ik} \mathfrak{f}^\cX_k(\tau^\cX).
\label{eq:changevargen}
}
Since $\deg (\bbU^{\cX,Y}_{\rho,0})_{ij}>0$ for $j>l_Y$, in the Hard Lefschetz
case the condition that the coefficients of $\U$ are Taylor series in $1/z$
implies that $\mathfrak{u}_{ik}=0$ for $k>l_Y$.
} between neighborhoods $V_\cX$, $V_Y$ of the
projections of the large radius points of $\cX$ and $Y$ to $\Delta_\cX$ and
$\Delta_Y$. 
When $\cX$ satisfies the Hard--Lefschetz condition, the coefficients of $\U$ contain
only non-positive powers of $z$  \cite{MR2529944} and the non-equivariant limit coincides with the
$z\to\infty$ limit; then the isomorphism
$\cI_\rho^{\cX,Y}$ extends to the full
cohomology rings of $\cX$ and $Y$, and induces an affine linear
 change of variables
$\widehat{\cI}_\rho^{\cX,Y}$, which gives 
an isomorphism of Frobenius
manifolds.

\subsection{Quantization and the higher genus CRC}

\cref{conj:ccit} shapes the genus zero CRC as the existence of a
classical canonical transformation identifying the Givental phase spaces of
$\cX$ and $Y$. As the higher genus theory is obtained from the genus zero
picture by quantization, it is expected that the
full Gromov--Witten partition functions should be identified, upon analytic
continuation, via a {\it quantum} canonical transformation identifying the Fock
spaces, and that such quantum transformation is related to the 
quantization of the symplectomorphism $\U$ in \cref{eq:Uc}. 
\begin{conj}[\cite{MR2529944, coates2007quantum}]
Let
$\U = \bbU_- \bbU_0 \bbU_+$ be the Birkhoff factorization of $\U$.
Then
\eq{
Z_Y = \widehat{\bbU_-}\widehat{\bbU_0} \widehat{\bbU_+} Z_\cX.
\label{eq:qcrc}
}
\end{conj}
A much stronger statement stems from \cref{eq:qcrc} in the Hard Lefschetz case,
when $\bbU_+=\mathbf{1}_{\HH_\cX}$ \cite[Theorem~5.10]{MR2529944}. 
\begin{conj}[The Hard Lefschetz quantized CRC]\label{conj:scr}
Let $\cX \rightarrow X \leftarrow Y$ be a Hard Lefschetz crepant resolution
diagram. Then
\eq{
Z_Y =  \widehat{\U} Z_\cX
\label{eq:sqcrc}
}
\end{conj}

By Proposition~5.3 in
\cite{MR1901075}, \cref{eq:sqcrc} gives, up to quadratic genus zero terms,
\eq{\label{eqpartscr}
Z_Y = Z_\cX|_{t_\cX = \l[\mathbb{U}_{\cX, Y}^{-1} t_Y\r]_+}}
where $[f(z)]_+$ denotes the projection $[]_+:\bbC((z))\to \bbC[[z]]$. In
other words, \cref{conj:scr} states that the full descendent partition function
of $\cZ$ and $Y$ coincide to all genera, upon
analytic continuation and the identification of the Fock space variables
dictated by the classical symplectomorphism \cref{eq:Uc}.


\section{The Open Crepant Resolution Conjecture}
\label{sec:ocrc}

\subsection{Open string maps and Givental's formalism}

\label{ssec:dcrc}

\subsubsection{Open Gromov--Witten theory of toric 3-orbifolds}
\label{sec:ogw}

For a three-dimensional toric Calabi--Yau variety, open Gromov--Witten invariants
are defined ``via
localization" in \cite{Katz:2001vm, Diaconescu:2003qa}. This theory
has been first introduced for orbifold targets in \cite{MR2861610} and developed in
full generality in \cite{r:lgoa} (see also \cite{fang2012open} for recent
results in this context). 

   Boundary conditions are given by choosing special type of Lagrangian
   submanifolds introduced by Aganagic--Vafa in
   \cite{Aganagic:2000gs}. 
 These Lagrangians are defined locally in a formal neighborhood of each torus invariant line: in particular if $p$ is a torus fixed point adjacent to the torus fixed line $l$, and the local coordinates at $p$ are $(z,u,v)$, then $L$ is defined to be the fixed points of the anti-holomorphic involution
   \eq{
   (z,u,v)\rightarrow (1/\overline{z}, \overline{zu}, \overline{zv})
   }
   defined away from $z=0$.  Boundary conditions can then be thought of as ``formal'' ways
   of decorating the web diagram of the toric target. \\
     
 Loci of fixed maps are described in terms of closed
curves mapping to the compact edges of the web diagram in the usual way and disks mapping rigidly to
the torus invariant lines with Lagrangian conditions. Beside Hodge integrals coming from the contracting
curves, the contribution of each fixed locus to the invariants has a factor
for each disk, which is constructed as follows. The map from the disk to a neighborhood
of its image is viewed as the quotient  via an involution of a map of a
rational curve to a canonical target. The obstruction theory in ordinary
Gromov--Witten theory admits a natural $\Z_2$ action, and the equivariant Euler
class of the involution invariant part of the obstruction theory is chosen as
the localization contribution from the disk \cite[Section~2.2]{MR2861610}, \cite[Section~2.4]{r:lgoa}. This construction is
encoded via the introduction of a ``disk function", which we now review in the
context of cyclic isotropy  (see \cite[Section~3.3]{r:lgoa} for the general
case of finite abelian isotropy groups). \\

Let $\cZ$ be a three-dimensional CY toric orbifold,  $p$ a fixed point such
that a neighborhood is isomorphic to $[\C^3/\Z_{n+1}]$, with representation
weights $(m_1, m_2,m_3)$ and CY torus weights  $(w_1,w_2,w_3)$. Fix a Lagrangian boundary condition $L$
which we assume to be on the first coordinate axis in this local chart.
Define ${n_{e}}=
(n+1)/\gcd(m_1,n+1)$ to be the size of the effective part of the action along
the first coordinate axis. 
  There exist a map from an orbi-disk mapping to the first coordinate axis with winding $d$ and twisting\footnote{Here twisting refers to the image of the center of the disk in the evaluation map to the inertia orbifold.} $k$ if the compatibility condition 
\eq{
\frac{d}{{n_{e}}}-\frac{km_1}{n+1}\in \bbZ
\label{compat}
}
is satisfied. In this case the positively oriented disk function  is
\begin{equation}
D_k^+(d;\vec{w})=
\left( \frac{ {n_{e}}w_1}{d} \right)^{\text{age}(k)-1}\frac{{n_{e}}}{d(n+1)\left\lfloor \frac{d}{{n_{e}}} \right\rfloor !}\frac{\Gamma\left( \frac{dw_{2}}{{n_{e}}w_1}+\left\langle \frac{k m_{3}}{n+1} \right\rangle + \frac{d}{{n_{e}}} \right)}{\Gamma\left( \frac{dw_{2}}{{n_{e}}w_1}-\left\langle \frac{k m_{2}}{n+1} \right\rangle +1 \right)}.
\end{equation}
The negatively oriented disk function is obtained by switching the indices $2$
and $3$. By renaming the coordinate axes this definition applies to the
general boundary condition. \\

In \cite{r:lgoa} the disk function is used to construct the GW orbifold
topological vertex, a building block for open and closed GW invariants of
$\cZ$. The \textit{scalar disk potential} is  expressed in terms of the
disk and of the $J$ function of $\cZ$. 
The fixed point basis for  the equivariant  Chen Ruan  cohomology of $\cZ$  has $n+1$ classes supported at the fixed point $p$, corresponding to all irreduible representations of $\bbZ_{n+1}$. For $k=1, \ldots n$, denote by $\mathbf{1_{p,k}}$ the fundamental class of the twisted sector corresponding to the characted $k$; we denote $\mathbf{1_{p,n+1}}$ the untwisted class of the fixed point $p$. Raising indices using the orbifold
Poincar\'e pairing, and extending the disk function to be a cohomology valued
function
\begin{equation}
\DD^+(d;\vec{w})=\sum_{k=1}^{n+1}D_k^+(d;\vec{w}) \mathbf{1_p^k},
\end{equation}
the (genus zero) \textit{scalar disk potential} is obtained by contraction with the $J$ function:
\ea{
F_{L}^{\rm disk}(\tau,y,\vec{w}) &\triangleq  \sum_d\frac{y^d}{d!}\sum_n
\frac{1}{n!}\langle \tau, \ldots, \tau \rangle_{0,n}^{L,d} \nn \\
&= \sum_d \frac{y^d}{d!}\left( \DD^+(d;\vec{w}), J^\cZ\left(\tau,\frac{n_ew_1}{d}\right)\right)_{\cZ},
\label{sdp}
}
where we denoted by $\langle \tau, \ldots, \tau \rangle_{0,n}^{L,d}$ the disk
invariants with boundary condition $L$, winding $d$
and $n$ general cohomological insertions.
\begin{rmk}
We may consider the disk potential relative to multiple Lagrangian boundary conditions. In that case, we define the disk function  by adding the disk functions for each Lagrangian, and we introduce a winding variable for each boundary condition. 
Furthermore, it is not conceptually difficult (but book-keeping intensive) to express the general  open potential in terms of appropriate contractions of arbitrary copies of these disk functions with the full descendent Gromov--Witten potential of $\cZ$.
\end{rmk}

\subsubsection{The disk function, revisited}

We reinterpret the disk function as a symmetric tensor of Givental
space. First we homogenize Iritani's Gamma class \cref{eq:gammaT} and make it
of degree zero:
\eq{
\overline{\Gamma}_\cZ(z) \triangleq z^{-\frac{1}{2} \deg}\overline{\Gamma}_\cZ \triangleq \sum \overline{\Gamma}_\cZ^k \mathbf{1_{p,k}},
\label{irihomo}
}
where the second equality defines  $\overline{\Gamma}_\cZ^k$ as the  $\mathbf{1_{p,k}}$-coefficient of  $\overline{\Gamma}_\cZ(z)$. With notation as in \cref{sec:ogw}, we define
\begin{equation}
\overline{\DD}_\cZ^+(z;\vec{w})(\mathbf{1_{p,k}})\triangleq \frac{\pi }{w_1(n+1)\sin\left(\pi\left(\left\langle \frac{km_{3}}{n+1} \right\rangle -\frac{w_{3}}{z}\right)\right)} \frac{1}{ \overline{\Gamma}_\cZ^k }\mathbf{1_{p}^k}.
\label{eq:dr}
\end{equation}

The dual basis of inertia components   diagonalizes  the tensor $\overline{\DD}_\cZ^+$.

\begin{lem}
The $k$-th coefficient of $\overline{\DD}_\cZ^+$ coincides with $ D^+_k (d;\vec{w})$ when $z= n_ew_1/d$ and the winding/twisting compatibility condition is met:
\eq{
\delta_{1,\exp\left({2\pi i \left(\frac{d}{n_e}-\frac{km_1}{n+1}\right)}\right)}
\left(\overline{\DD}_\cZ^+\left(\frac{n_ew_1}{d};\vec{w}\right)(\mathbf{1_{p,k}}),\mathbf{1_{p,k}} \right)_\cZ = D^+_k (d;\vec{w})
}
\end{lem}
\begin{proof}
This formula follows from the explicit expression of $\overline{\Gamma}_\cZ$ in the localization/inertia basis, manipulated via the identity $\Gamma(\star)\Gamma(1-\star)= \frac{\pi}{\sin(\pi\star)}$. The Calabi--Yau condition $w_1+w_2+w_3=0$ is also used. The $\delta$ factor encodes the degree/twisting condition.
\end{proof}

\subsubsection{Open crepant resolutions}

Let  $\X \rightarrow X \leftarrow Y$ be a diagram of toric Calabi--Yau
threefolds for which the Coates--Iritani--Tseng/Ruan version of the closed crepant
resolution conjecture holds.
Choose a
Lagrangian boundary condition $L_\X$ in $\X$ and denote by $L_Y$ the transform of
such condition in $Y$; notice that in general this can consist of several Lagrangian
boundary conditions. 
\begin{prop}
\label{prop:o}
There exists a $\bbC((z^{-1}))$-linear transformation $\bbO: \HH_\X \to \HH_Y$ of Givental spaces such that
\eq{
\overline{\DD}_Y^+ \circ \U = \bbO \circ \overline{\DD}_\X^+.
}
\end{prop}

This proposition is trivial, as $\bbO$ can be constructed as $\overline{\DD}_Y^+ \circ \U\circ ( \overline{\DD}_\X^+)^{-1}$ where $( \overline{\DD}_\X^+)^{-1}$ denotes the inverse of $\overline{\DD}_\X^+$ after restricting to the basis of eigenvectors with nontrivial eigenvalues and $\bbO$ is defined to be $0$ away from these vectors. However we observe that interesting open crepant resolution statements follow from this simple fact, and that $\bbO$ is a  simpler object than $\U$. For a good reason: the disk function almost completely cancels the transcendental part in Iritani's central charge. We make this precise in the following  observation.

\begin{lem}
\label{lem:ga}
Referring to  \cref{eq:gammaT,eq:dr} for the relevant definitions, we have
\ea{
 \Theta_\cZ(\mathbf{1_{p,k}}) &\triangleq 
\frac{w_1(n+1)}{z^{\frac{3}{2}}\pi } \overline{\DD}_\cZ^+ (\mathbf{1_{p,k}}\otimes z^{-\mu}\overline{\Gamma}_\cZ) \mathbf{1_{p}^k} \nn \\ &=  \frac{1}{\sin\left(\pi\left(\left\langle \frac{km_{3}}{n+1} \right\rangle -\frac{w_{3}}{z}\right)\right)} \mathbf{1_{p}^k} 
\label{Theta}
}
\end{lem}

Combining \cref{lem:ga} with Iritani's \cref{eq:iritanisymp}, we obtain the following prediction.
\begin{conj}
\label{conj:iri}
Choose bases for the equivariant CR cohomologies of $\X$ and $Y$.  Consider a set $\mathfrak{W}$\footnote{ In \cite{ed:window} such a set is called a \textit{grade restriction window}. In the hypotheses and notation of \cref{prop:o}, note that $\cX$ and $Y$ must be related by variation of GIT, and therefore they are quotients of a common space $Z=\bbC^{l_Y+3}$; the grade restriction window may be chosen from the coordinate axes of $Z$, thought of as topologically trivial, but not equivariantly trivial, line bundles. See also \cite{bfk:window} and \cite{hhp:window}.}  of equivariant bundles on $Z$ that descend bijectively to bases for  $K(\cX)\otimes \bbC$ and $K(Y)\otimes \bbC$. For $\bullet  = \X,Y$, let $\mathrm{CH}_\bullet$ denote the  matrix of   Chern characters in the chosen bases.

Denote $$\overline{\rm CH}_\bullet= {\left(\frac{2\pi \rm{i}}{z}\right)}^{\frac{1}{2}\deg }inv^\ast \mathrm{CH}_\bullet.$$ With $\Theta_\bullet$ be as in equation \cref{Theta}, we have:
\eq{
\bbO= \Theta_Y\circ\overline{\rm CH}_Y \circ \overline{\rm CH}^{-1}_\X \circ {\Theta_\X}^{-1}.
}
\end{conj}

We verify \cref{conj:iri} for the resolution of $A_n$ singularities in \cref{sec:An,sec:j}. We also note that while we are formulating the statement in the case of cyclic isotropy to keep notation lighter, it is not hard to write an analogous prediction in a completely general toric setting.\\

\subsection{The OCRC}
\label{ssec:ocrc}
Having modified our perspective on the disk functions, we also update our take
on open disk invariants to remember the twisting of the map at the origin of
the disk. In correlator notation, denote  $\langle \tau, \ldots, \tau \rangle_{0,n}^{L,d,k}$ the disk
invariants with Lagrangian boundary condition $L$, winding $d$, twisting $k$ 
and $n$ cohomology insertions. 
We then define the \textit{cohomological disk potential} as a cohomology valued function, which is expressed as a composition of the $J$ function with the  disk function \cref{eq:dr}:
\ea{
\cF_{L}^{\rm disk}(\tau,y,\vec{w}) & \triangleq \sum_d\frac{y^d}{d!}\sum_n
\frac{1}{n!}\langle \tau, \ldots, \tau \rangle_{0,n}^{L,d,k}\mathbf{1_{p}^{k}}, \nn \\
\label{cohdp}
&= \sum_d \delta_{1,\exp\left({2\pi \ri \left(\frac{d}{n_e}-\frac{km_1}{n+1}\right)}\right)}\frac{y^d}{d!}\overline{\DD}_\cZ^+\circ J^\cZ\left(\tau,\frac{n_ew_1}{d},\vec{w}\right).
}
We define a section of Givental space that  contains equivalent information to the disk potential:
\begin{equation}
\mathbb{F}_{L}^{\rm disk}(t,z,\vec{w}) \triangleq \overline{\DD}_\cZ^+ \circ J^\cZ\left(\tau,z;\vec{w}\right).
\end{equation}

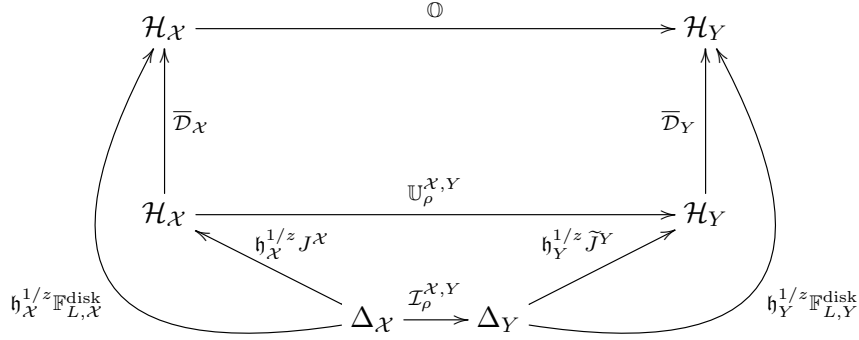
\begin{figure}[tb]
$$
\xymatrix{
{\HH_\X}   \ar[rrrrr]^\bbO &  &   &   & & \HH_Y \\
\\
\HH_\X  \ar[rrrrr]^{\U} \ar[uu]_{\overline{\DD}_\X}& &   &   & & \HH_Y \ar[uu]^{\overline{\DD}_Y}\\
& & 
\Delta_\cX  \ar[ull]_{\mathfrak{h}^{1/z}_{\cX} J^\X} \ar@/^6pc/[uuull]^{\mathfrak{h}^{1/z}_{\cX} \mathbb{F}_{L,\X}^{\rm disk}}  
\ar[r]^{\cI^{\cX,Y}_\rho} &
\Delta_Y   \ar[urr]^-{\mathfrak{h}^{1/z}_{Y}\widetilde J^Y}
\ar@/_6pc/[uuurr]_{{\mathfrak{h}^{1/z}_{Y} \mathbb{F}_{L,Y}^{\rm disk}}} & 
}
$$
\caption{Open potential comparison diagram. 
In the Hard Lefschetz case this same diagram holds with the $\mathfrak{h}$ factors omitted, and $\Delta_\bullet$ identified with the full cohomologies of either target.}
\label{diag}
\end{figure}

We call $\mathbb{F}_{L}^{\rm disk}(t,z,\vec{w}) $ the \textit{winding
  neutral disk potential}. For any pair of integers $k$ and $d$ satisfying
\cref{compat}, the twisting $k$ and winding $d$ part of the disk potential
is obtained by substituting $z=\frac{n_ew_1}{d}$. A general ``disk crepant
resolution statement" that follows from the closed CRC is a comparison of
winding-neutral potentials, as illustrated in \cref{diag}. 

\begin{prop}
\label{prop:wncrc}
Let $\X \rightarrow X \leftarrow Y$ be a diagram for which the
Coates--Iritani--Tseng/Ruan form of the closed crepant resolution conjecture
holds and identify quantum parameters in $\Delta_\cX$ and
  $\Delta_Y$ via $\cI_\rho^{\cX,Y}$ as in \cref{eq:iddelta}. Then: 
\eq{
\mathfrak{h}^{1/z}_Y {\mathbb{F}_{L,Y}^{\rm disk}}\big|_{\Delta_Y}= \mathfrak{h}^{1/z}_\cX \bbO\circ \mathbb{F}_{L,\X}^{\rm disk}\big|_{\Delta_\X}.
\label{eq:dcrc}
}
Assume further that $\X$ satisfies the Hard Lefschetz condition and
identify cohomologies via the affine linear change of variables $\widehat{\cI}_\rho^{\cX,Y}$. Then:
\eq{
{\mathbb{F}_{L,Y}^{\rm disk}}=\bbO\circ \mathbb{F}_{L,\X}^{\rm disk}.
}
\end{prop}
Here we also understand that the winding-neutral disk potential of $Y$ is
analytically continued appropriately (we suppressed the tilde to avoid
excessive proliferation of superscripts). 

\begin{rmk}
At the level of cohomological disk
potentials, the normalization factors $\mathfrak{h}_\cX$ and $\mathfrak{h}_Y$ enter as a
redefinition of the winding number variable $y$ in \cref{cohdp} depending on small quantum cohomology parameters; this is the manifestation
of the the so-called {\it open mirror map} in the physics literature on open
string mirror symmetry \cite{Lerche:2001cw,
  Aganagic:2001nx, Bouchard:2007ys, Brini:2011ij}. 
\end{rmk}

\begin{rmk}
The statement of the proposition in principle hinges on the  possibility to identify
quantum parameters as in \crefrange{eq:iddelta}{eq:changevargen}. Restricting to the coordinate hyperplanes of the fundamental class insertions, the existence of the non-equivariant limits of
  $\U$ and the $J$-functions  is guaranteed by the fact that we employ
  a torus action acting trivially on the canonical bundle of $\cX$ and $Y$; see e.g.~\cite{moop}.\\
\end{rmk}

\subsection{The Hard Lefschetz OCRC}

\label{sec:hg}
In the Hard Lefschetz case the comparison of disk potentials naturally extends to the full open potential. We define the \textit{genus g, $\ell$-holes winding neutral potential}, a function from $H(\cZ)$ to the $\ell$-th tensor power of Givental space $\HH_\cZ^{\otimes \ell}= H(\cZ)((z_1^{-1}))\otimes\ldots \otimes H(\cZ)((z_\ell^{-1}))$:
\eq{
\mathbb{F}^{g,\ell}_{\cZ, L}(\tau,z_1,\ldots,z_\ell,\vec{w}) \triangleq \overline{\DD}_\cZ^{+\otimes \ell} \circ J^\cZ_{g,\ell}\left(\tau,z_1,\ldots, z_\ell;\vec{w}\right),
\label{mholepot}
}
where $J^\cZ_{g,\ell}$ encodes  genus $g$, $\ell$-point descendent invariants:
\eq{
J^\cZ_{g,\ell}(\tau,\bfz; \vec{w})\triangleq \bra\bra
\frac{\phi_{\alpha_1}}{z_1-\psi_1}, \dots,  \frac{\phi_{\alpha_\ell}}{z_\ell-\psi_\ell}
\ket\ket_{g,\ell}\phi^{\alpha_1}\otimes\cdots\otimes\phi^{\alpha_\ell}.
\label{eq:JZn}
}
In \cref{eq:JZn}, we denoted $\bfz=(z_1,\dots,z_\ell)$ and a sum over repeated
Greek indices is intended.  
Just as in the disk case, one can now define a winding neutral open potential by summing over all genera $g$ and integers $\ell$ and a cohomological open potential by introducing winding variables and summing over appropriate specializations of the $z$ variables. For a pair of spaces $\X$ and $Y$ in a Hard Lefschetz CRC diagram the respective potentials can be compared as in \cref{ssec:dcrc} - this all follows  from the comparison of the $l$-hole winding neutral potential, which we now spell out with care.

\begin{thm}\label{thm:manyholes}
Let $\X \rightarrow X \leftarrow Y$ be a Hard Lefschetz diagram for which the
higher genus closed Crepant Resolution Conjecture holds (\cref{conj:scr}). With all notation as in \cref{prop:wncrc}, and  $\bbO^{\otimes \ell}= \bbO(z_1)\otimes\ldots\otimes \bbO(z_\ell)$, we have:
\eq{
{\mathbb{F}_{L',Y}^{g,\ell}}= \bbO^{\otimes \ell}\circ \mathbb{F}_{L,\X}^{g,\ell}
}
\end{thm}

\begin{proof}

The generating function $J^\cZ_{g,\ell}$ is obtained from the genus-$g$
descendent potential by first $\ell$ applications of the total differential
$d$, and then restricting to the small phase space variables
$\tau=\{\tau^\alpha\}=\{t^{\alpha,0}\}$. Under the natural identification of
the $i$-th copy of $T^\ast \HH^+_\cZ\cong \HH_\cZ$ with the auxiliary variable $z_i$, we have
\eq{
dt^{\alpha,k}= \frac{\phi^\alpha}{z^{k+1}}.
\label{eq:diff}
}
\cref{conj:scr} give us the equality of  the Gromov--Witten partition functions \eqref{eqpartscr} after a change of variable given by the linear identification:
\eq{
\pi_+ \circ \U \circ i: \HH_\cX^+ \to \HH_Y^+.
\label{givid}
} 
If we decompose the symplectomorphism $\U$ as a series in $1/z$ of linear maps, $\U \triangleq \sum_{n\geq 0} \frac{1}{z^n}\bbU^n$, then differentiating the change of variable given by  \eqref{givid} gives:
\eq{
dt_Y^{\alpha, k}= \sum_{n=0}^k\bbU^n_{\alpha,\mu} dt_\cX^{\mu,k+n},
\label{eqdcv}
}
where we denoted by $\bbU^n_{\alpha,\mu}$ the $(\alpha,\mu)$ entry of the matrix representing $\bbU^n$ after having chosen bases for the cohomologies of $\cX$ and $Y$. Combining \eqref{eq:diff} and \eqref{eqdcv}:
\ea{
\sum_{k=0}^\infty \frac{\phi^\alpha}{z^{k+1}} = \sum_{k=0}^\infty
dt_Y^{\alpha,k}= \sum_{k=0}^\infty  \sum_{n=0}^k\bbU^n_{\alpha,\mu}
dt_\cX^{\mu,k+n}=  \sum_{k=0}^\infty  \sum_{n=0}^k\bbU^n_{\alpha,\mu}
\frac{\phi^\mu}{z^{k+n+1}}
=\U\left( \sum_{k=0}^\infty \frac{\phi^\mu}{z^{k+1}}\right).
\label{toh}
}
Now we differentiate  \eqref{eqpartscr} $\ell$ times, and restrict to primary variables that are identified via \eqref{givid}: such identification reduces to  $\widehat{\cI}_\rho^{\cX,Y}$. Using \eqref{toh} we obtain
\eq{
J^Y_{g,\ell}=  {\U}^{\otimes \ell}\circ  J^\cX_{g,\ell}.
}
The statement of the theorem follows by composing by $ \overline{\DD}_Y^{+\otimes \ell} $ and then using the commutativity of the diagram in \cref{diag}.


\end{proof}

\section{OCRC for $A_n$ resolutions}
\label{sec:An}

\subsection{Equivariant  $\U$ and Integral Structures}
\label{sec:u}
For the  pairs $(\X,Y)=\left([\C^{3}/\Z_{n+1}],A_n\right)$,
\cref{prop:wncrc,thm:manyholes} imply a Bryan--Graber type CRC statement
comparing the open GW potentials.  Notice that since $\X$ is a Hard Lefschetz
orbifold we do not have to deal with the trivializing scalar factors $\mathfrak{h}_\bullet$.
In \cref{sec:ia,sec:ea} we study the two essentially distinct types of Lagrangian boundary conditions.\\

The reader may find a detailed review of the toric geometry describing our
targets in \cref{sec:GIT}, which is summarized by \cref{fig:web}. A generic Calabi--Yau torus action is taken on
$\cX$ and $Y$, with weights as in \cref{fig:web}. We  denote by $\phi_j$,
$j=1\dots n$, the duals of the torus invariant lines $L_j \in
H_2(Y)$ and by $P_i$, $i=1,\dots, n+1$ the equivariant cohomology classes
concentrated on the torus fixed points of the resolution.  On the
orbifold, we label by $\fc_k$, $k=1,\dots, n+1$ the fundamental classes of the
components of the inertia stack $\cI\cX$ twisted by $\re^{2\pi \ri k/(n+1)}$. A generic point 
$t\in H(Y)$ is written as $t^{n+1} \mathbf{1}_Y + \sum_j t^j\phi_j$;
similarly, we  write
$x=\sum_{k=1}^{n+1} x^k \fc_k$ for $x\in H(\cX)$.\\

Let now $\cY_\epsilon$ be the ball of radius $\epsilon$ around the large radius
limit point of $Y$ with respect to the Euclidean metric $(\rd s)^2 = \sum_i
(\rd \re^{t_i})^2$ in exponentiated
flat coordinates $\re^{t_j}$. We define a path $\rho$ in $\cY_1$,
%
\eq{
\bary{cccc}
\rho  : & [0,1] & \rightarrow & \cY_1, \\ 
& s & \to & (\rho(s))_j = s \omega^{-j}.
\eary
\label{eq:rho}
}
as the straight line in the coordinates $\re^{t_j}$ 
connecting the large radius point
$\mathrm{LR} \triangleq \{\re^{t_j}=0\}$ of $Y$ to
the one of $\cX$, which we  denote as $\mathrm{OP}\triangleq
\{\re^{t_j}=\omega^{-j}\}$. 

\begin{thm}
\label{thm:sympl}
Let $\widetilde{J}^Y(z)$ denote the analytic continuation of $J^Y$ along the path $\rho$ to the
point $\rho(1)$ composed with the identification \cref{eq:changevar} of
quantum parameters. Then the linear transformation
\eq{
\label{eq:Uiri}
\U\fc_{k}=\sum_i P_i\frac{1}{(n+1)}\frac{\overline{\Gamma}_Y^i}{\overline{\Gamma}_\X^k}\left( \sum_{j=0}^{i-1}\omega^{-jk}\re^{2\pi\ri\frac{j\alpha_1}{z}}+\sum_{j=i}^{n}\omega^{-jk}\re^{2\pi\ri\frac{(n+1-j)\alpha_2}{z}} \right)
}
is an isomorphism of Givental spaces such that
\eq{
\widetilde{J}^Y= \U\circ {J}^\X.
}
\end{thm}

We prove \cref{thm:sympl} in the fully equivariant setting in
\cref{sec:compsymp} as an application of the one-dimensional mirror
construction of \cref{sec:mirror}.

This result is compatible with Iritani's 
\cref{eq:iritanisymp}. We now describe the canonical identification
$\bbU_{K,\rho}^{\cX,Y}$. Denote by $\so(\lambda_k)$ the geometrically trivial line bundle on $\bbC^{n+3}$ where the torus $(\bbC^*)^n$ acts via the $k$th factor with weight $-1$ and the torus $T$ acts trivially.  We define our grade restriction window $\mathfrak{W}\subset K(\bbC^{n+3})$ to be the subgroup generated by the $\so(\lambda_k)$.  Using the description of the local coordinates in \cref{sec:GIT}, we compute that the quotient \cref{orbgit} identifies $\so(\lambda_k)$ with $\so_{-k}$ (with trivial $T$-action) and the quotient \cref{resgit} identifies $\so(\lambda_k)$ with $\so(\phi_k)$ (with canonical linearization \cref{canwts}).  Therefore, we define $\bbU_{K,\rho}^{\cX,Y}$ by identifying
\begin{align}
\so_Y&\leftrightarrow\so_\cX\\
\so(\phi_k)&\leftrightarrow\so_{-k}
\end{align}
where the $T$-linearizations are trivial on the orbifold and canonical on the
resolution. \\

On the orbifold, all of the bundles $\so_j$ are linearized trivially, so the higher Chern classes vanish.  The orbifold Chern characters are:
\eq{
(2\pi\ri)^{\deg/2}{\rm inv}^*\ch(\so_j)=\sum_{k=1}^{n+1}\omega^{-jk}\fc_k.
}
The $\Gamma$ class is
\begin{align}
z^{-\frac{1}{2}\deg}\overline\Gamma_\X=\Gamma&\left(1+\frac{\alpha_1+\alpha_2}{z}\right)\\
&\hspace{-1cm}\cdot\left[ \sum_{k=1}^n \Gamma\left(1-\frac{k}{n+1}-\frac{\alpha_1}{z}\right)\Gamma\left(\frac{k}{n+1}-\frac{\alpha_2}{z}\right)\frac{\fc_k}{z}
+
\Gamma\left(1-\frac{\alpha_1}{z}\right)\Gamma\left(1-\frac{\alpha_2}{z}\right)\fc_{n+1}\right]
\end{align}

On the resolution, the Chern roots at each $P_i$ are the weights of the action on the fiber above that point:
\eq{
(2\pi\ri)^{\deg/2}\ch(\so(\phi_j))=\sum_{i=1}^{j}\re^{2\pi\ri(n+1-j)\alpha_2}P_i+\sum_{i=j+1}^{n+1}\re^{2\pi\ri
  j\alpha_1}P_i
}
and
\eq{
(2\pi\ri)^{\deg/2}\ch(\so)=\sum_{i=1}^{n+1}P_i.
}
The $\Gamma$ class is
\eq{
z^{-\frac{1}{2}\deg}\overline\Gamma_Y=\Gamma\left(1+\frac{\alpha_1+\alpha_2}{z}\right)\left[ \sum_{i=1}^{n+1} \Gamma\left(1+\frac{w_i^+}{z} \right)\Gamma\left( 1+\frac{w_i^-}{z} \right) P_i\right]
}
With this information one can compute the symplectomorphism as in  \cref{eq:iritanisymp} and obtain the formula in \cref{thm:sympl}. \\

We now derive explicit disk potential CRC statements for the two distinct types of Lagrangian boundary conditions.

\subsection{$L$ intersects the ineffective axis}
\label{sec:ia}
We impose a Lagrangian boundary condition on the gerby leg of the orbifold (the third coordinate axis - $m_3=0$); correspondingly there are $n+1$ boundary conditions $L'$ on the resolution, intersecting the horizontal torus fixed lines  in \cref{fig:web}.
\begin{thm}
\label{thm:dcrccoh}
Consider the cohomological disk potentials $\cF_{L',Y}^{\rm disk}(t,y_{P^1},
\dots, y_{P^{n+1}},\vec{w})$ and $\cF_{L,\X}^{\rm disk}(t,y,\vec{w})$.
Choosing the dual bases $\mathbf{1^k}$ and $P^i$ (where $k$ and $i$ both range from $1$ to $n+1$), define a linear transformation 
$
\bbO_\Z:H(\X) \to H(Y)
$
 by the matrix
\eq{
{\bbO_{\Z}}_k^i= \left\{\begin{array}{ll}- \omega^{\left(\frac{1}{2}-i\right)k} & k\not=n+1\\-1 & k=n+1.\end{array}\right.
}
After the identification of variables from \cref{thm:crc}, and the specialization of winding parameters
\eq{
y_{P^i}= e^{\pi\ri\left[\frac{w_i^-+(2i-1)\alpha_1}{\alpha_1+\alpha_2}\right]}y
\label{wvcov}
}
we have
\eq{
\cF_{L',Y}^{\rm disk}(t,y,\vec{w}) = \bbO_\bbZ \circ \cF_{L,\X}^{\rm disk}(t,y,\vec{w}) .
}
\end{thm}

\begin{proof}

From \cref{eq:dr}, we have

\eq{
\overline{\DD}_{L,\X}^+(z;\vec{w})(\mathbf{1_{k}})=\sum_{k=1}^{n+1}\frac{\pi\fc^k}{(n+1)(\alpha_1+\alpha_2)\sin\left(\pi\left(\left\langle\frac{k}{n+1}\right\rangle+\frac{\alpha_2}{z}\right)\right)\overline{\Gamma}_\X^k}
}
and
\eq{
\overline{\DD}_{L',Y}^+(z;\vec{w})(P_i)=\sum_{i=1}^{n+1}\frac{\pi P^i}{(\alpha_1+\alpha_2)\sin\left(\pi\left(-\frac{w_i^-}{z}\right)\right)\overline{\Gamma}_Y^i}}
The transformation $\bbO$ is now obtained as $\overline{\DD}^+_Y\circ\U\circ \left(\DD^+_\X\right)^{-1}$:
\eq{
\bbO(\mathbf{1^k})= \sum_{i=1}^{n+1} \left[\frac{\sin\left(\pi\left(\bra\frac{k}{n+1}\ket+\frac{\alpha_2}{z}\right)\right)}{\sin\left(\pi\left(-\frac{w_i^-}{z}\right)\right)}\left(\sum_{j=0}^{i-1}\omega^{-jk}\re^{2\pi\ri\frac{j\alpha_1}{z}}+\sum_{j=i}^{n}\omega^{-jk}\re^{2\pi\ri\frac{(n+1-j)\alpha_2}{z}} \right)\right] P^i.
}
We now specialize  $z=\frac{\alpha_1+\alpha_2}{d}$, for $d\in \bbZ$. The $i,k$
coefficient for $k\neq n+1$, after some gymnastics with telescoping sums, becomes:
\ea{
\bbO_{k}^i 
&= (-1) e^{d\pi\ri\left[n-i+2 +(2i-n-2)\frac{\alpha_1}{\alpha_1+\alpha_2}\right]} \omega^{\left(\frac{1}{2}-i\right)k}.
}
For $k=n+1$
\eq{
\bbO_{n+1}^i=(-1) e^{d\pi\ri\left[n-i+2 +(2i-n-2)\frac{\alpha_1}{\alpha_1+\alpha_2}\right]}.
}
%
It is now immediate to see that we can incorporate the part of the transformation that depends multiplicatively on $d$ into a specialization of the winding variables, and that the remaining linear map is precisely $\bbO_\bbZ$. 
\end{proof}
From this formulation  of the disk  CRC one can deduce  a statement about scalar disk potentials which essentially says that the scalar potential of the resolution compares with the untwisted disk potential on the orbifold. 
\begin{cor}
With all notation as in \cref{thm:dcrccoh}:
\eq{
\left( \cF_{L',Y}^{\rm disk}(t,y,\vec{w}),\sum_{i=1}^{n+1}P_i\right)_Y = -\frac{1}{n+1}\left(\cF_{L,\X}^{\rm disk}(t,y,\vec{w}),\mathbf{1_{n+1}}\right)_\X
}
\label{cor:sc}
\end{cor}
\begin{proof}
This statement amounts to the fact that  the coefficients of all but the last
column of the matrix $\mathbb{O}_\bbZ$ add to zero. 
\end{proof}
\subsection{$L$ intersects the effective axis}
\label{sec:ea}

We impose  our boundary condition $L$ on the first coordinate axis, which is  an effective quotient of $\bbC$ with representation weight  $m_1=-1$ and torus weight $-\alpha_1$. We can obtain results for the boundary condition on the second axis by switching $\alpha_1$ with $\alpha_2$, $m_1$ with $m_2$ and $+$ with $-$ in the orientation of the disks. In this case there is only one corresponding boundary condition $L'$ on the resolution, which intersects the (diagonal) non compact leg incident to $P_{n+1}$ in \cref{fig:web}.

\begin{thm}
\label{thm:dcrccoheff}
Consider the cohomological disk potentials $\cF_{L,\X}^{\rm disk}(t,y,\vec{w}) $ and $\cF_{L',Y}^{\rm disk}(t, y_{P_{n+1}},\vec{w})$.
Choosing the bases $\mathbf{1^k}$ and $P^i$ (where $k$ and $i$ both range from $1$ to $n+1$), define $\bbO_\Z(\mathbf{1^k})=P^{n+1}$ for every $k$.
After the identification of variables from \cref{thm:crc}, and the identification of winding parameters $y=y_{P^{n+1}}$
we have
\eq{
\cF_{L',Y}^{\rm disk}(t,y,\vec{w}) = \bbO_\bbZ \circ \cF_{L,\X}^{\rm
  disk}(t,y,\vec{w}). \\
}
\end{thm}
We obtain as an immediate corollary a comparison among scalar potentials.
\begin{cor}
Setting $y=y_{P_{n+1}}$, we have
\eq{
F_{L',Y}^{\rm disk}(t,y,\vec{w}) = F_{L,\X}^{\rm disk}(t,y,\vec{w}) .
}
\label{cor:esc}
\end{cor}

\begin{proof}
The orbifold disk endomorphism is:
\eq{
\overline{\DD}^+_\X(z;\vec{w})(\mathbf{1_k})=\frac{\pi}{-\alpha_1(n+1)\sin\left(\pi\left(-\frac{\alpha_1+\alpha_2}{z}\right)\right)\overline{\Gamma}_\X^k}\fc^k
}

The resolution disk endomorphism  is
\eq{
\overline{\DD}^+_\X(z;\vec{w})(P_i) =\frac{\pi}{-(n+1)\alpha_1\sin\left(\pi\left(-\frac{\alpha_1+\alpha_2}{z}\right)\right)\overline{\Gamma}_Y^{n+1}} \delta_{i,n+1}P^{n+1}
}

We can now compute $\bbO$:
\eq{
\bbO(\mathbf{1^k})= \frac{1}{n+1}\left(\sum_{j=0}^{n}\omega^{-jk}\re^{2\pi\ri\frac{j\alpha_1}{z}}\right)P^{n+i}.
}

Specializing $z=\frac{-(n+1)\alpha_1}{d}$ for any positive integer $d$, we obtain:

\eq{
\bbO_{k}^{n+1}=\frac{1}{n+1}\sum_{j=0}^{n}\omega^{-jk}\re^{2\pi\ri j\frac{-d}{n+1}}=\delta_{k,-d\mod n+1},
}
which implies the statement of the theorem.

\end{proof}

\section{One-dimensional mirror symmetry}
\label{sec:j}

It is known that the quantum $D$-modules associated to the equivariant Gromov--Witten theory of
the $A_n$-singularity $\cX$ and its resolution $Y$ admit a Landau--Ginzburg
description in terms of $n$-dimensional oscillating integrals
\cite{MR1408320, MR1328251, MR2700280, MR2529944}. We provide here an alternative description
via one-dimensional twisted periods of a genus zero double Hurwitz space
$\cF_{\lambda, \phi}$. \\
\subsection{Weak Frobenius structures on double Hurwitz spaces}
\begin{defn}
Let  $\vec{\mathsf{x}}\in \bbZ^{n+3}$ be a vector of integers  adding to $0$. The
{\rm genus zero double Hurwitz space} $\HH_{\vec{\mathsf{x}}} \triangleq \cM_0( \bbP^1; \mathbf{x})$ parameterizes isomorphism classes of covers $\lambda$ of the projective line by a smooth genus $0$ curve $C$, with marked ramification profile over $0$ and $\infty$ specified by $\vec{\mathsf{x}}$. This means that  the principal divisor of  $\lambda$ is of the form
$$
(\lambda)= \sum \mathsf{x}_i q_i.
$$
We denote by $\pi$ and $\lambda$ the universal family and universal map, and by $\Sigma_i$ the sections marking the $i$-th point in $(\lambda)$:
\eq{
\xymatrix{\bbP^1  \ar[d]  \ar@{^{(}->}[r]& \mathcal{U}\ar[d]^\pi  \ar[r]^{\lambda}  &  \bbP^1 \\
                     [\lambda]  \ar@{^{(}->}[r]^{pt.}   \ar@/^1pc/[u]^{P_i}&                                             \HH_{\vec{\mathsf{x}}}  \ar@/^1pc/[u]^{\Sigma_i}& 
}
}
\end{defn}

A genus zero double Hurwitz space is naturally isomorphic to $\bbC^\star
\times M_{0,n+3}$, and
is therefore an open set in affine space $\bbA^{n+1}$. The genus zero case is the only case we  consider in this paper and it may seem overly sophisticated to use the language of moduli spaces to
work on such a simple object: we choose to do so in order to connect to the work of Dubrovin \cite{Dubrovin:1992eu, Dubrovin:1994hc} and Romano \cite{2012arXiv1210.2312R} (after Saito \cite{MR723468}; see also \cite{Krichever:1992qe}), who  studied
existence and construction of Frobenius structures on double Hurwitz spaces for
arbitrary genus. \\

Write $\mathrm{supp}(\lambda)=\{q_i \in C\}_i$ for the set of points
$(\lambda)$ is supported on, and let 
$\phi\in \Omega^1_{C}(\log (\lambda))$ be a meromorphic one form having simple poles at $\mathrm{supp}(\lambda)$ with
constant residues; we call $(\lambda, \phi)$ respectively the {\it
  superpotential} and the {\it primitive differential} of $\HH_{\vec{\mathsf{x}}}$.
Borrowing the terminology from \cite{2012arXiv1210.2312R, phdthesis-romano},
we say that an analytic Frobenius manifold structure $(\cF, \circ, \eta)$ on
a complex manifold $\cF$ is
weak if
\ben
\item the $\circ$-multiplication gives a commutative and associative
unital $\cO$-algebra structure
on the space of holomorphic vector fields on $\cF$;
\item the metric $\eta$ provides a flat
  pairing which is Frobenius w.r.t. to $\circ$;
\item the algebra structure
admits a  potential, meaning that the 3-tensor
\eq{
R(X,Y,Z) \triangleq \eta(X,Y \circ Z)
}
satisfies the integrability condition
\eq{
(\nabla^{(\eta)} R)_{[\a \b] \g\d}=0.
}
\een
In particular, this encompasses non-quasihomogeneous solutions of
WDVV, and solutions without a flat identity element.\\

By choosing the last three sections to be the constant sections $0, 1,
\infty$, we realize  $\HH_{\vec{\mathsf{x}}}\simeq \bbC^\star \times M_{0,n+3}$ as an open subset of $\bbA^{n+1}$ and trivialize the universal family. 
In homogeneous coordinates $[u_0:\dots:u_n]$ for $\bbP^n$,
\ea{
\HH_{\vec{\mathsf{x}}} &= \bbC^\star \times \bbP^n\setminus \mathrm{discr} \HH_{\vec{\mathsf{x}}}, \\
\mathrm{discr} \HH_{\vec{\mathsf{x}}} &\triangleq  \mathrm{Proj} \frac{\bbC[u_0, \dots, u_n]}{{\bra
   \prod_{i=0}^n u_i \prod_{j<k} (u_j-u_k)\ket}}.
 \label{eq:discr}
}

\begin{thm}[\cite{MR2070050, 2012arXiv1210.2312R}]
For vector fields $X$, $Y$, $Z \in \fX(\HH_{\vec{\mathsf{x}}})$, define the
non-degenerate symmetric pairing $g$ and quantum product $\star$ as
\ea{
\label{eq:gmetr}
g(X,Y) &\triangleq  \sum_{P\in\mathrm{supp}(\lambda)}\Res_P\frac{X(\log\lambda)
  Y(\log\lambda)}{\rd_\pi \log\lambda}\phi^2, \\
g(X,Y \star Z) &\triangleq  \sum_{P\in\mathrm{supp}(\lambda)}\Res_P\frac{X(\log\lambda)
  Y(\log\lambda) Z(\log\lambda)}{\rd_\pi \log\lambda}\phi^2,
\label{eq:star}
}
where $\rd_\pi$ denotes the relative differential with respect to the
universal family (i.e. the differential in the fiber direction). Then the triple $\cF_{\lambda,\phi}=\l(\HH_{\vec{\mathsf{x}}}, \star, g\r)$ endows
$\HH_{\vec{\mathsf{x}}}$ with a holomorphic weak Frobenius manifold structure. The
embedding $\HH_{\vec{\mathsf{x}}} \hookrightarrow \bbC^\star \times \bbP^n$ induces uniquely a meromorphic
weak Frobenius structure on $\bbP^1\times\bbP^n$.
\end{thm}
%
\crefrange{eq:gmetr}{eq:star} are the
Dijkgraaf--Verlinde--Verlinde formulae \cite{Dijkgraaf:1990dj} for a
topological Landau--Ginzburg model on a sphere with $\log\lambda(q)$ as its
superpotential. The case in which $\lambda(q)$ itself is used as the
superpotential gives rise to a {\it different} Frobenius manifold structure,
which is the case originally analyzed by Dubrovin in his study of Frobenius
structures on Hurwitz spaces \cite[Lecture 5]{Dubrovin:1994hc}; the
situation at hand is its dual
in the sense of \cite{MR2070050}, where $g$ plays the role of the
intersection form and $\star$ the dual product, whose poles coincide with the discriminant ideal in the
Zariski closure \cref{eq:discr} of $\HH_{\vec{\mathsf{x}}}$. \\
%
\begin{rmk}
\label{rmk:loglambda}
Since $\lambda$ is a genus zero covering map, in an affine chart parametrized by $q\in\bbC$ its logarithm takes the
form 
\eq{
\log\lambda = \sum_{i}\mathsf{x}_i \log(q-q_i) + \mathsf{y},
\label{eq:logl}
}
where $\mathsf{x}_i, \mathsf{y}\in \bbZ$. In fact, the existence of the weak Frobenius structure 
\crefrange{eq:gmetr}{eq:star} carries through unscathed \cite{phdthesis-romano} to the case where $\mathrm{d}_\pi\log\lambda$
is a meromorphic differential on $C$ upon identifying
$\mathrm{supp}(\lambda)=\{q_i\}$; this in particular encompasses the case
where $\mathsf{x}_i, \mathsf{y}\in \bbC$ in \cref{eq:logl}. The locations $q_i$ of the punctures
provide a special type of local coordinates on $\HH_{\vec{\mathsf{x}}}$: by the general theory of
double Hurwitz spaces \cite{2012arXiv1210.2312R}, for suitable choices of
$\phi$ their logarithms are flat
coordinates for the pairing $g$ in \cref{eq:gmetr}.
\end{rmk}

\subsubsection[Twisted homology and the QDE]{Twisted homology and the quantum
  differential equation}

Let $C_\lambda \triangleq C\setminus \mathrm{supp}(\lambda)$ and denote by
$\pi :\tilde C_\lambda \to C_\lambda$ its universal covering space. Fix $z\in
\bbC$ and pick the canonical principal branch for
$\lambda^{1/z} = \exp(z^{-1}\log\lambda)$ in \cref{eq:logl}, defined as
\eq{
\lambda^{1/z}(q) = \prod_{i=1} |q-q_i|^{\xi_i}\re^{\ri \xi_i \arg_i(q)} 
}
where $\xi_i:=\mathsf{x}_i/z$ and $\arg_i(q)\in [0,2\pi)$ is the angle formed by $q-q_i$ with the real axis. 
Then we have a
monodromy representation $\rho_\lambda : \pi_1(C_\lambda)\to L_\lambda \simeq \bbC$ 
on the complex line $L_\lambda$ parametrized by $\lambda^{1/z}$, a simple loop
$l_{q_i}$ around $q_i$ resulting in
multiplication by $\mathfrak{q}_i:=\rho_\lambda(l_{q_i}) = \re^{2\pi \ri \sum_{j=1}^i\xi_j}$. 
We  denote by $H_\bullet(C_\lambda, L_\lambda)$ (resp.  $H^\bullet(C_\lambda, L_\lambda)$) the homology (resp.~cohomology)
groups of $C_\lambda$ twisted by the set of local coefficients determined by
$\mathfrak{q}_i$. Integration over $\gamma \in H_1(C_\lambda,
L_\lambda)$ of $\lambda^{1/z}\phi \in H^1(C_\lambda, L_\lambda)$  defines the
{\it twisted period mapping}
\eq{
\bary{ccccc}
\pi_{\lambda,\phi} & : & H_1(C_\lambda,L_\lambda) & \to & \cO(\HH_{\vec{\mathsf{x}}}) \\
& & \gamma & \to & \int_\gamma \lambda^{1/z}\phi.
\eary
\label{eq:periodmap}
}
Let now $\nabla^{(g,z)}:\fX(\HH_{\vec{\mathsf{x}}}) \to \fX(\HH_{\vec{\mathsf{x}}}) \otimes
\Omega^1(\HH_{\vec{\mathsf{x}}})$ be the Dubrovin connection associated to $\cF_{\lambda,
  \phi}$
\eq{
\label{eq:defconn}
\nabla^{(g,z)}_X(Y,z) \triangleq  \nabla^{(g)}_X Y+z^{-1} X \star Y
}
and write $\mathrm{Sol}_{\lambda,
  \phi}$ for its $\bbC(\xi_1, \dots, \xi_n)$-vector space of parallel sections
\eq{
\mathrm{Sol}_{\lambda,
  \phi} = \{s \in \fX(\HH_{\vec{\mathsf{x}}}), \nabla^{(g,z)}s=0 \}.
\label{eq:QDELG}
}
The following statement \cite{MR2070050} is a {\it verbatim} application of
the arguments of \cite{Dubrovin:1998fe} for the ordinary Hurwitz case.
\begin{prop}
\label{thm:tp}
The solution space of the quantum differential equations of
$\cF_{\lambda,\phi}$ is generated by gradients of the twisted periods
\cref{eq:periodmap}
\eq{
\mathrm{Sol}_{\lambda, \phi} = \mathrm{span}_{\bbC(a_1, \dots, a_n)}
\{\mathrm{grad}_{g} \pi_{\lambda,\phi}(\gamma) \}_{\gamma \in H_1(C_\lambda,
L_\lambda)}.
}

\end{prop}

In other words, twisted periods are a flat coordinate frame for the Dubrovin
connection on $T\cF_{\lambda,\phi}$.

\subsection{A one-dimensional Landau--Ginzburg mirror}
\label{sec:mirror}

We now fix the ramification profile
\eq{
\vec{\mathsf{x}}=((n+1) \alpha_1, -\a_1-\a_2,(n+1) \a_2, \underbrace{-\a_1-\a_2, \dots, -\a_1-\a_2}_{n}).
}
Define $\cM_A \triangleq \HH_{\vec{\mathsf{x}}}$. We pick global
coordinates on it as follows: we write $\kappa_0$ for an
(exponentiated-linear) coordinate in the first factor of $\cM_A \simeq \bbC^\star
\times M_{0,n+3}$, and we pick $\kappa_i=u_i/u_0$, $i=1, \dots, n$ as a
set of global coordinates on $M_{0,n+3}$. As before, we  write $q$ to denote an affine coordinate on the fibers of the universal
family. We give $\cM_A$ the structure of a one parameter family of
double Hurwitz spaces as follows:
\eq{
\lambda(\kappa_0, \ldots \kappa_n, q) =  \prod_{j=0}^n \kappa_j^{\a_1}
\frac{q^{(n+1)\a_1}}{\left(1-q\right)^{\a_1+\a_2}\prod _{k=1}^{n}
\left(1-q\kappa_k\right)^{\a_1+\a_2}} , 
\label{eq:superpot}
}
\eq{
\phi(q) = \frac{1}{\a_1+\a_2}\frac{\rd q}{q}.
\label{eq:primeform}
}
The Frobenius structure  on $\cM_A$ determined by \cref{eq:gmetr,eq:star,eq:superpot,eq:primeform} is denoted by $\cF_{\lambda,
  \phi}$. By \cref{rmk:loglambda}, and since both the metric and the
associative product in \cref{eq:gmetr,eq:star,eq:superpot,eq:primeform} depend
rationally on $(\a_1,\a_2)$, we will in the following consider them as
complex parameters. \\

We claim that there exist
neighborhoods $V_\cX, V_Y \subset \cM_A$ such that $\cF_{\lambda,
  \phi}$ is locally isomorphic to the quantum cohomologies of
$\cX=[\bbC^{3}/\bbZ_{n+1}]$ and its canonical resolution $Y$. The ultimate
justification of this statement resides in the relation of the
Gromov--Witten theory of $\cX$ and $Y$ with integrable systems, and notably
the two-dimensional Toda hierarchy; the details of this connection 
can be found in \cite{Brini:2014mha}. For the purposes of this paper, it is enough
to offer a direct proof of the existence of said local isomorphisms. \\
\begin{thm}
\label{thm:mirror}
\ben
\item With notation as at the beginning of \cref{sec:u}, let 
\ea{
\label{eq:kappa0Y}
\kappa_0 &= \re^{(t_{n+1}+\delta_Y)/\a_1}, \\
\label{eq:kappaY}
\kappa_j &= \prod_{i=j}^n \re^{t_i}, \quad 1\leq j\leq n.
}
where $\delta_Y$ is an arbitrary constant. Then, in a neighborhood $V_Y$ of
$\mathrm{LR} = \{ \re^{t_i}=0\}$, 
\eq{
 \cF_{\lambda,\phi} \simeq QH(Y).
}
\item Let
\ea{
\label{eq:kappa0X}
\kappa_0 &= \re^{(x_{n+1}+\delta_\cX)/\a_1}, \\
\kappa_j &= \exp\l[-\frac{2\ri}{n+1}\l(\pi j+ \sum_{k=1}^n
  \re^{-\frac{\ri \pi  k (j-1)}{n+1}} \sin \left(\frac{\pi  j
    k}{n+1}\right)x_k\r)\r], \quad 1\leq k\leq n. 
\label{eq:kappakX}
}
where $\delta_\cX$ is an arbitrary constant. Then, in a neighborhood $V_\cX$
of $\mathrm{OP} = \{x_i=0\}$,
\eq{
 \cF_{\lambda,\phi} \simeq QH(\cX).
}
\een
\end{thm}
\begin{proof} The proof is a direct computation from the
  Landau--Ginzburg formulae \crefrange{eq:gmetr}{eq:star}. 
\ben
\item 
Consider the three-point
correlator $R(\kappa_i \de_i, \kappa_j \de_j, \kappa_k \de_k)$, where
$\de_k \triangleq \frac{\de}{\de \kappa_k}$, and define
\ea{
R^{(l)}_{i,j,k} &\triangleq \Res_{q=\kappa_l^{-1}} \frac{\kappa_i \frac{\de \ln\lambda}{\de \kappa_i}
  \kappa_j\frac{\de \ln\lambda}{\de \kappa_j} \kappa_k \frac{\de
    \ln\lambda}{\de \kappa_k} }{(\a_1+\a_2)^2 q \frac{\de \ln\lambda}{\de q}}\frac{\rd q}{q}.
}
Inspection shows that $R^{(l)}_{ijk}=0$ unless $l=i=j$, $l=i=k$ or
$l=j=k$. Assume without loss of generality $l=j=i$, and suppose that $i,k>0$. We compute
\eq{
\label{eq:lgqu1}
R^{(i)}_{i,i,k}
= 
\frac{\kappa_i}{\kappa_k-\kappa_i}+\frac{\a_2}{\a_1+\a_2}, 
}
\eq{
\label{eq:lgqu2}
R^{(i)}_{i,i,i}
= \frac{(n-1) \a_1+\a_2}{\a_1+\a_2}+
\sum_{l\neq i}^{n+1}\frac{\kappa_l}{\kappa_i-\kappa_l}, \quad
R^{(i)}_{0,i,i} 
= -\frac{1}{\a_1+\a_2}.
}
Moreover, for all $i$, $j$ and $k$ we have
\ea{
R^{(0)}_{i,j,k} &\triangleq \Res_{q=0} \frac{\kappa_i \frac{\de \ln\lambda}{\de \kappa_i}
  \kappa_j\frac{\de \ln\lambda}{\de \kappa_j} \kappa_k \frac{\de
    \ln\lambda}{\de \kappa_k} }{(\a_1+\a_2)^2 q \frac{\de \ln\lambda}{\de q}}\frac{\rd
  q}{q} = \frac{\a_1^{2-\delta_{i,n+1}-\delta_{j,n+1}-\delta_{k,n+1}}}{(n+1)(\a_1+\a_2)^2} \\
R^{(\infty)}_{i,j,k} &\triangleq \Res_{q=\infty} \frac{\kappa_i \frac{\de \ln\lambda}{\de \kappa_i}
  \kappa_j\frac{\de \ln\lambda}{\de \kappa_j} \kappa_k \frac{\de
    \ln\lambda}{\de \kappa_k} }{(\a_1+\a_2)^2 q \frac{\de \ln\lambda}{\de q}}\frac{\rd
  q}{q} = -\frac{(-\a_2)^{2-\delta_{i,n+1}-\delta_{j,n+1}-\delta_{k,n+1}}}{(n+1)(\a_1+\a_2)^2}. 
\label{eq:resinf}
}
It is immediate to see that \crefrange{eq:lgqu1}{eq:lgqu2} under the
identification \eqref{eq:kappaY} imply that the quantum part of the
three-point correlator $R(\de_{t_{i_1}}\de_{t_{i_2}}\de_{t_{i_3}})$
coincides with that of $\bra\bra \phi_{i_1}, \phi_{i_2}, \phi_{i_3} \ket\ket^Y_{0}$ in
\cref{eq:yukY}. A tedious, but straightforward computation shows that
\crefrange{eq:lgqu1}{eq:resinf} yield the expressions \crefrange{ccorr1}{ccorr4}
for the classical triple intersection numbers of $Y$. \\
\item Is obtained by composing the computation above with the Coates--Corti--Iritani--Tseng isomorphism of
  quantum cohomologies (\cref{thm:crc}).
\een

\end{proof}
\begin{figure}
\includegraphics{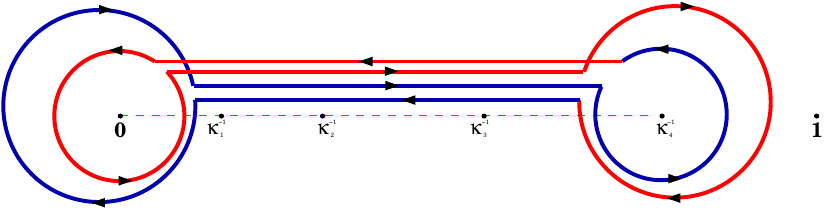}
\caption{The double loop contour $\gamma_4$ for $n=4$.}
\label{fig:pochcont}
\end{figure}

\begin{rmk}
\label{rmk:string}
The freedom of shift by $\delta_\cX$ and $\delta_Y$ respectively along
$H^0(\cX)$ and $H^0(Y)$ in \cref{eq:kappa0Y,eq:kappa0X} is a consequence of the restriction of the String Axiom to the small phase
space. We  set $\delta_\cX=\delta_Y=0$ throughout this section, but it will be useful to reinstate the shifts in the computations of \cref{sec:compsymp}.
\end{rmk}

\subsection{The global quantum $D$-module}
\label{sec:globdpic}

\cref{thm:mirror,thm:tp} together imply the existence
of a global quantum $D$-module $(\cM_A, F, \nabla,
H(,)_g)$ interpolating between $\mathrm{QDM}(\cX)$ and $\mathrm{QDM}(Y)$. Let
$F \triangleq T\cF_{\lambda, \phi}$ be
endowed with the family of connections $\nabla^{(g,z)}$ as in
\cref{eq:defconn} and for $\nabla^{(g,z)}$-flat sections $s_1$, $s_2$ define
\eq{
H(s_1, s_2)_g = g(s_1(\kappa, -z),s_2(\kappa,z)).
}
With notation as in \cref{sec:u}, let $V_\cX$ and $V_Y$ be neighborhoods of OP and
LR respectively. Then \cref{thm:mirror,thm:tp} imply that 
\ea{
(\cF_{\lambda,\phi}, T\cF_{\lambda,\phi}, \nabla^{(g,z)},H(,)_{g})|_{V_\cX} &\simeq 
\mathrm{QDM}(\cX), \\
(\cF_{\lambda,\phi}, T\cF_{\lambda,\phi}, \nabla^{(g,z)},H(,)_{g})|_{V_Y} &\simeq 
\mathrm{QDM}(Y).
}
In particular, choosing a basis of integral twisted 1-cycles
yields a global flat frame for the
quantum differential equations of $\cX$ and $Y$ upon analytic continuation in
the $\kappa$-variables, and
\eq{
\mathrm{Sol}_{\lambda,\phi}|_{V_\cX} = \cS_\cX, \quad 
\mathrm{Sol}_{\lambda,\phi}|_{V_Y} = \cS_Y. 
\label{eq:globqdmrestr}
}

Representatives of one such integral basis can be constructed as follows. For generic monodromy weights, the monodromy representation
$\rho_\lambda(l_{q_i})$ factors through a faithful representation $\tilde
\rho: H_1(C_\lambda, \bbZ) \to V_\lambda$. Then in this case the twisted homology
coincides with the integral homology of the Riemannian covering \cite{MR1930577} of $C_\lambda$
\eq{
H^\bullet(C_\lambda, L_\lambda) \simeq  H^\bullet(\tilde
C_\lambda/[\pi_1(C_\lambda),\pi_1(C_\lambda)]), \bbZ).
}
In particular, compact
  loops in the kernel of the abelianization morphism $h_*:\pi_1(C_\lambda)\to
  H_1(C_\lambda,\bbZ)$ may have non-trivial lifts to $H_1(C_\lambda,
  L_\lambda)$. One such basis is given explicitly
  \cite{MR1424469, MR1930577} by the Pochhammer double
  loop contours $\{\gamma_i\}_{i=1}^{n+1}$: these are compact loops encircling the points
  $q=0$ and $q=\kappa_i^{-1}$, $i=1, \dots, n+1$  as in \cref{fig:pochcont}
 (that is $\gamma_i=[l_{0},l_{\kappa_i^{-1}}]$, where the $l_q$'s are simple
  oriented loops around each of the punctures). Then the twisted periods
\ea{
\Pi_i & \triangleq  \frac{z a \pi_{\lambda,\phi}(\gamma_i)}{(1-\re^{2\pi\ri a})(1-\re^{-2\pi\ri b})}
\label{eq:eulerint}
}
where we defined
\eq{
\bary{ccccccc}
a & \triangleq & \frac{(n+1) \a_1}{z}, & \quad & \mathfrak{q}_a & \triangleq
& \re^{2\pi\ri a}, \\
b & \triangleq & \frac{\a_1+\a_2}{z}, & \quad & \mathfrak{q}_b & \triangleq &
\re^{2\pi\ri b},
\eary
}
are a $\bbC(\mathfrak{q}_a,\mathfrak{q}_b)$-basis of $\mathrm{Sol}_{\lambda,\phi}$. 

\begin{rmk}
\label{rmk:line}
We have a natural isomorphism with the homology of the complex line relative to the punctures
\eq{
\bary{cccc}
\mathfrak{P}: & H_1(C_\lambda, L_\lambda) &
\stackrel{\sim}{\rightarrow} & H_1(\bbP^1, \mathrm{supp}(\lambda)), \\
& \gamma_i & \to & (1-\mathfrak{q}_a)(1-\mathfrak{q}_b) [0, \kappa_i^{-1}]
\eary
}
obtained by associating to any Pochhammer contour the path in $C_\lambda$ that it
encircles. The choice of coefficient reflects the existence  \cite{MR1424469}, when $\Re(a)>0$,
$\Re(b)<1$, of an Euler-type integral reprentation: namely, a factorization
of the period mapping
\eq{
\begin{xy}
(0,20)*+{H_1(C_\lambda,L_\lambda)}="a"; 
(0,0)*+{H_1(\bbP^1,(\lambda))}="b"; (40,0)*+{\bbC(\mathfrak{q_a},\mathfrak{q_b})}="c";
{\ar_{\mathfrak{P}} "a";"b"};{\ar^{\pi_{\lambda,\phi}} "a";"c"};
{\ar^{ \int \lambda^{1/z} \phi} "b";"c"};
\end{xy}
\label{eq:line}
}
which reduces 
\cref{eq:eulerint} to convergent line integrals of
$z a \lambda^{1/z}\phi$ over the interval $\mathfrak{P}(\gamma_i)$.\\
\end{rmk}

By the above remark, the period integrals of \cref{eq:eulerint} are 
a multi-variable generalization of the classical Euler representation for the
Gauss hypergeometric function. Explicitly, they take the form \cite{MR0422713}
\ea{
\Pi_i(\kappa,z) &=
\frac{\Gamma(a)\Gamma(1-b)}{\Gamma(1+a-b)}
\kappa_i^{-a}\prod_{j=0}^n \kappa_j^{\a_1/z}
 \nn \\
&\times 
\Phi^{(n)}\l(a,b,1+a-b;
\frac{\kappa_1}{\kappa_i}, \dots, \frac{\kappa_{i-1}}{\kappa_i},
\frac{1}{\kappa_i}, \frac{\kappa_{i+1}}{\kappa_i}, \dots,
\frac{\kappa_n}{\kappa_i}\r), \quad 1\leq i\leq n, \label{eq:pilaur1} \\
\Pi_{n+1}(\kappa,z) &= \frac{\Gamma(a)\Gamma(1-b)}{\Gamma(1+a-b)} \left( \prod_{j=0}^n \kappa_j^{\a_1/z}\right)
\Phi^{(n)}(a,b,1+a-b;
\kappa_1, \dots, \kappa_n),
\label{eq:pilaur2}
}
where we defined
\eq{
\label{eq:Phi}
\Phi^{(m)}(a, b, c, w_1, \dots, w_m) \triangleq  F_D^{(m)}(a; b, \dots, b; c; w_1, \dots, w_{m}),
}
and $F_D^{(m)}(a; b_1, \dots, b_M; c; w_1, \dots, w_m)$ in \cref{eq:Phi} is the
Lauricella function of type $D$ \cite{lauric}:
\eq{
F_D^{(m)}(a; b_1, \dots, b_m; c; w_1, \dots, w_m) \triangleq \sum_{i_1, \dots, i_m}
\frac{(a)_{\sum_j i_j}}{(c)_{\sum_j i_j}}\prod_{j=1}^m \frac{(b_j)_{i_j} w_j^{i_j}}{i_j!}.
\label{eq:FD}
}
In \cref{eq:FD}, we used the Pochhammer symbol $(x)_m$ to denote the ratio $(x)_m=
\Gamma(x+m)/\Gamma(x)$. \\


\subsubsection{Example: $n = 1$ and the Gauss system.} In this case
$\cF_{\lambda,\phi}$ has dimension 2. The equations
for the flat coordinates 
$\tilde t(\kappa_0 , \kappa_1 , z)$ of the Dubrovin connection,
\cref{eq:defconn}, reduce to the classical Gauss 
hypergeometric system for a function $f(\kappa_1 , z)$ such that
\eq{
\tilde{t}(\kappa_0 , \kappa_1 , z) = (\kappa_0 \kappa_1 )^{−a/2} f (\kappa_1 , z),
\label{eq:fhyp}
}
where
\eq{
\kappa_1 (\theta  + a) (\theta + b) f = \theta (\theta  + a − b) f, \quad 
\theta= \kappa_1 \de_{\kappa_1}.
\label{eq:eq2F1}
}
When $n = 1$, we have from \crefrange{eq:pilaur1}{eq:pilaur2} that
\ea{
\Pi_1 (\kappa, z) = &
\frac{\Gamma(a)\Gamma(1-b)}{\Gamma(1+a-b)}
\kappa_0^{a/2} \kappa_1^{-a/2}
{}_2F_1 \l(a,b,1+a-b,\frac{1}{\kappa_1}\r),\\
\Pi_2 (\kappa, z) = &
\frac{\Gamma(a)\Gamma(1-b)}{\Gamma(1 + a - b)}
(\kappa_0\kappa_1)^{a/2} 
{}_2F_1 \l(a, b, 1 + a -b,\kappa_1\r).
}
These are immediately seen to satisfy \crefrange{eq:fhyp}{eq:eq2F1}.

\begin{rmk}
Equivariant mirror symmetry for toric
Deligne--Mumford stacks implies that flat sections of $\mathrm{QDM}(\cX)$ and
$\mathrm{QDM}(Y)$ take the form
of generalized hypergeometric functions in so-called 
$B$-model variables; see \cite[Appendix A]{MR2510741} for the case under
study here. Less expected, however, is the fact that they are
hypergeometric functions in {\it exponentiated flat variables} for the Poincar\'{e} pairing,
that is, in $A$-model variables. This is a consequence of the particular form
(\cref{eq:yukY,eq:lgqu1,eq:lgqu2})
of the quantum product: its rational dependence\footnote{From the vantage
point of mirror symmetry, the rational dependence of the $A$-model three-point
correlators on the quantum parameters is an epiphenomenon of the Hard Lefschetz
condition, which ensures that the inverse mirror map is a rational
function of exponentiated $A$-model variables.} on the variables $\kappa$ gives
the quantum differential equation
\cref{eq:QDE} the form of a generalized hypergeometric system in
exponentiated flat coordinates. The explicit equivalence between twisted
periods and solutions of the Picard--Fuchs equations of $\cX$ and $Y$, which is a consequence of
\cref{thm:mirror} here and Proposition~A.3 in \cite{MR2510741}, should
follow by comparing\footnote{Equivalence between the two types of
  hypergeometric functions can be derived from the quadratic
transformations for the Gauss function for $n=1$, and from a generalized
Bayley identity for $n=2$; the higher rank case appears to be non-trivial. 
} the respective Barnes integral representations
\cite{MR0422713, MR2271990}. A significant advantage of the Hurwitz-space picture is that sections of the quantum
$D$-modules have 1-dimensional integral representations, as opposed to
the $n$-fold Mellin--Barnes integrals of \cite{MR2271990}; 
this drastically reduces the complexity of computing the analytic continuation
from the large radius to the orbifold chamber, as we now  show.
\end{rmk}

We are almost ready to compute the analytic
continuation map $\bbU_\rho^{\cX, Y} : \HH_\cX \to \HH_Y$ that identifies the
corresponding flat frames and Lagrangian cones upon analytic continuation
along the path $\rho$ in \cref{eq:rho}. The main missing technical tool is
provided by the following
\begin{lem}
In $\bbC^m$ with coordinates $(w_1, \dots, w_m)$, $m\geq 1$, let $\chi_i$, for
every $i=1, \dots, m$, be any path in $\bbC^m \setminus  \{w_k \neq w_l,0,1\}$, up to homotopy, that connects the origin
with the point at
infinity $W^\infty_i$,
\eq{
W_i^\infty\triangleq(\overbrace{0,\dots, 0}^{i-1~\rm times}, \overbrace{\infty,\dots,
  \infty}^{m-i+1~\rm times}),
}
and has zero winding number along the hyperlanes $w_k=w_l$ ($k \neq l$) and $w_k=0,1$. Denote
$\tilde{\Phi}^{(m)}_i(a, b, c; w_1, \dots, w_m)$ the analytic continuation
of $\Phi^{(m)}(a, b, c; w_1, \dots, w_m)$ in \cref{eq:Phi} along $\chi_i$ to
the neighborhood 
\ea{
|w_l|<1, & \quad l< i, \nn \\
|w_{l}^{-1}|<1, & \quad l=i, \nn \\
|w_l^{-1}|<1, |w_l|<|w_k|, & \quad l>k\geq i
\label{eq:regw}
}
of $W_i^{\infty}$. Then we have that
%
\ea{
\tilde{\Phi}^{(m)}_i(a, b, c; w_1, \dots, w_m) & \sim 
\sum_{j=0}^{m-i}\frac{\Gamma(c)\Gamma(a-j b) \Gamma((j+1) b-a)}{\Gamma(a)
  \Gamma(b) \Gamma(c-a)} \nn \\ & \quad \times  \prod_{k=1}^j (-w_{m-k+1})^{-b}
(-w_{m-j})^{-a+j b}\l(1+\cO(w)\r) \nn \\ &+ \prod_{j=i}^{m} (-w_j)^{-b}
\frac{\Gamma(c)\Gamma(a-(m-i+1) b)}{\Gamma(a)\Gamma(c-(m-i+1) b)} \l(1+\cO(w)\r).
\label{eq:phias}
}
around $W_i^\infty$ in the region of \cref{eq:regw}.
\label{lem:ancont}
\end{lem}

\begin{proof}
The statement of the lemma follows from computing the analytic continuation along $\chi_i$ of the
Lauricella function 
$F_D^{(m)}(a,b_1, \dots, b_m, c, w_1, \dots, w_{i-1}, w_{i}^{-1}, \dots, w_m^{-1})$
from an open ball centered on $W^\infty_i$ to the origin 
$W^\infty_{m+1}=(0, \dots, 0)$ in the sector where $w_k \ll 1$ for $k< i$,
$w_{i}\ll 1$, $w_k/w_j \ll 1$ for $k>j\geq i$. One possible way to do this is to perform the
continuation in each individual variable $w_j$, $j>i$ appearing in \cref{eq:FD} through an iterated
use of Kummer's identity, \cref{eq:2F1conn}. This is done in
\cref{sec:anFD}, to which we refer the reader for the details of the
derivation; the final result is
\cref{eq:fdinf}, from which \cref{eq:phias} follows by \cref{eq:Phi}.
\end{proof}

\subsubsection{Proof of \cref{thm:sympl}}
\label{sec:compsymp}

We recall here the notation we used in \cref{sec:u,sec:globdpic}: we write
$P_i$ for the equivariant class concentrated
on the $i^{\rm th}$-fixed point of $Y$, $\fc_k$ for the fundamental class of
the $k^{\rm th}$-twisted sector of $\cX$, $i,k=1, \dots,n+1$, and 
$V_Y$ and $V_\cX$ for the neighborhoods of the large radius point (LR) and the
orbifold point (OP) respectively, such that the isomorphisms of \cref{eq:globqdmrestr} hold.
We also let $\rho$ be the path in $QH(Y)\simeq QH(\cX)$
connecting the large radius point LR
to the orbifold point OP as spelled out in \cref{eq:rho}
%
%
%
and we write 
\eq{
J^\cX(x,z) = \sum_{k=1}^{n+1} \tilde J^\cX_{k}(x,z) \fc_k, \qquad
J^Y(t,z) = \sum_{i=1}^{n+1}  J^Y_{i}(t,z) P_i
}
for the decomposition of the $J$-function of the
orbifold and the resolution in the bases above.  \\

The String Equation for $\cX$ and $Y$ and \crefrange{eq:fundsol}{eq:Jfun1} in \cref{sec:anFD} 
 together imply that the power series $\{J^\cX_{k}\}_{k=1}^{n+1}$,
and $\{J^Y_i\}_{i=1}^{n+1}$ give systems of flat coordinates of
$\nabla^{(g,z)}$ locally around OP and LR respectively. Also,
by \cref{thm:tp} and \cref{thm:mirror}, the twisted periods $\{\Pi_{j}\}_{j=1}^{n+1}$ yield a system of {\it global} flat
coordinates of $\nabla^{(g,z)}$; we here single out the principal branch of
\crefrange{eq:pilaur1}{eq:pilaur2} obtained by analytically continuing along
$\rho$. This means that, upon restriction to the neighborhood $V_\bullet$, the
gradients of $\{\Pi_j\}_j$ and $\{J^\bullet_i\}_i$ are a priori different linear
bases of the {\it same} vector space.
This entails the existence of invertible,
$\bbC[[a,b]]$-linear maps $\cA \in
\mathrm{Hom}(\mathrm{Sol}_{\lambda,\phi},\cS_Y), \cB\in \mathrm{Hom}(\cS_\cX, \mathrm{Sol}_{\lambda,\phi})$,
\eq{
\bary{rcccc}
\cA ~\mathrm{grad}_{\eta_Y} \pi_{\lambda,\phi}|_{\cS_Y} & : & H_1\l(C_\lambda,
L_\lambda\r) & \to & \mathrm{Sol}_{\lambda,\phi}|_{V_Y}\simeq \cS_Y, \\
\cB^{-1}~\mathrm{grad}_{\eta_\cX} \pi_{\lambda,\phi}|_{\cS_\cX} & : & H_1\l(C_\lambda, L_\lambda\r) & \to &
\mathrm{Sol}_{\lambda,\phi}|_{V_\cX}\simeq \cS_\cX,
\label{eq:AB}
\eary
}
such that
\eq{
\cA \{\Pi_j\}_{j=1}^{n+1} =  \{J^Y_{i}\}_{i=1}^{n+1}
}
and
\eq{
\cB \{ J^\cX_k\}_{k=1}^{n+1} =  \{\Pi_j\}_{j=1}^{n+1}.
}
In particular, the sought-for identification of $J$-functions factorizes as
\eq{
\bbU_\rho^{\cX, Y} = \cA \cB.
\label{eq:UBA}
}

To compute $\cA$, notice that the components $J_i^Y$ of $J^Y$ in the localized
basis $\{P_i\}_{i=1}^{n+1}$ of $H(Y)$ are eigenvectors  of the  monodromy around LR (see 
\cref{eq:Jloc}), generically
with distinct eigenvalues. $\cA$ can thus be computed by determining the monodromy decomposition
of the twisted periods,
\crefrange{eq:pilaur1}{eq:pilaur2}, from their asymptotic behavior around
LR. For each $1\leq j\leq n+1$, consider the principal branch of $\Pi_j$ given by the
integral expression of \cref{eq:eulerint}. The unit polydisk $|\re^{t_l}|<1$ centered at LR coincides with the region of \cref{eq:regw} for the arguments 
\eq{
w_k \triangleq \l\{\bary{cc} 
\kappa_k/\kappa_i & k\neq i \\
\kappa_i^{-1} & k=i \\
\eary\r.
}
of \cref{eq:pilaur1} by virtue of \cref{eq:kappaY}. This puts squarely the problem of analytic continuation
of $\{\Pi_j\}_j$ to LR within the setup of
\cref{lem:ancont}: by \cref{eq:pilaur1,eq:pilaur2}, for each $1\leq j\leq n$, the analytic continuation problem of $\Pi_j(\kappa,z)$ to LR along
$\rho$ in the $\kappa$-variables translates to the analytic continuation of a generalized
hypergeometric function $\Phi^{(n)}(a,b,1+a-b,w_1, \dots w_n)$ to $W_j^\infty$
along $\chi_j$ in the $w$
variables of the lemma. Applying the final result, \cref{eq:phias}, entails (compare with \cref{eq:ab1,eq:ab2,eq:Jloc})
\eq{
\Pi_i = \sum_{j=1}^{n+1}\cA^{-1}_{ij} J_j^Y,
}
where
\eq{
\cA_{ij} = \left\{\bary{cl} \re^{\pi\ri (n-i+1) b}
\frac{z \Gamma(1+a-(n-i+2) b)}{\Gamma(1-b) \Gamma(a-(n-i+1) b)} & i=j, \\
\re^{-i \pi  (a-b (2 n-2j+3))} \frac{z \sin (\pi  b)  \Gamma (1-a+b (n+1-i)) \Gamma (1+a-b (n-i+2))}{\pi  \Gamma (1-b)}
 & j>i,  \\
0 & j<i.
  \eary\right.
\label{eq:matrA}
}
\\

Consider now the situation at the orbifold point $\mathrm{OP} = \{\kappa_j=\omega^{-j}\}$. Since 
\ea{
J^\cX(0,z) &= z \mathbf{1}_0, \\
\frac{\de J^\cX}{\de x_k}(0,z) &= \mathbf{1}_{k},
}
to compute the operator $\cB$ in \cref{eq:AB} it suffices to evaluate the expansion of the Lauricella functions \crefrange{eq:pilaur1}{eq:pilaur2}
at OP to linear order in $x_k$, $k=0, \dots, n$.
A remarkable feature here, by \cref{eq:kappakX,eq:eulerint}, is that the Lauricella function
of \cref{eq:Phi} at these roots of unity reduces to Euler's Beta integral, a
statement whose easy verification we leave to the reader. Explicitly,
\ea{
\Pi_j(\kappa,z)\Big|_{x=0} &=
\omega^{(j-n/2) a}
\frac{\Gamma(a)\Gamma(1-b)}{\Gamma(1+a-b)} 
\Phi\l(a,b, 1+a-b;
\omega, \dots, \omega^n \r) \nn \\ &= 
\frac{\omega^{(j-n/2) a}}{n+1}B\l(\frac{a}{n+1},1-b\r)
 \label{eq:beta} \\ &=\frac{
\omega^{(j-n/2) a}}{n+1}
\frac{\Gamma(a/(n+1))\Gamma(1-b)}{\Gamma(1-b+a/(n+1))}, \qquad j=1, \dots, n+1.
\label{eq:dePix0}
}
Similarly, a short computation shows that
\ea{
\label{eq:dePik}
\kappa_k\frac{\de \Pi_j}{\de \kappa_k}(0,z) 
&= \frac{b\omega^{(j-n/2)a}}{n+1} \sum_{l=1}^{n}\omega^{(j-k)l}
B\l(\frac{a+l}{n+1},-b\r)
\\
\frac{\de \Pi_j}{\de x_k}(\kappa,z)\bigg|_{x=0} &=
 \frac{b\omega^{(j-n/2) a -jk +k/2}}{n+1} B\l(1+\frac{a+l}{n+1},-b\r) \label{eq:beta2}
\\ &=
-\frac{\omega^{(j-n/2) a -jk +k/2}}{n+1} 
\frac{\Gamma\l(\frac{a-k}{n+1}+1\r)\Gamma(1-b)}{\Gamma\l(\frac{a-k}{n+1}+1-b\r)}.
\label{eq:dePix}
}
In matrix form we have: 
\eq{
\Pi = \cB J^\cX = D_1 V D_2 J^\cX
\label{eq:matrB1}
}
where
\ea{
\label{eq:matrD1}
(D_1)_{jk} &= \omega^{(j-n/2) a}\delta_{jk}  \\
\label{eq:matrD2}
(D_2)_{jk} &=  \delta_{jk}\left\{\bary{cc} -\omega^{k/2} 
\frac{\Gamma\l(\frac{a-k}{n+1}+1\r)\Gamma(1-b)}{\Gamma\l(\frac{a-k}{n+1}+1-b\r)} &
\mathrm{for} \quad 1\leq k\leq n
\\ \frac{\Gamma(a/(n+1))\Gamma(1-b)}{z
  \Gamma(1-b+a/(n+1))} & \mathrm{for} \quad
k=n+1 \eary\right.  \\
V_{jk} &= \frac{\omega^{-jk}}{n+1}
\label{eq:matrV}
}
Piecing \cref{eq:matrA,eq:matrB1,eq:matrD1,eq:matrD2,eq:matrV} together
yields\footnote{This amounts to a rather tedious exercise in telescoping sums and additions of roots of unity. The computation can be made available upon request.} \cref{eq:Uiri}, up to a scalar factor of $\mathfrak{q}_{a}$. By
\cref{rmk:string}, this corresponds to our freedom of a String Equation
shift along either of $H^0(\cX)$ and $H^0(Y)$. Setting
$\delta_\cX-\delta_Y=2\pi\ri\alpha_1$ in \cref{eq:kappa0Y,eq:kappa0X} concludes the proof.

{\begin{flushright} $\square$ \end{flushright}}

\subsubsection{Monodromy and pure braids}
\label{sec:monodromy}

\begin{figure}[t]
\centering
\includegraphics{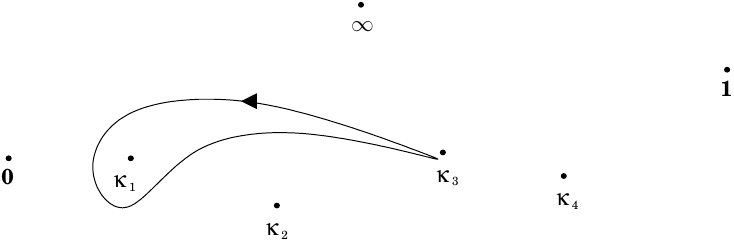}
\caption{The path $\sigma_{13}$ in $\pi_1(\cM_A)$ for $n=4$.}
\label{fig:mono}
\end{figure}

The expression \cref{eq:Uiri} for
the symplectomorphism $\bbU_\rho^{\cX, Y}$ was obtained for the analytic
continuation path $\rho$ of \cref{eq:rho}. Fixing a
reference point $m_0=(\widetilde{\kappa_1}, \dots, \widetilde{\kappa_n})
\in\cM_A$, for a general path $\rho \circ \sigma$ with $[\sigma] \in \pi_1(\cM_A, m_0)$
we get a composition 
\eq{
\bbU_{\rho\circ\sigma}^{\cX, Y} = \bbU_\rho^{\cX, Y} M_\sigma
}
where 
\eq{
M_\sigma : \pi_1(\cM_A,m_0) \to \mathrm{Aut}(\mathrm{Sol}_{\lambda,\phi})
\label{eq:msigma}
}
is the monodromy representation of the fundamental group of $\cM_A$ in the space
of solutions of the Lauricella system $F_D^{(n)}$. \\

By definition \cref{eq:discr}, $\cM_A$ is the configuration space of $n$
distinct points in $\bbP^1\setminus\{0,1,\infty\}$. Therefore, its fundamental group 
is the pure braid group in $n+2$ strands
\eq{
\pi_1(\cM_A) \simeq \mathrm{PB}_{n+2};
}
with the monodromy action \cref{eq:msigma} given by the reduced Gassner
representation \cite{MR849651, MR1816048} of $\mathrm{PB}_{n+2}$.
Writing $\widetilde{\kappa_i}=0,1,
\infty$ for $i=n+1$, $n+2$ and $n+3$ respectively, generators $P_{ij}$, $i=1, \dots, n+3$,
$j=1, \dots, n$  of $\mathrm{PB}_{n+2}$ are in bijection with paths
$\sigma_{ij}:[0,1]\to\cM_A$ given by lifts to $\cM_A$ of closed contours in the
$j^{\rm th}$ affine coordinate plane that start at
$\kappa_j=\widetilde{\kappa_j}$, turn counterclockwise around
$\widetilde{\kappa_i}$ (and around no other point) and then return to their
original position, as in \cref{fig:mono}.  \\

The image of the period map \cref{eq:periodmap}, by \cref{thm:tp},
is a lattice in $\mathrm{Sol}_{\lambda,\phi}$:
\eq{
\mathrm{Sol}_{\lambda,\phi} = \nabla^{(g)}\pi_{\lambda,\phi}\l(H_1\l(C_\lambda,
L_\lambda\r)\r) \otimes_{\bbZ(\mathfrak{q}_{a},\mathfrak{q}_{b})}\bbC(\mathfrak{q}_{a},\mathfrak{q}_{b}),
}
and by \cref{eq:matrA,eq:matrB1,eq:matrD1,eq:matrD2,eq:matrV} the induced morphism $H_1(C_\lambda,  L_\lambda) \simeq K(Y) $ is a lattice isomorphism. The
monodromy action on $\mathrm{Sol}_{\lambda, \phi}$, at the level of
equivariant $K$-groups, is given by lattice automorphisms $\pi_1(\cM_A) \to
\mathrm{Aut}_{\bbZ(\mathfrak{q}_{a},\mathfrak{q}_{b})} K(Y)$;
this can be verified explicitly from the form of the monodromy matrices in the
twisted period basis \cite{MR2962392}. For example, when $n=1$, the action on $K(Y)$ is given by the classical monodromy of the
Gauss system for $c=a-b+1$. With reference to \cref{fig:modspace1}, we
have in the standard basis $\{\cO_Y, \cO_Y(1)\}$ for $K(Y)$,
\ea{
M_{\rm LR1} &=
\l(
\begin{array}{cc}
 \re^{-\ri a \pi } \left(\re^{2 \ri a \pi }+\re^{2 \ri b \pi }\right) & \re^{2 \ri b \pi } \\
 -1 & 0 \\
\end{array}\r), \\ 
M_{\rm CP} &=
\left(
\begin{array}{cc}
 1 & -2 \ri \re^{-\ri (a-b) \pi } \sin (b \pi ) \\
 -2\ri\re^{-\ri (a+ b) \pi } \sin(b\pi) & 1-4 \re^{-2 \ri a \pi } \sin^2(b\pi) \\
\end{array}
\right), \\
M_{\rm LR2}
&=
\left(
\begin{array}{cc}
 2 \cos (a \pi ) & 1-2\ri \re^{\ri  \pi(b-2a) }\sin(b\pi) \\
 -1 & 2 \ri \re^{-\ri (a-b) \pi } \sin (b \pi ) \\
\end{array}
\right),
} 
for the large radius and the conifold monodromy of $\mathrm{QDM}(Y)$. It is
straightforward to check that they induce symplectic automorphisms of $\HH_Y$. \\

\begin{rmk}
In the non-equivariant setting, representations of the braid group $B_n$ have a natural place in
the derived context where they correspond to elements of
$\mathrm{Auteq}(D^b(Y))$  generated by spherical twists \cite{MR1831820}. From a
quantum $D$-module perspective, the
interpretation of flat sections as $B$-branes
identifies this braid group action with the monodromy
action. Recently, different flavors of braid group actions, including 
mixed and pure braids, have been shown to arise upon
taking deformations of the Seidel--Thomas setup \cite{donovan-segal}; the
$D$-module picture of \cref{eq:superpot,eq:eulerint}
indicates that the lift to the equivariant theory should naturally provide
another such extension, whose origin in the derived context would be
fascinating to trace in detail.
\end{rmk}

\begin{figure}
\includegraphics{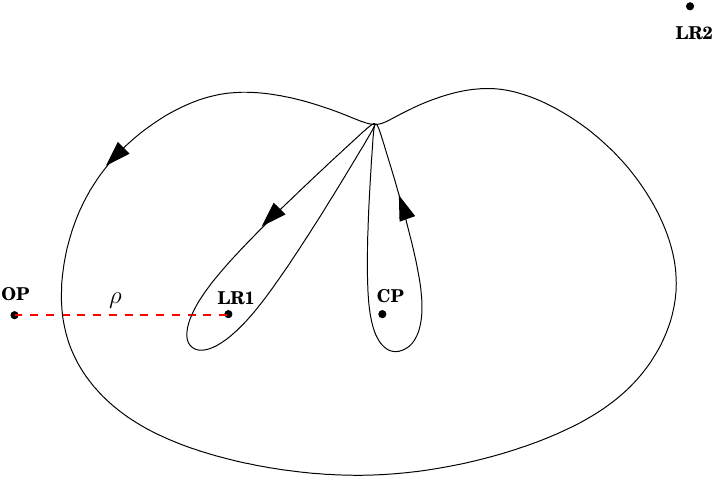}
\caption{The K\"ahler moduli space of the $A_1$ singularity in $A$-model
  coordinates. LR1 and LR2 indicate the large radius points $\kappa=0, \infty$
respectively, CP is the conifold point, and OP is the orbifold
point. Circuits around LR1-2 and CP generate the monodromy group of the global
quantum $D$-module. The dashed segment depicts the analytic continuation path $\rho$ of
\cref{thm:crc,thm:sympl}.}
\label{fig:modspace1}
\end{figure}

\section{Quantization}
\label{sec:quantum}

The goal of this section is to prove the following
\begin{thm}
\label{thm:cqcrc}
The Hard Lefschetz quantized CRC (\cref{conj:scr}) holds for the pairs $(\cX,Y)$, where $\cX= [\C^2/\Z_{n+1}]\times \C$ has the threefold $A_n$ singularity as a coarse space and  $Y$ is its crepant resolution.
\end{thm}

We outline the proof of Theorem \ref{thm:cqcrc}: Givental's quantization formula
\eqref{eq:qgiv} for $\cX$ and $Y$ and the Hard Lefschetz condition are used in \cref{lem:qcrcR} to show that 
\cref{thm:cqcrc} follows from appropriately comparing the canonical $R$-calibrations for $\cX$ and $Y$. 
The existence of the canonical $R$-calibrations is a consequence of Teleman's
reconstruction theorem \cite{telemangiv}. Work of Jarvis--Kimura and the orbifold quantum Riemann--Roch theorem of Tseng
(\cref{sec:quantr}) computes
the Gromov--Witten $R$-calibration for $\cX$ at the orbifold point. In
\cref{sec:ac}, we verify that this agrees with the $R$-calibration of $Y$ upon analytic
continuation, concluding the proof.\\

Givental's quantization formalism for semisimple quantum cohomology
\cite{MR1866444, MR1901075, telemangiv} gives an expression for the all-genus
GW partition function of a target $\cZ$ with semisimple quantum product of
rank $N_\cZ$ as the action of a sequence of differential operators on $n$-copies
of the partition function of a point. The operators in question are obtained
through Weyl-quantization of infinitesimal symplectomorphisms determined by
the genus $0$ Gromov--Witten theory of $\cZ$ - the $S-$ and $R-$ calibrations
of $QH(\cZ)$, defined in \cref{sec:qdm,sec:secRq} respectively. Here we assume familiarity with this
story and standard notation, and review them in \cref{sec:givental}.\\

Givental's formula at a semi-simple point $\tau\in QH(\cZ)$ reads
\eq{
Z_\cZ(\mathrm{t}_\tau) = \re^{C_\cZ(\tau)}\widehat{S_\cZ^{-1}} \widehat{\Psi_\cZ} \widehat{R_\cZ}
\re^{\widehat{u/z}}\prod_{i=1}^{N_\cZ} Z_{\rm pt}(q^i),
\label{eq:qgiv}
}
where 
\eq{
C_\cZ(\tau)\triangleq \sum_{i=1}^{N_\cZ}\int (R_\cZ^{(1)})_{i}^{i}(\tau)\rd
u^i,\ \ \  R_\cZ^{(1)}(\tau)=\de_z R_\cZ(\tau,0)
\label{eq:defC}
}
 and the shifted descendent times $\mathrm{t}_\tau$ are defined in \cref{eq:shiftvar}. In \cref{eq:qgiv}, $S_\cZ$
and $R_\cZ$ are the canonical Gromov--Witten $S$- and $R$-calibrations of
$QH(\cZ)$, viewed as morphisms of Givental's symplectic space $\HH_\cZ$. 
\begin{rmk}
\label{rmk:existR}
The existence of canonical $R$-calibrations
such that \cref{eq:qgiv} holds is a consequence of Teleman's theorem \cite[Theorem~2]{telemangiv}. In the conformal case, their 
form is uniquely determined by homogeneity. In the non-conformal
case, the lack of an Euler vector field constrains the form of asymptotic
solutions of the quantum differential equation only up to right multiplication by a constant
diagonal matrix in odd powers of $z$. Therefore, in order to verify that a given $R$-calibration is equal to the canonical 
$R$-calibration guaranteed by Teleman's theorem, we need only check \cref{eq:qgiv} at a single semi-simple point. The specialization of $R_{\cX}$ to the orbifold point will be the focus of \cref{sec:quantr}.
\end{rmk} 

\begin{lem}
\label{lem:qcrcR}
Let  $\X \to X \leftarrow Y$ be a resolution diagram of Hard Lefschetz targets
with generically semi-simple quantum cohomology. 
\cref{conj:scr}  holds if and
only if the canonical $R$-calibrations coincide on the semi-simple locus,
\eq{
R_{\cX} = R_{Y}.
\label{eq:qcrcR}
}
\end{lem}

\begin{rmk}
\label{rm:spec}
The local independence
 of Givental's formula on the choice of a base point \cite{MR1901075} implies that it suffices
 to check \cref{eq:qcrcR} at any given
semi-simple point. 
\end{rmk}

\begin{proof}
We start by observing that \eqref{eq:qcrcR} implies $Z_Y =  \widehat{\U} Z_\cX$. It is
immediate that $C_\cX=C_Y$.
Moreover, \cref{eq:qgiv} gives
\ea{
Z_Y &= \re^{C_Y}\widehat{S_Y^{-1}} \widehat{\psi_Y} \widehat{R_Y} \re^{\widehat{u/z}}
\prod_{i=1}^{N_Y}Z_{\rm pt, i}\\
&= \re^{C_Y}\widehat{S_Y^{-1}}
\widehat{\bbU_0} \widehat{\psi_\cX} \widehat{R_\cX}
\re^{\widehat{u/z}} \prod_{i=1}^{N_\cX}Z_{\rm pt, i}\nn \\
 &=  
\widehat{\U} \re^{C_\cX}\widehat{S_\cX^{-1}} \widehat{\psi_\cX}
 \widehat{R_\cX} \re^{\widehat{u/z}} \prod_{i=1}^{N_\cX}Z_{\rm pt, i} 
 =  
\widehat{\U} Z_\cX.
\label{pr:pr}
}
We have made essential use of the HL condition twice: to identifty
$S$-calibrations via $\U$, which is only true if the analytic continuation of
the quantum product gives an isomorphism in big quantum
cohomology, and to ensure that no cocyle is generated 
in the quantization of products. To see the reverse implication, we note that for the string of equations \cref{pr:pr} to hold, 
\eq{
\l(\re^{C_\cX}\widehat R_\cX - \re^{C_Y}\widehat R_Y\r) \prod_{i=1}^{n+1} Z_{\rm pt,i} = 0.
\label{eq:hatqcrc}
}
 \cref{eq:hatqcrc}
implies
\eq{
\re^{C_\cX}\widehat R_\cX = \re^{C_Y}\widehat R_Y \re^{\widehat{\mathfrak{D}}}
}
for some quantized quadratic Hamiltonian $\widehat{\mathfrak{D}}$ such that 
$\widehat{\mathfrak{D}} \in  \oplus_{i=1}^n \mathscr{B}(\mathrm{Vir}_i)$, the  
Borel subalgebra of level $k\geq -1$ Virasoro constraints acting on the
product of Witten--Kontsevich tau functions $\prod_{i=1}^{n+1} Z_{\rm pt,i}$.
Imposing that
$\re^{\mathfrak{D}} \in \mathrm{Sp}_+(\HH_Y)$ then sets
$\mathfrak{D}=0$ and $C_\cX=C_Y$.
\end{proof}

\subsection{$R$-calibrations in orbifold Gromov--Witten theory}
\label{sec:quantr}

\cref{eq:qcrcR} reduces the quantum CRC to a comparison between asymptotic expression of horizontal
sections of the global quantum $D$-module, given by the Gromov--Witten
$R$-calibrations of $\cX$ and $Y$. 
In the context of toric orbifolds, the $R$-calibration is uniquely constructed \cite{jk:bg, MR2578300} in terms of group theoretic and toric data.

\begin{lem}
\label{lem:rcal}
Consider the orbifold $\cX=[\C^m/G]$ given by a diagonal representation $V$ of a finite abelian group
$G$, together with a compatible torus action. Then the canonical
$R$-calibration $R_\cX$ is uniquely determined locally around the large
radius point of $\cX$ by \cref{eq:defR}.
\end{lem}
\begin{proof}
Denote by $N_\cX$ the number of elements of $G$, which is also the
rank of the Chen--Ruan cohomology of $\cX$, and by $\fc_g$ the fundamental class of the component of the inertia orbifold labeled by $g$. Jarvis--Kimura
\cite[Prop. 4.3]{jk:bg} establish that the partition function $Z_{BG}$
agrees with $N_\cX$ copies of $Z_{\rm pt}$ after a change of variables
given by the character table $\chi_G$ of $G$ \cite[Prop. 4.1]{jk:bg}. In
operator notation,
\eq{
Z_{BG} = \widehat{\chi^{-1}_G}\prod_{i=1}^{N_\cX} Z_{\rm pt, i}.
}
On the other hand, the Gromov--Witten theory of $\cX$ is the twisted
Gromov--Witten theory of $BG$ with twisting class the inverse Euler class of
the representation $V$, thought of as a vector bundle on the classifying
stack.  Since the representation  is diagonal, $V \cong \oplus_{i=1}^m V_i$  is the sum of $m$ orbifold line bundles. Denote by $\{\mathtt{w}_i\}_{i=1}^m$ the weights of the torus action on each line bundle. Tseng, in \cite{MR2578300},  constructs a
symplectomorphism $\mathscr{D}_\cX \in \mathrm{Aut}_+(\HH_\cX)$  defined by
\eq{
\mathscr{D}_\cX \fc_g \triangleq 
\sqrt{
\frac{1}{e^{\rm eq}(V^{(0)})}} 
\exp \left(
\sum_{i=1}^m\sum_{ k\geq 0 }s_{i,k}  \frac{B_{k+1}(l_i(g)/|g|)}{(k+1)!} z^k \right)\fc_g ,
\label{eq:hht}
}
where $s_{i,0}=-\ln(\mathtt{w}_i)$, $s_{i,k}= (-\mathtt{w}_i)^{-k}(k-1)!$ ,
$V^{(0)}$ is the trivial part of the representation $V$, $B_k(x)$ is the
order $k$-Bernoulli polynomial
\eq{
\frac{\re^{x y} y}{\re^{y}-1}=\sum_{k\geq 0}\frac{B_k(x) y^k}{k!},
}
and the integer
$l_i(g)\in[0,N_\cX-1]$ is defined by  $gv_i= \re^{2\pi\ri l_i(g)/|g|}v_i$ for $v_i\in V_i$.
The orbifold Quantum Riemann--Roch theorem of \cite{MR2578300} then asserts that, upon
quantization, $\widehat{\mathscr{D}_\cX}$ acts on $Z_{BG}$ to return the GW partition
function for $\cX$, up to a scalar prefactor $\mathcal{E}_\cX$, whose precise
form will not concern us, and a
rescaling of the Darboux coordinates by $\sqrt{e^{\rm eq}(V^{(0)})}$
\cite[Section 1.2]{MR2578300}. Then,
\eq{
Z_\cX = \mathcal{E}_\cX \widehat{\l(e^{\rm eq}(V^{(0)})\r)^{-1/2}}\widehat{\mathscr{D}_\cX} \widehat{\chi^{-1}_G} \prod_{i=1}^{N_\cZ} Z_{\rm pt,i}.
\label{eq:oqrr}
} 
To compare
\cref{eq:oqrr} with \cref{eq:qgiv}, we fix the integration constant in
\cref{eq:defC} so that $\mathcal{E}_\cX|_{\rm OP}=\re^{C_\cX}$; notice that this is always
possible, since OP is a regular point for the Dubrovin connection
of $\cX$. Moreover, at the large radius point for $\cX$ we have
\eq{
(S_\cX^{-1})_{\a}^{\b}\Big|_{\rm OP} = \delta_{\a}^{\b}
\label{eq:Sdelta}
}
in flat coordinates for the orbifold Poincar\'e pairing 
of $\cX$. Define now $R_\cX(\tau,z)$
locally around OP by parallel transporting the symplectomorphism $\l(e^{\rm eq}(V^{(0)})\r)^{-1/2} \mathscr{D}_\cX
\chi_G^{-1} \in  \mathrm{Sp}_+(\HH_\cX)$: in other words $(\psi R_\cX(\tau,z)
e^{u/z})_{\a j}$ is a matrix whose columns are horizontal sections for the
Dubrovin connection such that 
\eq{
\psi_\cX R_\cX \Big|_{\rm OP}  = \l(e^{\rm eq}(V^{(0)})\r)^{-1/2} \mathscr{D}_\cX
\chi_G^{-1}.
\label{eq:defR}
}
%
Altogether, \cref{eq:defR,eq:oqrr,eq:Sdelta} 
imply that Givental's formula, \cref{eq:qgiv}, holds by construction at OP with the
$R$-calibration determined by \cref{eq:defR}. This verifies that $R_{\cX}(\tau,z)$ is the canonical
$R$-calibration guaranteed by Teleman's theorem.




\end{proof}

\begin{rmk}\label{rmk:block}
Pinning down the canonical $R$-calibration for an arbitrary toric orbifold $\cX$ can be achieved by localization.
Choose a basis for equivariant cohomology supported on the fixed points:
naturally vectors supported on different fixed points are mutually
orthogonal. The $R$-calibration is then computed as a block matrix by applying
\cref{lem:rcal} to the local geometry of  each fixed point. In particular,
when $\cX=Y$ is a toric manifold and denoting by LR the large radius limit point
for $Y$, \cref{eq:defR} becomes \cite[Thm.~9.1]{MR1901075}
\eq{
(R_Y)_{ij} \big|_{\rm LR}  = (\mathscr{D}_Y)_{i}\delta_{ij}.
\label{eq:defRY}
}
\\

\end{rmk}

We now turn our attention to the specific geometries we are investigating in depth: $\cX=[\C^3/\Z_{n+1}]$ and $Y$ its crepant resolution.
\subsection{Prolegomena on asymptotics and analytic continuation}
By virtue of \cref{eq:qcrcR}, \cref{conj:scr} can be formulated
as an identification of bases of horizontal 1-forms of
$\nabla^{(g,z)}$
upon analytic continuation to some 
chosen semi-simple point. In our case, the proof of \cref{thm:sympl} in \cref{sec:compsymp}
contains already most of the technical ingredients to compute the analytic continuation of flat
coordinates of $\nabla^{(g,z)}$ from LR to OP; 
however, a few  substantive details in the
formulation of \cref{eq:qcrcR}, particularly
in what concerns the continuation of asymptotic series,  are worth spelling
out with care.

\subsubsection{Global canonical coordinates}
\label{sec:cancoord}

First of all, the reasoning leading to \cref{eq:qcrcR} assumed implicitly a choice of global
canonical coordinates $\{u^i \in \cO(\cM_A)\}_{i=1}^{n+1}$ on $\cM_A$ - or at
least, two consistent choices of canonical coordinates for both $QH(Y)$ and
$QH(\cX)$; recall that two such sets of coordinates may differ by 
permutations and shifts by constants.  A natural way to fix this
ambiguity  is to define globally $u^i$ as the critical values
\eq{
u^i = \log\lambda(q^{\rm cr}_i)
\label{eq:uidef}
}
of the Hurwitz space superpotential \cref{eq:superpot}, where the critical points $q^{\rm cr}_i$
of $\lambda(q)$ are the roots of the polynomial equation
\eq{
\frac{a}{q^{\rm cr}_i}+b \sum_{j=1}^{n+1}\frac{\kappa_j}{1-\kappa_j q^{\rm cr}_i}=0.
\label{eq:qi}
}
The leftover permutation ambiguity is fixed upon ordering the set of critical
points such that
\eq{
\frac{\de}{\de u^i}\bigg|_{\kappa=0} \simeq P_i
\label{eq:iducl}
}
under the identification $T_{\rm LR}\cF_{\lambda,\phi}\simeq H_T(Y_T)$. \\

\subsubsection{Sectors, thimbles,  walls}

A second aspect pertains to the nature of $R_\bullet(\kappa, z)$ as a {\it formal 
asymptotic series} in $z$ (see \cref{sec:secRq}). Since $z=0$ is an irregular singularity for the global
$D$-module, asymptotic expansions of components of horizontal 1-forms at $z=0$
 depend on a choice of Stokes sector, namely, a choice of phase for $z$, as well as for the other external
parameters $\a_1$, $\a_2$ and $\kappa$ in the asymptotic analysis. Picking
one such choice poses no restriction  for the purpose of proving
\cref{eq:qcrcR}: as individual Gromov--Witten correlators depend analytically on $a$
and $b$, it is enough for us to prove \cref{eq:qcrcR} in a wedge of
parameter space. A particularly convenient choice is to pick the Stokes sector $\cS_+$ defined
by 
\eq{
\cS_+ \triangleq \{(a,b,\kappa) | \Re(a)>0, \Re(b)<0, \arg(\kappa_j)=-2\pi\ri j/(n+1) \}
\label{eq:S+}
}
where the phase of the quantum cohomology parameters in \cref{eq:S+} is fixed by our choice of
path $\rho$ in \cref{eq:rho}.  This choice turns out to trigger two
favorable consequences. \\

First off, when $(a,b,\kappa)\in \cS_+$ we can  employ  the
factorization of 
the twisted period mapping through the line integral representation
\cref{eq:line} to obtain an interpretation of the twisted periods as a sum of
steepest descent integrals (see \cref{rmk:line}). In detail, note that throughout $\cS_+$ the superpotential has
algebraic zeroes at
$q=\{0,1\}\cup\{\kappa_i^{-1}\}_{i=1}^n$. Upon regarding
$\Re(\log\lambda(q))$ as a perfect Morse function, the Lefschetz thimbles
$\mathfrak{L}_i$ emerging from each of the the critical points 
$q_i^{\rm cr}$ give a canonical basis
of the relative homology group $H_1(\bbP^1, (\lambda))$, with the
negative infinity of each downward gradient flow coinciding with the
log-divergences of $\log\lambda(q)$ at the zeroes of the superpotential. In this
basis, the Laplace expansion at small $z$ gives asymptotic solutions
$\Sigma_i(\kappa,z)$ for the flat
coordinates of $\nabla^{(g,z)}$ in the form
\eq{
\Sigma_i(\kappa,z)\triangleq \int_{\mathfrak{L}_i} \lambda^{1/z}\phi \simeq \re^{u^i/z}\cQ_i(\kappa,z), 
\label{eq:piasymp}
}
where $\cQ_i(\kappa,z) \in \cO(\cM_A) \otimes \bbC[[z]]$, and the equivalence
sign is to be intended in the sense of classical (Poincar\'e)
asymptotics. \\

The second useful consequence of the choice of parameters \cref{eq:S+} relates
to the nature of the canonical $R$-calibrations as asymptotic series. In light of the
representation \cref{eq:piasymp} of $R$-operators as the (perturbative) Laplace expansion
of a steepest descent integral around a saddle, 
in proving \cref{eq:qcrcR} we are supposed to {\it
    discard} any exponentially suppressed (non-perturbative) contribution 
  from neighboring critical points that may arise in the process of analytic
  continuation (see \cref{rmk:stokes}). Now, throughout $\cS_+\setminus
  \{|\kappa|=1\}$ we have 
\eq{
\Re \l(\frac{u_i}{z}\r) > \Re \l(\frac{u_j}{z}\r), \quad i<j,
\label{eq:uigruj}
}
which means that, away from $|\kappa|=1$,  the $i^{\rm th}$-saddle is
exponentially dominant over saddles $q_j^{\rm cr}$ with $j>i$.
In the following, we  repeatedly exploit the
fact that terms of the form
$\re^{z^{-1}\l[u_j(\rho(s))-u_i(\rho(s))\r]}$ with $j>i$, $s \in [0,1)$ are exponentially suppressed, and
therefore invisible, in the classical small $z$-asymptotics inside this region.  

\begin{rmk}
\label{rmk:stokes}
One potential source of such exponential contributions is due to the Stokes
phenomenon.
Since we are dealing with the analytic continuation of asymptotic series
of the form \cref{eq:piasymp}, a complication that may
occur when varying \cref{eq:piasymp} along $\cM_A$ is given by the possibility of a
non-trivial ``monodromy'' of the Lefschetz thimbles along the analytic continuation
path $\rho$ in \cref{eq:rho}: when $\Im (u^i/z) = \Im (u^j/z)$
for some $j>i$,
the $i^{\rm th}$ Lefschetz thimble  passes through a sub-dominant
saddle point, and in turn
an exponentially subleading contribution in the asymptotic expansion of the $i^{\rm
  th}$-period integral appears. Such jump in the asymptotics
arises across {\it walls} - and not just divisors - in moduli space,
and it may affect in principle\footnote{Generically, there is no Stokes
  phenomenon for $n=1$, where we can compute that $\Im(u_1-u_2) = 2 \ri \pi
  \a_1$ identically in $\cS_+$. For
  higher $n$, however, the possibility of the existence of the Stokes
  phenomenon can be tested numerically.  A little experimentation shows that
  $\rho$ {\it does} indeed cross one or more Stokes walls for $n>1$ and fairly
  generic $a$, $b$.} the continuation of \cref{eq:piasymp} along the 
trivial path $\rho$ in \cref{eq:rho}. The existence of Stokes walls may be all
the more delicate in light of the fact that the orbifold point belongs to the maximal anti-Stokes
submanifold $\{p \in \cM_A | \Re (u^i(p)/z) = \Re (u^j(p)/z) \quad \forall
(i,j)\}$ - see \cref{eq:uiOP} below. In the following,
we must be wary of the possible generation of exponentially
suppressed terms generated when crossing a wall, as they are no longer subdominant
when they are continued all the way up to $|\kappa|=1$, where their contribution
should be included in the
asymptotics\footnote{A typical example of this phenomenon the reader may be
  familiar with is the appearance of subleading exponentials in the
  asymptotics of the Airy integral along its anti-Stokes ray, that is, for large
negative values of the argument.}.  \\
\end{rmk}

\subsection{Proof of Theorem \ref{thm:cqcrc}}
\label{sec:ac}

\subsubsection{$R$-normalizations for $\cX$ and $Y$}

\label{sec:RX}
Let us first compute from \cref{lem:rcal} the canonical $R$-calibration for $\cX$, thought of as
 an equivariant vector bundle over the classifying stack:
\ea{
\cX &= \cO_{-1}^{-\alpha_1}\oplus\cO_{1}^{-\alpha_2}\oplus
\cO^{\alpha_1+\alpha_2} \rightarrow B \Z_{n+1} \nn \\
& \triangleq  V_1\oplus V_2\oplus V_3 \rightarrow B \Z_{n+1}.
}
$V^{(0)}$ is an equivariant vector bundle on the inertia stack; it agrees with the whole three-dimensional bundle on the component of the identity, and to the line bundle corresponding to the untwisted direction in all twisted sectors.
Therefore, 
\eq{
e^{\rm eq}(V^{(0)})=(\a_1+\a_2)\sum_{j=1}^n\fc_j +\a_1\a_2(\a_1+\a_2)\fc_{n+1}.
}
For $i=1,2,3$ and $j\in \Z_{n+1}$, the integers $l_i(j)$ are $(n+1)-j, j, 0$.
Then, using that 
$B_k(x)=(-1)^kB_k(1-x)$  and  $B_{2k+1}=0$ for $n>0$, \cref{eq:hht} gives
\ea{
\mathscr{D}_\cX &= \sum_{j=1}^n \left(\frac{\a_2}{\a_1}\right)^{\frac{1}{2}-\frac{j}{n+1}}\exp\left[
 \sum_{k>0} \left(-\frac{B_{k+1}\left(\frac{j}{n+1}\right)}{(-\alpha_1)^k}+
 \frac{B_{k+1}\left(\frac{j}{n+1}\right)}{\alpha_2^k}    \nn
 -\frac{B_{k+1}}{(\a_1+\alpha_2)^k}
 \right) \frac{z^k}{k(k+1)}
 \right]\fc_j \\
 & +  \exp\left[
 \sum_{k>0} 
\left( \frac{1
 }{\alpha_1^{2k-1}}  +
  \frac{1
 }{\alpha_2^{2k-1}}-
  \frac{1
 }{(\alpha_1+\a_2)^{2k-1}}
 \right)
 \frac{ B_{2k}
z^{2k-1}}{2k(2k-1)}\right] \fc_{n+1}. 
\label{eq:deltaX}
}
As far as $Y$ is concerned, by \cref{rmk:block}, we apply \cref{lem:rcal} to
the local geometry of each fixed point $p_i$. Then, denoting by $(w_i^-,w_i^+,\a_1+\a_2)$ the
characters of the torus action on the tangent space $T_{p_i}$ at the $i^{\rm th}$ fixed point, as in \cref{eq:tweights},
 we have from \cref{eq:hht} that
\eq{
\mathscr{D}_Y p_i = \exp\l[-\sum_{k>0}\frac{B_{2k} z^{2k-1}}{2k(2k-1)}\l( ({w}_i^+)^{1-2k}+({w}_i^-)^{1-2k}+(\a_1+\a_2)^{1-2k} \r)\r]p_i.
\label{eq:deltaY}
}

\subsubsection{Analytic continuation}

To compare the classical $R$-operators in \cref{eq:qcrcR},
we  avoid troubles with the Stokes phenomenon 
as follows: we fix the $R$-calibration {\it first} at OP, where pinning down
the contribution of each critical point is potentially delicate,
and then compute its continuation to $|\kappa|<1$ where the
classical asymptotics are controlled  by the leading saddle. Then, \cref{eq:uigruj}
grants us the right to safely ignore any possible issues stemming from the generation of subleading
exponential terms by analytic continuation through a wall when $|\kappa|<1$. \\

At the orbifold point $\kappa_j=\omega^{-j}$, \cref{eq:qi} gives
\eq{
q^{\rm cr}_{i}\big|_{\rm
  OP}=\omega^{\sigma(i)}\l(\frac{a}{a-(1+n)b}\r)^{\frac{1}{n+1}}
\label{eq:qiOP}
}
for some permutation $\sigma\in S_{n+1}$, which by continuity is locally constant in $(a,b)$. Noting that the roots of \cref{eq:qi} admit a smooth limit
at $b=0$, where
\eq{
q^{\rm cr}_i\big|_{b=0} = \kappa_i^{-1},
}
and comparing to \cref{eq:uidef,eq:iducl} sets $\sigma = \mathrm{id}$. Therefore,
\eq{
u^j|_{\rm OP}= \a_1 \log (\a_1)+\a_2 \log(-\a_2) -(\a_1+\a_2) \log
\left(-\a_1-\a_2\right)+2\ri \a_1(j-n/2).
\label{eq:uiOP}
}
Note that $\Re(u^i/z)|_{\rm OP}=\Re(u^j/z)|_{\rm OP}$ for all $(i,j)$, as
anticipated in \cref{rmk:stokes}. Since $\arg\lambda(q)\big|_{\rm OP}=\a_1
\arg q$ by \cref{eq:superpot}, at the orbifold point the constant phase/steepest descent contour
$\mathfrak{L_j}$  must be contained in the straight line through the origin making an angle
of $2\pi j/(n+1)$ with the positive semi-axis. In particular, since $|q^{\rm
  cr}_{i}|<1$ in $\cS_+\cap \mathrm{OP}$ by \cref{eq:qiOP}, the union of the downward gradient lines
emanating from $|q^{\rm cr}_{j}|<1$ is given by the segment $[0, \omega^{j}]$. Then, by \cref{eq:line},
twisted periods coincide with line integrals over steepest descent paths for $\lambda(q)^{1/z}\phi|_{\rm
  OP}$, and we have
\eq{
  \Pi_i(\kappa,z)\big|_{\rm OP} = \int_{\mathfrak{L_i}}\lambda^{1/z}\phi \sim \re^{u^i/z}.
\label{eq:PiSigma}
}
%
Now, in flat coordinates $x^\alpha$ for the orbifold quantum product, the
$R$-calibration must satisfy by \cref{eq:PiSigma}
\eq{
 (\psi_\cX R_\cX \re^{u/z})_{\a,i} \rd x^\a = \rd \Pi_i \cN^\cX_i
\label{eq:NX}
}
in a neighborhood of OP for some constant normalization factor $\cN^\cX\in
H(\cX) \otimes \bbC[[z]]$. 
The left hand side is uniquely determined by \cref{eq:defR,eq:deltaX}.
For the right hand side, we have already computed the differential of the twisted period map at the
orbifold point in \crefrange{eq:dePix0}{eq:dePix}. Putting it all together, we
obtain\footnote{Notice that this is a severely overconstrained system for the
  unknown $\cN_i$, as is apparent from the fact that the l.h.s. has no {\it a
    priori} reason to be diagonal in the indices $i$, $j$ (see also \cref{rmk:existR}).  Existence of solutions is a non-trivial
  statement about the boundary values at OP of the twisted periods and their
  derivatives, \cref{eq:dePix0,eq:dePix}.}
\eq{
\sum_{\a=1}^{n+1}(\cB^{-1})_{i,\a} (e^{\rm eq}(V^{(0)})^{-1/2}\mathscr{D}_\cX)_\a
(\chi^{-1})_{\a,j} \re^{u^j/z}\big|_{\rm OP} \simeq \cN^\cX_i \delta_{ij}.
\label{eq:NX2}
}
The small $|z|$-asymptotics of the left-hand side, by
\cref{eq:beta,eq:dePix0,eq:dePix}, is computed by the steepest descent
asymptotics of the Beta integral in \cref{eq:beta,eq:beta2}. With our choice of sector $\cS_+$, as all the
$\Gamma$-functions appearing in \cref{eq:matrD2} have arguments with large and positive
real part for small $|z|$, the latter is determined by the generalized
Stirling formula:
%
\eq{
\Gamma(x+y) x^{-x}\re^{x}x^{1/2-y} \stackrel{}{\simeq}
\sqrt{2\pi}\exp\l(\sum_{k>0}\frac{B_{k+1}(1-y)}{k(k+1)}x^{k}\r), \quad  \Re (x)\gg 0.
\label{eq:stirling}
}
Keeping track judiciously of the (rather massive) cancellations occurring 
upon plugging \cref{eq:deltaX,eq:stirling,eq:uiOP,eq:matrB1,eq:matrD1,eq:matrD2,eq:matrV}
into \cref{eq:NX2}, we get that \cref{eq:NX2} admits the unique solution
\eq{
\cN^\cX_i = -b^{-1}\sqrt{\frac{z}{2\pi}}.
\label{eq:NX3}
}
%
%
%

Let us now analytically continue \cref{eq:NX,eq:NX3} to LR along $\rho$. By \cref{eq:AB,eq:uigruj}, $\rho \cap (\cS_+\setminus
|\kappa|<1)$ does not contain anti-Stokes points. The classical asymptotics
around the $i^{\rm th}$-saddle of the continuation of \cref{eq:NX} is therefore
computed unambiguously as a formal power
series in $z$ from the classical asymptotics of $\re^{-u^i/z}\rd \Pi_i$. 
%
%
%
%
Denote by $(\widetilde{R_\cX})_{ij}\rd u^j\in \Omega^1(\cM_A)[[z]]$ the formal
series obtained for every $i=1, \dots, n$ from the
analytic continuation along $\rho$ of the coefficients of
$\cN_i^\cX\re^{-u^i/z}\rd\Pi_i=(\psi_\cX R_\cX)_{\a,i}\rd x^\a$ to $\kappa \sim
0$. From the discussion above, we isolate for each $i$ the contributions from
the leading saddle to obtain
\eq{
(\widetilde{R_\cX})_{ij}\rd u^j \simeq \re^{-u^i/z}\cA^{-1}_{ii} \cN_i^\cX \rd J^Y_i(u,z) 
\label{eq:omJ}
}
as 1-form valued formal series in $z$; notice that the off-diagonal terms of $\cA^{-1}_{ij}$
have become invisible in the asymptotics after projecting out subleading exponentials.
  Expressing the components of \cref{eq:omJ} in
normalized canonical coordinates and taking the $\kappa\to 0$ limit, 
%
%
by \cref{eq:ucl,eq:Jloc}, we have
\eq{
J^Y_i(\kappa,z) = z \re^{u^i_{\rm cl}/z}\l(1+\cO(\kappa)\r),
\label{eq:JYcl}
}
where $u^i_{\rm cl}$ are coordinates for the idempotents of the classical $T$-equivariant
cohomology ring of $Y$ defined by
%
$u^j_{\rm cl} P_j \triangleq i^*_j\l(t^\mu \phi_\mu\r)$
%
in terms of the localization \crefrange{eq:ab1}{eq:ab2} of $\phi_\mu\in
H^2(Y)$ to the $T$-fixed
points. Explicitly,
\ea{
u^i_{\rm cl}  &= t_0 +  \a_2 \sum_{j \geq i} t_j
(n+1-j) +  \a_1 \sum_{j<i} j t_j, \nn \\
&= t_0 + z ((n+1-i)b-a)\ln\kappa_i  +  \frac{za}{n+1}
\sum_{j=1}^{n} \ln\kappa_j - z b \sum_{j=i+1}^n \ln\kappa_j.
\label{eq:ucl}
}
By the discussion of \cref{sec:cancoord,eq:iducl}, the limit $\log\epsilon^i := \lim_{\kappa\to 0}(u^i -
u^i_{\rm cl})$ must be finite. 
A direct calculation from \cref{eq:superpot} gives
\eq{
\epsilon_i^{1/z} = \l(\frac{-w_i^+}{z}\r)^{-\frac{w_i^+}{z}
}\l(\frac{w_i^-}{z}\r)^{-\frac{w_i^-}{z}
}\l(-b\r)^{-b}\re^{-\ri \pi (n+1-i) b}.
}
Then,
\ea{
\lim_{\kappa\to 0}\re^{-u^j/z} \de_i J^Y_j   
&=
\sqrt{\Delta_i(\kappa)}\Big|_{\kappa=0} \epsilon_i^{-1/z} \delta_{ij} \nn \\
 &= \ri \sqrt{(\a_1+\a_2) w_i^+ w_i^-} \delta_{ij}.
}
where  $\Delta_i(\kappa)$ is the Poincar\'e square-norm of $\de_{u^i}$ at $\kappa$, $\de_i
\triangleq \de_{\tilde u^i} = \sqrt{\Delta_i} \de_{u^i}$, and
we pick the positive determination for the square root for all $i$. This sets
\eq{
(\widetilde{R_\cX})_{ij} \big|_{\rm LR} = 
\sqrt{\frac{ w_i^+
    w_i^-}{2\pi b}}\frac{1}{\epsilon_i^{1/z} \cA_{ii}}.
}
With our choice \cref{eq:S+} of Stokes sector, all arguments of the
$\Gamma$-functions appearing in the diagonal of \cref{eq:matrA} have large positive real
part for small $|z|$.  Making use again of Stirling's formula, \cref{eq:stirling}, yields
\ea{
\cA_{ii} 
&=  z \re^{\pi\ri \l(n-i+1\r) b}
\frac{\Gamma\l(1+w_i^-/z\r)}{\Gamma\l(1-b\r) \Gamma\l(-w_i^+/z\r)}, \nn
\\
& \simeq    
\frac{1}{\epsilon_i^{1/z}(\mathscr{D}_Y)_i}
 \sqrt{\frac{w_i^- w_i^+}{2\pi b}},
}
so that
\eq{
(\widetilde{R_\cX})_{ij} \big|_{\rm LR} = (\mathscr{D}_Y)_{i}\delta_{ij} 
}
and therefore $\widetilde{R_\cX}=R_Y$ near LR by \cref{eq:defRY}, which concludes the proof.
\begin{flushright}$\square$\end{flushright}

\begin{appendix}

\section{Gromov--Witten theory background}
\label{sec:back}
This appendix reviews and synthesizes key aspects of Gromov--Witten theory, for
the benefit of the non-expert reader. Let $\cZ$ be a smooth Deligne--Mumford stack with coarse moduli
space $Z$ and suppose that $\cZ$ carries an algebraic $T\simeq\bbC^*$ action with
zero-dimensional 
fixed loci.  Write $I\cZ$ for the inertia stack of $\cZ$, 
$\mathrm{inv}:I\cZ\to I\cZ$ for its canonical involution and
$i:I\cZ^T\hookrightarrow I\cZ$ for
the inclusion of the $T$-fixed loci into $I\cZ$. 
The equivariant Chen--Ruan cohomology ring $H(\cZ) \triangleq H^{\rm orb}_{T}(\cZ)$ of $\cZ$ is a finite rank free module over
the $T$-equivariant cohomology of a point $H(BT)\simeq
\bbC[\nu]$, where $\nu=c_1(\cO_{BT}(1))$; we define
$N_\cZ \triangleq\rank_{\bbC[\nu]} H(\cZ)$ and  denote by $\Delta_\cZ$ the free module over $\bbC[\nu]$ spanned by
the $T$-equivariant lifts of Chen--Ruan cohomology classes
having age-shifted degree at most two. We assume
that odd cohomology groups vanish in all degrees.



\subsection[Quantum $D$-modules in GW theory]{Quantum $D$-modules in
  GW theory}
\label{sec:qdm}

The $T$-action on $\cZ$ gives a non-degenerate inner product on
$H(\cZ)$ via the equivariant orbifold Poincar\'e pairing
\eq{
\eta(\theta_1,\theta_2)_{\cZ} \triangleq \int_{I\cZ^T}\frac{i^*(\theta_1 \cup \mathrm{inv}^*
  \theta_2)}{e(N_{I\cZ^T/I\cZ})},
\label{eq:pair}
}
and it induces a torus action on the moduli space $\overline{\cM}_{g,n}(\cZ, \beta)$ of degree $\beta$ twisted stable maps
\cite{MR2450211, MR1950941} from genus $g$ orbicurves to $\cZ$.
For classes $\theta_1, \dots, \theta_n\in H(\cZ)$ and
integers $r_1, \dots, r_n \in \bbN$, the Gromov--Witten
invariants of $\cZ$
\ea{
\bra \sigma_{r_1}(\theta_1) \dots \sigma_{r_n}(\theta_n) \ket_{g,n,\beta}^\cZ
& \triangleq \int_{[\overline{\cM}_{g,n}(\cZ,
    \beta)]_T^{\rm vir}} \prod_{i=1}^n \mathrm{\ev}^*_i \theta_i
\psi_i^{r_i}, \label{eq:gwdesc} \\
\bra \theta_1 \dots \theta_n \ket_{g,n,\beta}^\cZ & \triangleq \bra \sigma_{0}(\theta_1) \dots
\sigma_{0}(\theta_n) \ket_{g,n,\beta}^\cZ, 
\label{eq:gwprim}
}
define a sequence of multi-linear functions on $H(\cZ)$ with values in the
field of fractions $\bbC(\nu)$ of $H(BT)$ (the integrals in \cref{eq:gwdesc} are defined by localization). The correlators \cref{eq:gwprim}
(respectively, \cref{eq:gwdesc} with $r_i>0$) are the {\it
  primary} (respectively, {\it descendent}) Gromov--Witten invariants of
$\cZ$. \\

Fix a basis
$\{\phi_i\}_{i=0}^{N_\cZ-1}$ of $H(\cZ)$ such that $\phi_0=\mathbf{1}_\cZ$
and $\phi_j$, $1\leq j \leq b_2(Z)$ are untwisted 
$T$-equivariant divisors
in $Z$. Denote by $\{\phi^i\}_{i=0}^{N_\cZ-1}$ the dual basis with respect to the pairing.  Let $\tau=\sum\tau^i\phi_i$ denote a general point of $H(\cZ)$.  The WDVV equation for primary Gromov--Witten invariants \cref{eq:gwprim} defines a family of associative
deformations $\circ_\tau$ of the $T$-equivariant Chen--Ruan cohomology ring of $\cZ$ via
\eq{
\eta\l(\theta_1 \circ_\tau \theta_2, \theta_3\r)_{\cZ} \triangleq \bra\bra \theta_1, \theta_2, \theta_3 \ket\ket_{0,3}^\cZ(\tau)
\label{eq:qprod1}
}
where
\eq{
\label{doublebra}
\bra\bra \theta_1, \dots, \theta_k \ket\ket_{0,k}^\cZ(\tau) \triangleq \sum_{\b}\sum_{n\geq 0} \frac{\big\langle \theta_1,\dots,\theta_k,
  \overbrace{\tau,\tau,\ldots,\tau}^{\text{$n$
        times}} \big\rangle_{0,n+k,\beta}^\cZ}{n!} \in \bbC((\nu)) ,
}
and the index $\beta$ ranges over the semigroup of effective curves
$\mathrm{Eff}(\cZ) \subset H_2(Z, \bbQ)$; we denote by $l_\cZ \triangleq l_\cZ$
its rank. Applying the Divisor Axiom \cite{MR2450211},  \cref{doublebra} can be rewritten as
\eq{
\eta\l(\theta_1 \circ_\tau \theta_2, \theta_3\r)_{\cZ}= \sum_{\b\in \mathrm{Eff}(\cZ), n\geq 0} \frac{\big\langle \theta_1,\theta_2,\theta_3,
  \overbrace{\tau',\tau',\ldots,\tau'}^{\text{$n$
        times}} \big\rangle_{0,n+3,\beta}^\cZ}{n!}\re^{\tau_{0,2} \cdot \beta}
\label{eq:qprod2}
}
where we have decomposed $\tau=\sum_{i=0}^{N_\cZ-1} \tau^i \phi_i = \tau_{0,2}+\tau'$ as
\ea{
\tau_{0,2} &= \sum_{i=1}^{l_\cZ} \tau^{i} \phi_{i}, \\
\tau' &= \tau^0 \mathbf{1}_\cZ + \sum_{i=l_\cZ+1}^{N_\cZ-1} \tau^i \phi_i.
\label{eq:Tprime}
}

The quantum product \cref{eq:qprod2} is a formal Taylor series in $(\tau',
\re^{\tau_{0,2}})$. Suppose that it is actually {\it convergent} in a contractible
open set $U \ni (0,0)$; this is the case for many toric orbifolds
\cite{MR1653024, Coates:2012vs} and for
all the examples of \cref{sec:An}. The quantum product $\circ_\tau$ is an
analytic deformation of the Chen--Ruan cup product $\cup_{\rm CR}$, to which
it reduces in the limit $\tau' \to 0$, $\mathfrak{Re}(\tau_{0,2}) \to -\infty$. Thus, the holomorphic
family of rings $H(\cZ) \times U \to U$, together with the equivariant Poincare' pairing and the
associative product \cref{eq:qprod2}, gives $U$ the structure of a
(non-conformal) Frobenius manifold $QH(\cZ)\triangleq(U, \eta, \circ_\tau)$
\cite{Dubrovin:1994hc}; this is the {\it quantum cohomology ring} of $\cZ$. We  refer to the Chen--Ruan limit $\tau' \to 0$, $\mathfrak{Re}(\tau_{0,2})
\to -\infty$ as the {\it large radius limit point} of $\cZ$. \\

Assigning a Frobenius structure on $U$ amounts to endowing the trivial
cohomology bundle $TU \simeq H(\cZ) \times U \to U$ with a flat
pencil of affine connections \cite[Lecture~6]{Dubrovin:1994hc}. Denote by $\nabla^{(\eta)}$ the Levi--Civita connection
associated to the Poincar\'e pairing on $H(\cZ)$; in Cartesian coordinates
for $U\subset H(\cZ)$ this reduces to the ordinary de Rham differential
$\nabla^{(\eta)}=d$. The one parameter family of covariant
derivatives on $TU$
\eq{
\nabla^{(\eta,z)}_X \triangleq  \nabla^{(\eta)}_X+z^{-1} X \circ_\tau.
\label{eq:defconn1}
}
is called the \textit{Dubrovin connection}.
The fact that the quantum product is commutative, associative and integrable implies that
$R_{\nabla^{(\eta,z)}}=T_{\nabla^{(\eta,z)}}=0$ identically in $z$; this
statement is equivalent to the WDVV
equations for the genus zero Gromov--Witten potential.  The equation for the horizontal
sections of $\nabla^{(\eta,z)}$,
\eq{
\nabla^{(\eta,z)} \omega =0,
\label{eq:QDE}
}
is a rank-$N_\cZ$ holonomic system of
coupled linear PDEs. We  denote by $\cS_\cZ$ the vector space of solutions
of \cref{eq:QDE}: a $\bbC((z))$-basis of $\cS_\cZ$ is by definition given by the gradient of
a flat frame $\tilde \tau (\tau,z)$ for the deformed connection
$\nabla^{(\eta,z)}$. The Poincar\'e
pairing induces a non-degenerate inner product $H(s_1,s_2)_{\cZ}$ on $\cS_\cZ$ via
\eq{
H(s_1, s_2)_\cZ \triangleq \eta(s_1(\tau, -z),s_2(\tau,z))_\cZ.
\label{eq:pairDmod}
}
The triple $\mathrm{QDM}(\cZ)\triangleq(U,\nabla^{(\eta,z)}, H(,)_\cZ)$ defines a {\it
  quantum D-module} structure on $U$, and the system \cref{eq:QDE} is the {\it quantum differential
    equation} (in short, QDE) of $\cZ$. 
\begin{rmk}
\label{rmk:fuchsLR}
The assumption that the quantum product
  \cref{eq:qprod2} is analytic in $(\tau',\re^{\tau_{0,2}})$ around the large radius
  limit point translates into the statement that the QDE \cref{eq:QDE} has a
  Fuchsian singularity along $\cup_{i=1}^{l_\cZ} \{q_i\triangleq\re^{\tau^i}=0\}$. \\
\end{rmk}
In the same way in which the genus zero primary theory of $\cZ$ defines a quantum
$D$-module structure on $H(\cZ) \times U$, the genus zero gravitational
invariants \cref{eq:gwdesc} furnish a basis of horizontal sections
of $\nabla^{(\eta,z)}$ \cite{MR1408320}.  For every $\theta\in
H(\cZ)$, a flat section of the $D$-module is given by an
$\mathrm{End}(H(\cZ))$-valued function $S_\cZ(\tau,z):H(\cZ)\to \cS_\cZ$ defined as
\eq{
S_\cZ(\tau,z)\theta \triangleq \theta-\sum_{k=0}^{N_\cZ-1}\phi^k\bra\bra\phi_k,\frac{\theta}{z+\psi}\ket\ket_{0,2}^\cZ(\tau)
\label{eq:fundsol}
}
where $\psi$ is a cotangent line class and we expand the denominator as a
geometric series in $-\frac{\psi}{z}$.
The
pair $(\mathrm{QDM}(\cZ), S_\cZ)$ is called the {\it S-calibration} of the Frobenius structure
$(U, \circ_\tau, \eta)$. \\

The flows of coordinate vectors for the flat frame of $TH(\cZ)$ induced by
$S_\cZ(\tau,z)$ give a basis of flat coordinates
 of
$\nabla^{(\eta,z)}$, which is defined uniquely up to an additive $z$-dependent
 constant. A canonical basis is obtained upon applying the String Axiom:
define the {\it $J$-function} $J^\cZ(\tau,z):U \times \bbC \to H(\cZ)$ by
\eq{
J^\cZ(\tau,z) \triangleq zS_\cZ(\tau,-z)^\dagger\mathbf{1}_\cZ
\label{eq:Jfun1}
}
where $S_\cZ(\tau,z)^\dagger$ denotes the adjoint of $S_\cZ(\tau,z)$ under $\eta(-,-)_\cZ$. Explicitly, 
\eq{
\label{eq:resj}
J^\cZ(\tau,z) = (z+\tau^0)\mathbf{1}_\cZ+\tau^1\phi_1+...+\tau^{N_\cZ} \phi_{N_\cZ}+\sum_{k=0}^{N_\cZ-1} \phi^k\bra\bra
\frac{\phi_k}{z-\psi_{n+1}}\ket\ket_{0,1}^\cZ(\tau).
}
Components of $J^\cZ(\tau,-z)$ in the $\phi$-basis give flat coordinates of
\cref{eq:defconn1}; this is a consequence of \cref{eq:Jfun1} combined with
the String Equation. From \cref{eq:resj}, the undeformed flat coordinate system is obtained in the
limit $z\to\infty$ as
\eq{
\lim_{z\to \infty}  \l(J^\cZ(\tau,-z)+z \mathbf{1}_\cZ\r) = \tau.
}
\\

By \cref{rmk:fuchsLR}, a loop around the origin in the variables $q_i=\re^{\tau^i}$
gives a non-trivial monodromy action on the $J$-function. Setting $\tau'=0$ in \cref{eq:resj} and applying the Divisor
Axiom then gives \cite[Proposition~10.2.3]{MR1677117}
\ea{
& J^{\cZ, \rm sm}(\tau_{0,2},z)  \triangleq  J^\cZ(\tau,z)\Big|_{\tau'=0} \nn \\
=& z \re^{\tau^1 \phi_1/z}\dots\re^{\tau^{l_\cZ} \phi_{l_\cZ}/z}
\l(\mathbf{1}_\cZ+ \sum_{\beta,k}\re^{\tau^1 \beta_1}\dots\re^{\tau^{l_\cZ}\beta_{l_\cZ}}\phi^k\bra
\frac{\phi_k}{z(z-\psi_{1})}\ket_{0,1,\beta}^\cZ\r).
\label{eq:Jred}
}
In our situation
where the $T$-action has only zero-dimensional fixed loci $\{P_i\}_{i=1}^{N_\cZ}$, write 
\eq{
\phi_i \to \sum_{j=1}^{N_\cZ} c_{ij}(\nu) P_j, \quad i=1, \dots, l_\cZ,
}
for the image of $\{\phi_i \in H^2(\cZ, \bbC)\}_{i=1}^{l_\cZ}$ under the
Atiyah--Bott isomorphism.   The image of each $\phi_i$ is concentrated on the fixed point cohomology classes with trivial isotropy which 
    are idempotents of the
classical Chen--Ruan cup
  product on $H(\cZ)$. Therefore, the components of the $J$-function in the fixed points basis
\eq{
J^{\cZ, \rm sm}(\tau_{0,2},z) =: \sum_{j=1}^{N_\cZ} J_j^{\cZ, \rm sm}(\tau_{0,2},z) P_j
}
satisfy
\eq{
J_j^{\cZ, \rm sm}(\tau_{0,2},z) = z \re^{\sum_{i=1}^{l_\cZ} \tau^i
  c_{ij}/z}\l(1+\cO\l(\re^{\tau_{0,2}}\r)\r)
\label{eq:Jloc}
}
where the $\cO\l(\re^{\tau_{0,2}}\r)$ term on the right hand side is an analytic power
series around $\re^{\tau_{0,2}}=0$ by \cref{eq:Jred} and the assumption of convergence
of the quantum product. The localized basis $\{P_j\}_{j=1}^{N_\cZ}$ therefore
diagonalizes the monodromy around large radius: by \cref{eq:Jloc}, each
$J_j^{\cZ, \rm sm}(\tau_{0,2},z)$ is an eigenvector of the monodromy around a loop in the
$q_i$-plane encircling the large radius
limit of $\cZ$ with eigenvalue $\re^{2\pi\ri c_{ij}/z}$.

\subsubsection{Toric data and trivializations}
Suppose that
$c_1(\cZ)\geq 0$ and that the coarse moduli space $Z$ is a
semi-projective toric variety given by a GIT quotient of $\bbC^{\dim\cZ+n_\cZ}$ by $(\bbC^*)^{n_\cZ}$.
In this setting, the global quantum $D$-module arises naturally in the
form of the Picard--Fuchs system associated to $\cZ$ \cite{MR1653024,
  MR2271990, ccit2}. The scaling factor $\mathfrak{h}_\cZ^{1/z}$ then measures the
discrepancy between the small $J$-function and the canonical basis-vector of
  solutions of the Picard--Fuchs system (the {\it $I$-function}), restricted to zero
  twisted insertions\footnote{See \cite{MR2486673} for a discussion of why in equivariant Gromov--Witten theory
    $I^\cZ$ and $J^{\cZ, \rm small}$ could {\it a priori} differ, even in the
    semi-positive case, by a uniquely determined scaling factor induced by
    the String Equation.}:
\eq{
\mathfrak{h}_\cZ^{1/z}(\tau_{0,2}) J^{\cZ, \rm small}(\tau_{0,2}, z) =
I^\cZ({\frak a}(\tau_{0,2}),z),
\label{eq:scalingIJ}
}
where ${\frak a}(\tau_{0,2})$ is the inverse mirror map. As a consequence of \cref{eq:scalingIJ}, the
scaling factor $\mathfrak{h}_{\cZ}$ is 
determined by the toric data defining $\cZ$ \cite{MR1653024,
  MR2529944, ccit2}. Let $\Xi_i\in H^2(Z)$ be the $T$-equivariant Poincar\'e dual of the reduction to the quotient of the $i^{\rm th}$
coordinate hyperplane in $\bbC^{\dim\cZ+n_\cZ}$ and write
$\zeta^{(j)}_i=\mathrm{Coeff}_{\phi_j}\Xi_i \in \bbC[\nu]$ for the coefficient of the projection of
$\Xi_i$ along $\phi_j\in H(\cZ)$ for $j=0, \dots, n_\cZ$. Defining, for every
$\beta$, $D_i(\beta) \triangleq \int_\beta \Xi_i$ and $J^\pm_\beta\triangleq\l\{j \in
\{1,\dots,\dim\cZ+n_\cZ \} | \pm D_j(\beta)>0\r\}$, we have
\ea{
\tau^l &= \log{{\frak a}_l} + \sum_{\beta\in \mathrm{Eff}(\cZ)}{\frak a}^\beta \frac{\prod_{j_{-}\in
    J^-_\beta}(-1)^{D_{j_{-}}(\beta)} |D_{j_{-}}(\beta)|!}{\prod_{j_{+}\in
    J^+_\beta}D_{j_{+}}(\beta)!}\sum_{k_{-}\in
  J^-_\beta}\frac{-\zeta^{(l)}_{k_{-}}}{D_{k_{-}}(\beta)}, \quad l=1,\dots,
n_\cZ, 
}
\ea{
\mathfrak{h}_\cZ &= \exp\l[\sum_{\beta\in \mathrm{Eff}(\cZ)}{\frak a}^\beta \frac{\prod_{j_{-}\in
    J^-_\beta}(-1)^{D_{j_{-}}(\beta)} |D_{j_{-}}(\beta)|!}{\prod_{j_{+}\in
    J^+_\beta}D_{j_{+}}(\beta)!}\sum_{k_{-}\in J^-_\beta}\frac{-\zeta^{(0)}_{k_{-}}}{D_{k_{-}}(\beta)}\r].
\label{eq:hz}
}
%

\subsection{Givental's symplectic structures and quantization}
\label{sec:givental}

Let $(\HH_\cZ,\Omega_\cZ)$ be a pair given by
 the infinite dimensional vector space 
\eq{
\HH_\cZ\triangleq H(\cZ)\otimes\cO(\bbC^*)
\label{eq:givsp}
}
endowed with the symplectic form
\eq{
\Omega_\cZ(f,g)\triangleq \Res_{z=0} \eta(f(-z),g(z))_\cZ.
\label{eq:sympform}
}
A general point of $\HH_\cZ$ can be written in Darboux coordinates for
\cref{eq:sympform}; as
\eq{
\sum_{k\geq 0}\sum_{\a=0}^{N_\cZ-1} q^{\a,k} \phi_\a z^k+\sum_{l\geq 0}\sum_{\b=0}^{N_\cZ-1} p_{\b,l} \phi_\b (-z)^{-k-1}.
}
 We  call $\HH_\cZ^+$ the Lagrangian subpace spanned by
$q^{\a,k}$. \\

The genus zero Gromov--Witten theory of $\cZ$ can be compactly codified through the symplectic
geometry of $\HH_\cZ$ as follows \cite{MR2115767}. The generating function of genus zero descendent Gromov--Witten invariants of
$\cZ$,
\eq{
\cF_0^\cZ \triangleq \sum_{n=0}^\infty \sum_{\beta \in \mathrm{Eff}(\cZ)}\sum_{\substack{a_1, \dots a_n \\ r_1
    \dots r_n}} \frac{\prod_{i=1}^n \mathrm{t}^{a_i,r_i}}{n!}\bra
\sigma_{r_1}(\phi_{a_1}) \dots \sigma_{r_n}(\phi_{a_n}) \ket_{0,n,\beta}^\cZ,
\label{eq:descpot}
}
is the germ of an analytic function on $\HH_\cZ^+$ upon identifying
$\mathrm{t}^{0,1}=q^{0,1}+1$, $\mathrm{t}^{\a,n}=q^{\a,n}$; under the assumption of convergence
of the quantum product, the coefficients of $\mathrm{t}^{\a_1,n_1}\dots
\mathrm{t}^{\a_r, n_r}$ with 
$n_i (\deg \phi_{\a_i}-2) \neq 0$ are analytic functions of $\re^{\mathrm{t}_{2,0}}$
in a neighborhood of the origin; the mirror theorem of \cite{ccit2} guarantees that
this is the case when $\cZ$ is semi-positive. As is often common \cite{lee2004frobenius}, in writing \cref{eq:descpot} and
in the following we chose to dispose altogether of the Novikov variables; there is no loss of information
however here about the degree of the curves by virtue of the Divisor Axiom;
see e.g. \cite[Remarks~5.3 and 5.5]{MR2529944} for a discussion of both the
primary and the descendent theory.\\

The graph $\LL_\cZ$ of the differential of
\cref{eq:descpot}, 
\eq{
\LL_\cZ=\Big\{(q,p) \in \HH_\cZ | p_{l,\beta}=\frac{\de\cF_0^\cZ}{\de q^{l,\beta}}\Big\},
\label{eq:lcone}
}
is by design a formal germ of a Lagrangian submanifold. This is a ruled cone \cite{MR2115767}, as a consequence of the genus zero Gromov--Witten axioms,
depending analytically on the small quantum cohomology variables
$\mathrm{t}^{0,2}$ around the large radius limit point of $\cZ$. By the equations defining the cone, the $J$-function $J^\cZ(\tau,-z)$ yields a family of
elements of $\LL_\cZ$ parameterized by $\tau \in H(\cZ)$, which is uniquely
determined by its large $z$ asymptotics $J(\tau, -z)=-z+\tau +
\cO(z^{-1})$. Conversely, the genus zero topological recursion relations imply
that $\LL_\cZ$ can be reconstructed entirely from $J^\cZ(\tau, z)$.

\subsubsection{The $R$-calibration and quantization} 
\label{sec:secRq}
When $\cZ$ is a manifold,
Givental's theorem for equivariant Gromov--Witten invariants
\cite{MR1866444, MR1901075} asserts that the higher genus theory of $\cZ$
is obtained through Weyl-quantization of a pair $(S_\cZ, R_\cZ)$ of
symplectic automorphisms of $(\HH_\cZ,
\Omega_\cZ)$, both of which are determined {\it from genus zero
data alone}. More in
detail, the Gromov--Witten canonical $S$-calibration satisfies
\eq{
\eta\l(S_\cZ(\tau,-z) \theta, S_\cZ(\tau, z) \theta'\r)_\cZ = \eta(\theta, \theta')_\cZ
}
for arbitrary cohomology classes $\theta$, $\theta'\in H(\cZ)$,
as can be readily seen upon differentiating the left hand side with respect to
$\tau$. Allowing $\theta$, $\theta'$ to be formal
cohomology-valued Laurent series
in $z$, this implies that for fixed $\tau$, $S_\cZ(\tau,z)$ belongs to 
the {\it negative symplectic loop group} $\mathrm{Sp}_-(\HH_\cZ)$ of $\cZ$: an element of
$GL(H(\cZ))[[z^{-1}]]$ which is a 
symplectic automorphism of $(\HH_\cZ,\Omega_\cZ)$.\\

The $R_\cZ$-calibration is instead an element of the {\it positive} symplectic
loop group $\mathrm{Sp}_+(\HH_\cZ) \cap GL(H(\cZ))[[z]]$, constructed as follows. Let $\tilde\tau \in U$ be such that the Frobenius algebra on $T_{\tau} H(\cZ)$ is
semisimple for $\tau$ in a neighborhood $V_{\tilde\tau}$ of $\tilde\tau$. 
Then
\cite{Dubrovin:1994hc} there exist local coordinates
$\{u^i\}_{i=1}^{N_\cZ}$ such that their coordinate vector fields $\de_{u^i}\in
\mathfrak{X}(V_\tau)$ are a basis of idempotents of $\circ_{\tau}$. Denoting
by $\Delta_i$ their squared norm with respect to the flat pairing, we obtain a normalized orthonormal system $\tilde u^i \triangleq u^i/\sqrt{\Delta_i}$ of
local coordinates. Then an {\it $R$-calibration} for
$\cZ$ is a choice of a $\bbC((z))$-basis of horizontal 1-forms on $V_\tau$
\eq{
(R_\cZ)_i^j(\tau,z) \re^{u^j/z} \rd \tilde u^i, \quad i=1, \dots, N_\cZ
}
where $(R_\cZ)_i^j(\tau, z)$ is an asymptotic $\mathrm{End}(T_\tau H)$-valued series
in $z$ satisfying $\sum_j (R_\cZ)_i^j(\tau,z) (R_\cZ)_k^j(\tau, -z) = \delta_i^j$. If we
let $\Psi$ be the differential of $\tilde u$, the above implies that
$\Psi_\cZ R_\cZ \re^{u/z}$ is a holomorphic family of symplectomorphisms parameterized
by $\tau\in V_{\tilde\tau}$.\\

Givental's quantization formalism assigns quantum operators to the $S$- and
$R$-calibrations of $\cZ$. The {\it Fock space} $\bbF_\tau$ of $\cZ$ at $\tau$ is the space
of functions $f(\lambda, \mathrm{t}_\tau)$ of the form
\eq{
f(\lambda, \mathrm{t}_\tau) = 
\sum_{g\in \bbZ} \lambda^{g-1}f_g(\mathrm{t}_\tau)
}
where in terms of coordinates $\{q^{i,\a}\}$ we define
\eq{
\mathrm{t}_\tau^{\a,l}\triangleq q^{\a,l}-\delta^{l0}\tau^\alpha
+\delta^{l1}\delta^{\a0}
\label{eq:shiftvar}
}
and $f_g(\mathrm{t}_\tau)\in \bbC[[\mathrm{t}_\tau]]$. 
To each
symplectomorphism connected with the identity $\cQ \in \mathrm{Aut}_0(\HH_\cZ,
\Omega_\cZ)$,
we associate a quantized operator $\widehat \cQ$ acting on $\bbF_\tau$ via
%
$\widehat \cQ  = \re^{\widehat{\log \cQ}}$,
%
where we define the quantization of an infinitesimal symplectomorphism to be
the quantization of its quadratic Hamiltonian with the normal-ordering prescriptions
\eq{
\widehat{p_{\a,k} p_{\beta,l}} =  \lambda \frac{\de^2}{\de q^{\a,k} \de
  q^{\beta,l}}, \quad  \widehat{p_{\a,k} q^{\beta,l}} =  \widehat{q^{\beta,l} p_{\a,k}} = q^{\beta,l}\frac{\de}{\de q^{\a,k}}, \quad
  \widehat{q^{\a,k} q^{\beta,l}} = \lambda^{-1}q^{\a,k} q^{\beta,l}.
}

Let now $Z_\cZ\in\bbF_\tau$ be the generating function of disconnected
Gromov--Witten invariants of $\cZ$,
\eq{
Z_\cZ=\exp \sum_{g\geq 0}\lambda^{g-1}\cF_g^\cZ,
}
in the shifted variables \cref{eq:shiftvar}, and denote likewise by $Z_{\rm
  pt}\in (-q^1)^{-1/24}\bbC[1/q^1, q^{k>1}][[q^0]]$ the generating function
of disconnected Gromov--Witten invariants of the point in dilaton-unshifted
variables. In the context of
fixed point localization for toric orbifolds, knowledge of the $T$-action
fixes uniquely a canonical
choice
for the symplectomorphism $R_\cZ$  \cite{
MR1901075,
MR2578300}, which we 
call {\it the Gromov--Witten $R$-calibration}. The Givental--Teleman theorem
\cite{MR1866444, MR1901075, telemangiv} then asserts that $Z_\cZ$ can be
obtained from $Z_{\rm pt}$ via the action of operators $\widehat S_\cZ$ and
$\widehat R_\cZ$ obtained by quantizing the canonical $S$- and $R$-calibrations
of $QH(\cZ)$ at $\tau$, giving formula \cref{eq:qgiv}.

\section{$A_n$ resolutions}
\label{sec:an}

\subsection{GIT Quotients}
\label{sec:GIT}

Here we review the relevant toric geometry concerning our targets.  Let
$\cX\triangleq[\C^3/\Z_{n+1}]$ be the 3-fold $A_n$ singularity and $Y$ its resolution.
The toric fan for $\cX$ has rays $(0,0,1)$, $(1,0,0)$, and $(1,n+1,0)$, while
the fan for $Y$ is obtained by adding the rays $(1,1,0)$, $(1,2,0)$,...,
$(1,n,0)$.  The divisor class group is described by the short exact sequence
\eq{
0\longrightarrow\Z^{n}\stackrel{M^T}{\longrightarrow}\Z^{n+3}\stackrel{N}{\longrightarrow}\Z^3\longrightarrow
0,
\label{eq:divclass}
}
where
\eq{
M=\left[ \begin{array}{cccccccc}
1 & -2 & 1 & 0 & 0 &... &0 & 0\\
0 & 1 & -2 & 1 & 0 &... &0 & 0\\
\vdots & &\ddots & &\ddots & && \vdots\\
0 &... & 0 & 0 & 1 & -2 & 1 & 0
\end{array}
\right]
,\hspace{.5cm} N=\left[ \begin{array}{cccccc}
1 & 1 & 1 &  & 1 & 0\\
0 & 1 & 2 & ... & n+1 & 0\\
0 & 0 & 0 & & 0 & 1
\end{array}.
\right]
\label{eq:MN}
}
\\
Both $\cX$ and $Y$ are GIT quotients: 
\ea{\label{orbgit}
\cX &= \left[\frac{\C^{n+3}\setminus V(x_1\cdot...\cdot
    x_n)}{(\C^*)^n}\right], \\
Y &= \frac{\C^{n+3}\setminus V(I_1, \dots, I_{n+1}),
}{(\C^*)^n}
\label{resgit}
}
where 
\eq{
I_i=\prod_{j=0, j \neq i-1, i}^{n+1} x_j,
}
and the torus action is specified by $M$. 
From the quotient \cref{orbgit}, we can compute coordinates on the orbifold
\begin{equation}\label{orbcoords}
\left[\begin{array}{c}
z_1\\
z_2\\
z_3
\end{array}\right]
=
\left[\begin{array}{c}
x_0x_1^{\frac{n}{n+1}}x_2^{\frac{n-1}{n+1}}\cdot...\cdot x_n^{\frac{1}{n+1}}\\
x_1^{\frac{1}{n+1}}x_2^{\frac{2}{n+1}}\cdot...\cdot x_n^{\frac{n}{n+1}}x_{n+1}\\
x_{n+2}
\end{array}\right].
\end{equation}
These coordinates are only defined up to a choice of $(n+1)^{\rm st}$
root of unity for each $x_i$.  This accounts for a residual $\Z_{n+1}\subset
(\C^*)^n$ acting with dual representations on the first two coordinates.  We
identify this residual $\Z_{n+1}$ as the subgroup generated by 
$\left(\omega,\omega^2, \dots, \omega^n\right)\in(\C^*)^n$, where
  $\omega=\re^{\frac{2\pi \ri }{n+1}}$. This realizes the quotient
\cref{orbgit} as the 3-fold $A_n$ singularity where $\Z_{n+1}=\langle \omega
\rangle$ acts by $\omega \cdot(z_1,z_2,z_3)=(\omega z_1,\omega^{-1} z_2,z_3)$.  

\begin{rmk}\label{dualrmk}
The weights of the $\Z_{n+1}$ action on the corresponding fibers of $T\cX$ are
inverse to the weights on the local coordinates because a local trivialization
of the tangent bundle is given by $\frac{\partial}{\partial z^\alpha}$ where
$z^\alpha$ are the local coordinates. \\
\end{rmk}

\begin{figure}
\centering
\includegraphics{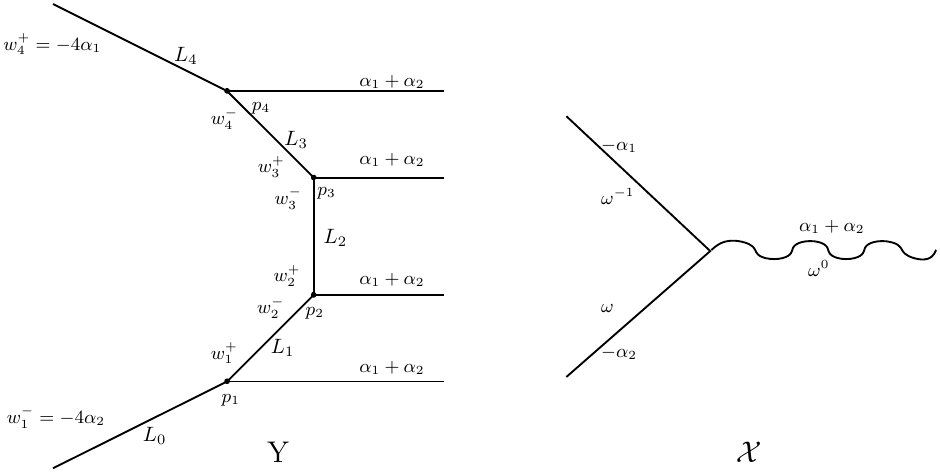}
\caption{The toric web diagrams for $Y$ and $\cX$ for $n=3$. Fixed points and invariants lines are labelled, together with the relevant torus and representation weights.}
\label{fig:web}
\end{figure}

The geometry of the space $Y$ is captured by the toric web diagram in \cref{fig:web}.  In particular, $Y$ has $n+1$ torus fixed points (corresponding
to the $n+1$ 3-dimensional cones in the fan) and a chain of $n$ torus
invariant lines connecting these points.  We label the points
$p_1$,...,$p_{n+1}$ where $p_i$ correspondes to the cone spanned by $(0,0,1)$,
$(1,i-1,0)$, and $(1,i,0)$ and we label the torus invariant lines by
$L_1$,...,$L_n$ where $L_i$ connects $p_i$ to $p_{i+1}$.  We also denote by
$L_0$ and $L_{n+1}$ the torus invariant (affine) lines corresponding to the
2-dimensional cones spanned by the rays $(1,0,0), (0,0,1)$ and $(1,n,0),
(0,0,1)$, respectively. From the quotient \cref{resgit} we compute
homogeneous coordinates on the line $L_i$
\begin{equation}\label{projcoords}
\left[\begin{array}{c}
x_0^ix_1^{i-1}\cdot...\cdot x_{i-1}\\
x_{n+1}^{n+1-i}x_{n-1}^{n-i}\cdot...\cdot x_{i+1}
\end{array}\right]
\end{equation}
where $p_i\leftrightarrow[0:1]$ and $p_{i+1}\leftrightarrow[1:0]$. \\

$H_2(Y)$ is generated by the torus invariant lines $L_i$.
Define $\phi_i\in H^2(Y)$ to be dual to $L_i$.  The $\phi_i$ form a
basis of $H^2(Y)$;  denote the corresponding line bundles by $\cO(\phi_i)$.
Note that $\cO(\phi_i)$ restricts to $\cO(1)$ on $L_i$ and $\cO$ on $L_j$ if
$j\neq i$ and this uniquely determines the line bundle $\cO(\phi_i)$. On the
orbifold, line bundles correspond to $\Z_{n+1}$ equivariant line bundles on
$\C^3$.  We denote $\cO_k$ the line bundle where $\Z_{n+1}$ acts on fibers
with weight $\omega^k$;  then, for example,
$T_\cX=\cO_{-1}\oplus\cO_{1}\oplus\cO_0$ where the subscripts are computed
modulo $n+1$ (c.f. \cref{dualrmk}).

\subsection{Classical equivariant geometry}\label{sec:elb}

Given that we are working with noncompact targets, all of our quantum
computations utilize Atiyah--Bott localization with respect to an additional
$T=\C^*$ action on our spaces. Let $T$ act on $\C^{n+3}$ with weights
$(\alpha_1,0,...,0,\alpha_2,-\alpha_1-\alpha_2)$.  Then the induced action on
the orbifold and resolution can be read off from the local coordinates in
\cref{orbcoords,projcoords}.  In particular, the three weights on the fibers
of $T_\cX$ are $-\alpha_1,-\alpha_2, \alpha_1+\alpha_2$.  As a
$H_T(\mathrm{pt})$-module, the $T$-equivariant
Chen--Ruan cohomology $H^\bullet(\cX)$ of $X$ is 
by definition the $T$-equivariant cohomology of the inertia stack
$\cI\cX$. The latter has components  $\cX_1, \dots, \cX_n, \cX_{n+1}$, the last being the untwisted sector\footnote{While it is more common to index the untwisted sector by $0$, we make this choice of notation for the sake of the computations of \cref{sec:j}, where certain matrices are triangular with this ordering.}: 
\ea{
\cX_k &= [\bbC/\bbZ_{n+1}], \quad 1\leq k \leq n, \nonumber \\
\cX_{n+1} &= [\bbC^3/\bbZ_{n+1}]
}
Writing $\fc_k$, $k=1, \dots, n+1$ for the fundamental class of
$\cX_k$ we obtain a $\bbC(\nu)$ basis of $H(\cX)$; the age-shifted grading
assigns degree $0$ to the fundamental class of the untwisted sector, and
degree $1$ to every twisted sector. 
The Atiyah--Bott localization isomorphism is trivial, i.e. the fundamental class on each twisted sector is identified with the unique $T$-fixed point on that sector.  We abuse notation and use $\fc_k$ to also denote the fixed point basis. 
The equivariant Chen--Ruan pairing in orbifold cohomology is
\eq{
\eta\l(\fc_i, \fc_j\r)_\cX = \frac{\delta_{i,n+1}\delta_{j,n+1}+\alpha_1\alpha_2 \delta_{i+j,n+1}}{\alpha_1 \alpha_2(\alpha_1+\alpha_2)(n+1)}.
}
\\

On the resolution $Y$, the three weights on the tangent bundle at $p_i$ are \eq{(w_i^-,w_i^+,\alpha_1+\alpha_2)\triangleq((i-1)\alpha_1+(-n+i-2)\alpha_2,-i\alpha_1+(n+1-i)\alpha_2,\alpha_1+\alpha_2).\label{eq:tweights}}
Moreover, $\cO(\phi_j)$ is canonically linearized via the homogeneous coordinates in \cref{orbcoords}.  The weight of $\cO(\phi_j)$ at the fixed point $p_i$ is
\begin{equation}\label{canwts}
\begin{cases}
(n+1-j)\alpha_2 & i\leq j,\\
j\alpha_1 & i>j.
\end{cases}
\end{equation}
%
Denote by $\{P_i\}_{i=1}^{n+1}$ the equivariant cohomology classes
corresponding to the fixed points of $Y$.  Choosing the canonical
linearization given in \cref{canwts}, the Atiyah--Bott localization
isomorphism on $Y$ is given by
\ea{
\label{eq:ab1}
\phi_j &\longrightarrow  \sum_{i\leq j}(n+1-j)\alpha_2 P_i +
\sum_{i>j}j\alpha_1P_i, \ \ \ \ \ \ \ \ j \not= n+1 \\
\phi_{n+1} & \longrightarrow \sum_{i=1}^{n+1}P_i.
\label{eq:ab2}
}
where $\phi_{n+1}$ is the fundamental class on $Y$.  
Genus zero, degree zero GW invariants are given by equivariant triple
intersections on $Y$,
\eq{
\bra \phi_i, \phi_j, \phi_k  \ket^Y_{0,3,0} = \int_Y\phi_i\cup\phi_j\cup\phi_k.
}
With $i\leq j \leq k<n+1$, \crefrange{eq:ab1}{eq:ab2} yield
\ea{
\bra \phi_{n+1}, \phi_{n+1}, \phi_{n+1}  \ket^Y_{0,3,0} &=
\frac{1}{(n+1)\alpha_1\alpha_2(\alpha_1+\alpha_2)},\label{ccorr1}\\
\bra \phi_{n+1}, \phi_{n+1}, \phi_i  \ket^Y_{0,3,0} &=
0,\label{ccorr2}\\
\bra \phi_{n+1}, \phi_i, \phi_j  \ket^Y_{0,3,0} &=
\frac{i(n+1-j)}{-(n+1)(\alpha_1+\alpha_2)}, \label{eq:eqpair}\\
\bra \phi_i, \phi_j, \phi_k  \ket^Y_{0,3,0} &=
-\frac{ij(n+1-k)\alpha_1+i(n+1-j)(n+1-k)\alpha_2}{(n+1)(\alpha_1+\alpha_2)}\label{ccorr4}.
}
The $T$-equivariant pairing $\eta\l(\phi_i,\phi_j\r)_Y$ is given by \cref{eq:eqpair} and diagonalizes in the fixed point basis:
\eq{
\eta\l(P_i,P_j\r)_Y=\frac{\delta_{i,j}}{w_i^-w_i^+(\alpha_1+\alpha_2)}.
}

\subsection{Quantum equivariant geometry}

We compute the genus $0$ GW invariants of $Y$ via localization (extending the computations of \cite[Section 2]{MR2411404} to a more general torus action):
\begin{equation}\label{qcorr}
\langle \phi_{i_1},....,\phi_{i_l} \rangle_{0,l,\beta} = \begin{cases}
- d^{l-3} & \text{if } \beta= d(L_j+...+L_k) \text{ with } j\leq\min\{i_\alpha\}\leq\max\{i_\alpha\}\leq k,\\
0 &\text{else.}
\end{cases} 
\end{equation}
Denote by $\Phi = \sum_{i=1}^{n+1} t_i \phi_i$ a general
cohomology class $\Phi \in H(Y)$.  The equivariant three-point correlators
used to define the quantum cohomology can be computed from \cref{ccorr1,ccorr2,eq:eqpair,ccorr4,qcorr}  (with $1\leq i\leq j\leq k < n+1$):
\begin{equation}\label{eq:yukY}
\bra\bra \phi_i,\phi_j,\phi_k\ket\ket_{0,3}^Y(t)=\int_Y\phi_i\cup\phi_j\cup\phi_k-\sum_{l\leq i \leq k\leq m}\frac{\re^{t_l+...+t_m}}{1-\re^{t_l+...+t_m}}.
\end{equation}

The equivariant quantum cohomology of $\cX$ can then be computed from the following result, which is proved in the appendix of \cite{MR2510741}.

\begin{prop}\label{thm:crc}
Let $\rho:[0,1]\to H^2(Y)$ be as in \eqref{eq:rho}.
Then upon analytic continuation in the quantum
parameters $t_i$ along $\rho$, the quantum products for $\X$ and $Y$ coincide
after the affine-linear change of variables
\eq{
t_i=\l(\widehat{\cI}^{\cX,Y}_\rho x\r)_i = \begin{cases}
\frac{2\pi\ri}{n+1}+\sum_{k=1}^{n}\frac{\omega^{-ik}(\omega^{\frac{k}{2}}-\omega^{-\frac{k}{2}})}{n+1}x_k &0<i<n+1\\
x_{n+1}, & i=n+1\\
\end{cases}
\label{eq:changevar}
}
and the linear isomorphism $\bbU_{\rho,0}^{\cX, Y}:H(\X)\rightarrow H(Y)$
\begin{align*}
\fc_k &\rightarrow
\sum_{i=1}^n\frac{\omega^{-ik}(\omega^{\frac{k}{2}}-\omega^{-\frac{k}{2}})}{n+1}\phi_i, 
\\
\fc_{n+1} &\rightarrow \phi_{n+1}.
\end{align*}
Furthermore, $\bbU_{\rho,0}^{\cX, Y}$ 
preserves the equivariant
Poincar\'e pairings of $\cX$ and $Y$. \\
\end{prop}

\section{Analytic continuation of Lauricella $F_D^{(N)}$}
\label{sec:anFD}
Consider the Lauricella function $F_D^{(M+N)}(a; b_1, \dots,  b_{M+N}; c; z_1, 
\dots,z_M, w_1, \dots, w_N)$ around $P=(0,0, \dots, \infty, \dots,
\infty)$. We  are interested in the leading terms of the asymptotics of
this function in the region $\Omega_{M+N}$ defined as
\eq{
\label{eq:omegaMN}
\Omega_{M+N}
\triangleq B(P,\epsilon) \bigcap_{i<j} H_{ij} 
}
given by the intersection of the ball $B(P,\epsilon)$ with the interior of the real hyperquadrics
\eq{
H_{ij} \triangleq \l\{ (z,w) \in \bbC^{M+N}  \big| |w_i/w_j| < \epsilon\r\}.
}
As our interest is confined to the
{\it leading} asymptotics only, we can assume without loss of generality that $M=0$. \\

Following \cite[Chapter~6]{MR0422713}, start from the power series expression
\cref{eq:FD} and perform the sum w.r.t. $w_N$
\ea{
F_D^{(N)}(a; b_1, \dots, b_N; c; w_1, \dots, w_N) &= \sum_{i_1, \dots, i_{N-1}}
\frac{(a)_{\sum_{j=1}^{N-1} i_j}}{(c)_{\sum_{j=1}^{N-1} i_j}}\prod_{j=1}^{N-1}
\frac{(b_j)_{i_j} w_j^{i_j}}{i_j!} \nn \\ &  \quad {}_2F_1\l(a+\sum_{j=1}^{N-1} i_j, b_N, c+\sum_{j=1}^{N-1} i_j,w_N\r) .
\label{eq:FD2F1}
}
The main idea then is to apply the connection formula for the inner Gauss
function 
\ea{
\, _2F_1(a,b;c;z) &= \frac{(-z)^{-a} \Gamma (c) \Gamma (b-a) \,
  _2F_1\left(a,a-c+1;a-b+1;\frac{1}{z}\right)}{\Gamma (b) \Gamma (c-a)} \nn
\\ &+ \frac{(-z)^{-b} \Gamma (c) \Gamma (a-b) \,
   _2F_1\left(b,b-c+1;-a+b+1;\frac{1}{z}\right)}{\Gamma (a) \Gamma (c-b)}
\label{eq:2F1conn}
}
to analytically continue it to $|z|=|w_N|>1$; in doing so, we fix a path of
analytic continuation by choosing the principal branch for both the power functions
$(-z)^{-a}$ and $(-z)^{-b}$ in \cref{eq:2F1conn} and continue $\,
_2F_1(a,b;c;z)$ to $|z>1|$ along a path that has winding number zero around
the Fuchsian singularity at $z=1$. As a power series in $w_N$ the analytic continuation
of 
\cref{eq:FD2F1} around $w_N=\infty$ then reads
\ea{
 & F_D^{(N)}(a; b_1, \dots, b_N; c; w_1, \dots, w_N) = (-w_N)^{-a}
\Gamma\l[\bary{cc}c, & b_N-a \\ b_N, & c-a \eary\r] \nn \\ & \times  F_D^{(N)}\l(a; b_1, \dots,
b_{N-1}, 1-c+a; 1-b_N+a,\frac{w_1}{w_N}, \dots, \frac{1}{w_N}\r) 
+
(-w_N)^{-b_N}
\Gamma\l[\bary{cc}c, & a-b_N \\ a, & c-b_N \eary\r] \nn \\ &  \times C_N^{(N-1)}\l(b_1, \dots,
b_{N}, 1-c+b_N; a-b_N,-w_1,-w_2, \dots, \frac{1}{w_N}\r),
\label{eq:FDcont1}
}
where we defined \cite[Chapter~3]{MR0422713}
\ea{
C_N^{(k)}\l(b_1, \dots,
b_{N}, a; a',x_1, \dots, x_N\r) & \triangleq \sum_{i_1, \dots, i_{N}}
(a)_{\a_N^{(k)}(\mathbf{i})}(a')_{-\a_N^{(k)}(\mathbf{i})} \prod_{j=1}^{N}
\frac{(b_j)_{i_j} w_j^{i_j}}{i_j!} 
}
and 
\ea{
\a_N^{(k)}(\mathbf{i}) &\triangleq \sum_{j=k+1}^{N} i_j-\sum_{j=1}^{k} i_j,\\
\Gamma\l[\bary{ccc}a_1, & \dots, & a_m \\ b_1, & \dots, & b_n \eary\r] &\triangleq
\frac{\prod_{i=1}^m \Gamma(a_i)}{\prod_{i=1}^l\Gamma(b_i)}.
}
Now, notice that the $F_D^{(N-1)}$ function in the right hand side of \cref{eq:FDcont1} is
analytic in $\Omega_N$; there is nothing more that should be done there. The
analytic continuation of the $C_N^{(N-1)}$ function is instead much more involved (see
\cite{MR0422713} for a complete treatment of the $N=3$ case); but as all we
are interested in is the leading term of the expansion around $P$ in
$\Omega_N$ we isolate the $\cO(1)$ term in its $1/w_N$ expansion to find
\ea{
 & C_N^{(N-1)}\l(b_1, \dots,
b_{N}, 1-c+b_N; a-b_N,-w_1,-w_2, \dots, \frac{1}{w_N}\r) = \nn \\
&= F_D^{(N-1)}\l(a-b_N,b_1, \dots,
b_{N-1}, c-b_N; w_1, \dots, w_{N-1}\r)+ \cO\l(\frac{1}{w_{N}}\r)
\label{eq:CNFD}
}
We are done: by \cref{eq:CNFD}, the form of the leading terms in the expansion of
$F_D^{(N)}$ inside $\Omega_N$ can be found recursively by iterating $N$ times the
procedure we have followed in \crefrange{eq:FD2F1}{eq:CNFD}; as at each step \crefrange{eq:2F1conn}{eq:CNFD} generate
one additional term, we end up with a sum of $N+1$ monomials each having
power-like monodromy around $P$. Explicitly:
\ea{
F_D^{(N)}(a; b_1, \dots, b_N; c; w_1, \dots, w_N) & \sim  
\sum_{j=0}^{N-1}\Gamma\l[\bary{ccc}c, & a-\sum_{i=N-j+1}^N b_i, & \sum_{i=n-j}^N
  b_i-a \\ a, & b_{N-j}, & c-a \eary\r] \nn \\ & \quad \prod_{i=1}^j (-w_{N-i+1})^{-b_{N-i+1}}
(-w_{N-j})^{-a+\sum_{i=N-j+1}^Nb_i}\nn \\ &+ \prod_{i=1}^N (-w_i)^{-b_i}
\Gamma\l[\bary{cc}c, & a-\sum_{i=1}^N b_j \\ a, & c-\sum_{i=1}^N b_j \eary\r].
\label{eq:fdinf}
}

\begin{rmk}
The analytic continuation to some other sectors of the ball $B(P,\epsilon)$ is
straightforward. In particular we can replace the condition $w_i/w_j \sim 0$
for $j>i$ by its reciprocal $w_j/w_i \sim 0$; this amounts to relabeling $b_i
\to b_{N-i+1}$ in \cref{eq:fdinf}.
\label{rmk:relabel}
\end{rmk}

\begin{rmk}\label{rmk:toscano}
 When $a=-d$ for $d\in \bbZ^+$, the function $F_D^{(N)}$ reduces to
  a polynomial in $w_1, \dots, w_N$. In this case the arguments above reduce
  to a formula of Toscano \cite{MR0340663} for Lauricella polynomials:
\ea{
 & F_D^{(N)}(-d; b_1, \dots, b_N; c; w_1, \dots, w_N)  \nn \\ &= (-w_N)^{d}\frac{(b)_d}{(c)_d}
 F_D^{(N)}\l(-d; b_1, b_2 \dots,
b_{N-1};1-d-c, 1-d-b_N,\frac{w_1}{w_N}, \dots, \frac{1}{w_N}\r).
\label{eq:FDtosc}
}
\end{rmk}

\end{appendix}

\bibliography{miabiblio}

\def\cprime{$'$} \def\cprime{$'$}
\begin{bibdiv}
\begin{biblist}

\bib{MR2450211}{article}{
      author={Abramovich, Dan},
      author={Graber, Tom},
      author={Vistoli, Angelo},
       title={Gromov-{W}itten theory of {D}eligne-{M}umford stacks},
        date={2008},
        ISSN={0002-9327},
     journal={Amer. J. Math.},
      volume={130},
      number={5},
       pages={1337\ndash 1398},
         url={http://dx.doi.org/10.1353/ajm.0.0017},
      review={\MR{2450211 (2009k:14108)}},
}

\bib{Aganagic:2003db}{article}{
      author={Aganagic, Mina},
      author={Klemm, Albrecht},
      author={Mari\~no, Marcos},
      author={Vafa, Cumrun},
       title={{The topological vertex}},
        date={2005},
     journal={Commun. Math. Phys.},
      volume={254},
       pages={425\ndash 478},
      eprint={hep-th/0305132},
}

\bib{Aganagic:2001nx}{article}{
      author={Aganagic, Mina},
      author={Klemm, Albrecht},
      author={Vafa, Cumrun},
       title={{Disk instantons, mirror symmetry and the duality web}},
        date={2002},
     journal={Z. Naturforsch.},
      volume={A57},
       pages={1\ndash 28},
      eprint={hep-th/0105045},
}

\bib{Aganagic:2000gs}{article}{
      author={Aganagic, Mina},
      author={Vafa, Cumrun},
       title={{Mirror symmetry, D-branes and counting holomorphic discs}},
        date={2000},
      eprint={hep-th/0012041},
}

\bib{bfk:window}{article}{
      author={Ballard, Matthew},
      author={Favero, David},
      author={Katzarkov, Ludmil},
       title={Variation of geometric invariant theory quotients and derived
  categories},
        date={2012},
      eprint={1203.6643},
}

\bib{MR1328251}{article}{
      author={Batyrev, Victor~V.},
      author={van Straten, Duco},
       title={Generalized hypergeometric functions and rational curves on
  {C}alabi-{Y}au complete intersections in toric varieties},
        date={1995},
        ISSN={0010-3616},
     journal={Comm. Math. Phys.},
      volume={168},
      number={3},
       pages={493\ndash 533},
         url={http://projecteuclid.org/getRecord?id=euclid.cmp/1104272487},
      review={\MR{1328251 (96g:32037)}},
}

\bib{MR2271990}{article}{
      author={Borisov, Lev~A.},
      author={Horja, R.~Paul},
       title={Mellin-{B}arnes integrals as {F}ourier-{M}ukai transforms},
        date={2006},
        ISSN={0001-8708},
     journal={Adv. Math.},
      volume={207},
      number={2},
       pages={876\ndash 927},
         url={http://dx.doi.org/10.1016/j.aim.2006.01.011},
      review={\MR{2271990 (2007m:14056)}},
}

\bib{Bouchard:2007ys}{article}{
      author={Bouchard, Vincent},
      author={Klemm, Albrecht},
      author={Mari{\~n}o, Marcos},
      author={Pasquetti, Sara},
       title={{Remodeling the B-model}},
        date={2009},
     journal={Commun. Math. Phys.},
      volume={287},
       pages={117\ndash 178},
      eprint={arXiv:0709.1453},
}

\bib{Bouchard:2008gu}{article}{
      author={Bouchard, Vincent},
      author={Klemm, Albrecht},
      author={Marino, Marcos},
      author={Pasquetti, Sara},
       title={{Topological open strings on orbifolds}},
        date={2010},
     journal={Commun.Math.Phys.},
      volume={296},
       pages={589\ndash 623},
      eprint={0807.0597},
}

\bib{talk-banff}{incollection}{
      author={Brini, Andrea},
       title={A crepant resolution conjecture for open strings, talk at the
  {BIRS} {W}orkshop on \textit{``{N}ew recursion formulae and integrability for
  {C}alabi--{Y}au manifolds''}, {O}ctober ~26},
        date={2011},
}

\bib{Brini:2011ij}{article}{
      author={Brini, Andrea},
       title={{Open topological strings and integrable hierarchies: Remodeling
  the A-model}},
        date={2012},
     journal={Commun.Math.Phys.},
      volume={312},
       pages={735\ndash 780},
      eprint={1102.0281},
}

\bib{Brini:2014mha}{article}{
      author={Brini, Andrea},
      author={Carlet, Guido},
      author={Romano, Stefano},
      author={Rossi, Paolo},
       title={{Rational reductions of the 2D-Toda hierarchy and mirror
  symmetry}},
        date={2017},
     journal={{\rm to appear on} J.~Eur.~Math.Soc.},
      eprint={1401.5725},
}

\bib{MR2861610}{article}{
      author={Brini, Andrea},
      author={Cavalieri, Renzo},
       title={Open orbifold {G}romov-{W}itten invariants of {$[\Bbb C\sp 3/\Bbb
  Z\sb n]$}: localization and mirror symmetry},
        date={2011},
        ISSN={1022-1824},
     journal={Selecta Math. (N.S.)},
      volume={17},
      number={4},
       pages={879\ndash 933},
      eprint={1007.0934},
         url={http://dx.doi.org/10.1007/s00029-011-0060-4},
      review={\MR{2861610 (2012i:14068)}},
}

\bib{Brini:2014fea}{article}{
      author={Brini, Andrea},
      author={Cavalieri, Renzo},
       title={{Crepant Resolutions and Open Strings II}},
        date={2014},
      eprint={1407.2571},
}

\bib{Brini:2008rh}{article}{
      author={Brini, Andrea},
      author={Tanzini, Alessandro},
       title={{Exact results for topological strings on resolved Y(p,q)
  singularities}},
        date={2009},
     journal={Commun. Math. Phys.},
      volume={289},
       pages={205\ndash 252},
      eprint={0804.2598},
}

\bib{MR2411404}{article}{
      author={Bryan, Jim},
      author={Gholampour, Amin},
       title={Root systems and the quantum cohomology of {$ADE$} resolutions},
        date={2008},
        ISSN={1937-0652},
     journal={Algebra Number Theory},
      volume={2},
      number={4},
       pages={369\ndash 390},
         url={http://dx.doi.org/10.2140/ant.2008.2.369},
      review={\MR{2411404 (2009b:14100)}},
}

\bib{MR2483931}{article}{
      author={Bryan, Jim},
      author={Graber, Tom},
       title={The crepant resolution conjecture},
        date={2009},
     journal={Proc. Sympos. Pure Math.},
      volume={80},
       pages={23\ndash 42},
      review={\MR{2483931 (2009m:14083)}},
}

\bib{cavalieri2011open}{article}{
      author={Cavalieri, Renzo},
      author={Ross, Dustin},
       title={Open {G}romov--{W}itten theory and the crepant resolution
  conjecture},
        date={2013},
     journal={Mich. Math. Journal, {\rm to appear}},
      eprint={1102.0717},
}

\bib{2013arXiv1306.0437C}{article}{
      author={{Chan}, K.},
      author={{Cho}, C.-H.},
      author={{Lau}, S.-C.},
      author={{Tseng}, H.-H.},
       title={{Gross fibrations, SYZ mirror symmetry, and open Gromov-Witten
  invariants for toric Calabi-Yau orbifolds}},
        date={2013},
      eprint={1306.0437},
}

\bib{hht2}{article}{
      author={{Chan, Kwokwai and Cho, Cheol-Hyun and Lau, Siu-Cheong and Tseng,
  Hsian-Hua}},
       title={{Lagrangian Floer superpotentials and crepant resolutions for
  toric orbifolds}},
        date={2012},
      eprint={1208.5282},
}

\bib{hht1}{article}{
      author={{Chan, Kwokwai and Lau, Siu-Cheong and Leung, Naichung Conan and
  Tseng, Hsian-Hua}},
       title={{Open Gromov-Witten invariants and mirror maps for semi-Fano
  toric manifolds}},
        date={2011},
      eprint={1112.0388},
}

\bib{MR1950941}{article}{
      author={Chen, Weimin},
      author={Ruan, Yongbin},
       title={Orbifold {G}romov-{W}itten theory},
        date={2002},
     journal={Contemp. Math.},
      volume={310},
       pages={25\ndash 85},
         url={http://dx.doi.org/10.1090/conm/310/05398},
      review={\MR{1950941 (2004k:53145)}},
}

\bib{MR2486673}{article}{
      author={Coates, Tom},
       title={On the crepant resolution conjecture in the local case},
        date={2009},
        ISSN={0010-3616},
     journal={Comm. Math. Phys.},
      volume={287},
      number={3},
       pages={1071\ndash 1108},
         url={http://dx.doi.org/10.1007/s00220-008-0715-y},
      review={\MR{2486673 (2010j:14098)}},
}

\bib{ccit2}{article}{
      author={Coates, Tom},
      author={Corti, Alessio},
      author={Iritani, Hiroshi},
      author={Tseng, Hsian-Hua},
       title={A mirror theorem for toric stacks},
      eprint={1310.4163},
}

\bib{MR2510741}{article}{
      author={Coates, Tom},
      author={Corti, Alessio},
      author={Iritani, Hiroshi},
      author={Tseng, Hsian-Hua},
       title={Computing genus-zero twisted {G}romov-{W}itten invariants},
        date={2009},
        ISSN={0012-7094},
     journal={Duke Math. J.},
      volume={147},
      number={3},
       pages={377\ndash 438},
         url={http://dx.doi.org/10.1215/00127094-2009-015},
      review={\MR{MR2510741}},
}

\bib{Coates:2012vs}{article}{
      author={Coates, Tom},
      author={Iritani, Hiroshi},
       title={{On the Convergence of Gromov-Witten Potentials and Givental's
  Formula}},
        date={2012},
      eprint={1203.4193},
}

\bib{cij}{unpublished}{
      author={Coates, Tom},
      author={Iritani, Hiroshi},
      author={Jiang, Yunfeng},
       title={{\rm in preparation}},
}

\bib{MR2529944}{article}{
      author={Coates, Tom},
      author={Iritani, Hiroshi},
      author={Tseng, Hsian-Hua},
       title={Wall-crossings in toric {G}romov-{W}itten theory. {I}. {C}repant
  examples},
        date={2009},
        ISSN={1465-3060},
     journal={Geom. Topol.},
      volume={13},
      number={5},
       pages={2675\ndash 2744},
         url={http://dx.doi.org/10.2140/gt.2009.13.2675},
      review={\MR{2529944 (2010i:53173)}},
}

\bib{coates2007quantum}{article}{
      author={Coates, Tom},
      author={Ruan, Yongbin},
       title={Quantum cohomology and crepant resolutions: A conjecture},
        date={2013},
     journal={Ann.~Inst.~Fourier, {\rm to appear}},
      eprint={0710.5901},
}

\bib{MR1677117}{book}{
      author={Cox, David~A.},
      author={Katz, Sheldon},
       title={Mirror symmetry and algebraic geometry},
      series={Mathematical Surveys and Monographs},
   publisher={American Mathematical Society},
     address={Providence, RI},
        date={1999},
      volume={68},
        ISBN={0-8218-1059-6},
      review={\MR{MR1677117 (2000d:14048)}},
}

\bib{MR849651}{article}{
      author={Deligne, P.},
      author={Mostow, G.~D.},
       title={Monodromy of hypergeometric functions and nonlattice integral
  monodromy},
        date={1986},
        ISSN={0073-8301},
     journal={Inst. Hautes \'Etudes Sci. Publ. Math.},
      number={63},
       pages={5\ndash 89},
         url={http://www.numdam.org/item?id=PMIHES_1986__63__5_0},
      review={\MR{849651 (88a:22023a)}},
}

\bib{Diaconescu:2003qa}{article}{
      author={Diaconescu, Duiliu-Emanuel},
      author={Florea, Bogdan},
       title={{Localization and gluing of topological amplitudes}},
        date={2005},
     journal={Commun.Math.Phys.},
      volume={257},
       pages={119\ndash 149},
      eprint={hep-th/0309143},
}

\bib{Dijkgraaf:1990dj}{article}{
      author={Dijkgraaf, Robbert},
      author={Verlinde, Herman~L.},
      author={Verlinde, Erik~P.},
       title={{Topological strings in $d<1$}},
        date={1991},
     journal={Nucl.Phys.},
      volume={B352},
       pages={59\ndash 86},
}

\bib{donovan-segal}{article}{
      author={Donovan, Will},
      author={Segal, Ed},
       title={{Mixed braid group actions from deformations of surface
  singularities}},
        date={2013},
      eprint={1310.7877},
}

\bib{Dubrovin:1992eu}{article}{
      author={Dubrovin, Boris},
       title={{Hamiltonian formalism of Whitham type hierarchies and
  topological Landau-Ginsburg models}},
        date={1992},
     journal={Commun.Math.Phys.},
      volume={145},
       pages={195\ndash 207},
}

\bib{Dubrovin:1994hc}{article}{
      author={Dubrovin, Boris},
       title={Geometry of {$2$}{D} topological field theories},
        date={1994},
     journal={{\rm in} ``Integrable systems and quantum groups'' ({M}ontecatini
  {T}erme, 1993), Lecture Notes in Math.},
      volume={1620},
       pages={120\ndash 348},
      eprint={hep-th/9407018},
      review={\MR{MR1397274 (97d:58038)}},
}

\bib{Dubrovin:1998fe}{article}{
      author={Dubrovin, Boris},
       title={{Painleve transcendents and two-dimensional topological field
  theory}},
        date={1998},
     journal={CRM Ser. Math. Phys.},
       pages={287\ndash 412},
      eprint={math/9803107},
}

\bib{MR2070050}{article}{
      author={Dubrovin, Boris},
       title={On almost duality for {F}robenius manifolds, in ``{G}eometry,
  topology, and mathematical physics"},
        date={2004},
     journal={Amer. Math. Soc. Transl. Ser. 2},
      volume={212},
       pages={75\ndash 132},
      review={\MR{2070050 (2005e:53142)}},
}

\bib{MR0422713}{book}{
      author={Exton, Harold},
       title={Multiple hypergeometric functions and applications},
   publisher={Ellis Horwood Ltd.},
     address={Chichester},
        date={1976},
        note={Foreword by L. J. Slater, Mathematics \& its Applications},
      review={\MR{0422713 (54 \#10699)}},
}

\bib{fang2012open}{article}{
      author={Fang, Bohan},
      author={Liu, Chiu-Chu~Melissa},
      author={Tseng, Hsian-Hua},
       title={Open-closed {G}romov-{W}itten invariants of 3-dimensional
  {C}alabi-{Y}au smooth toric {DM} stacks},
        date={2012},
      eprint={1212.6073},
}

\bib{Fang:2016svw}{article}{
      author={Fang, Bohan},
      author={Liu, Chiu-Chu~Melissa},
      author={Zong, Zhengyu},
       title={{On the Remodeling Conjecture for Toric Calabi-Yau 3-Orbifolds}},
        date={2016},
      eprint={1604.07123},
}

\bib{MR1408320}{article}{
      author={Givental, Alexander},
       title={Equivariant {G}romov-{W}itten invariants},
        date={1996},
        ISSN={1073-7928},
     journal={Internat. Math. Res. Notices},
      number={13},
       pages={613\ndash 663},
         url={http://dx.doi.org/10.1155/S1073792896000414},
      review={\MR{1408320 (97e:14015)}},
}

\bib{MR1653024}{article}{
      author={Givental, Alexander},
       title={A mirror theorem for toric complete intersections},
        date={1998},
     journal={Progr. Math.},
      volume={160},
       pages={141\ndash 175},
      review={\MR{1653024 (2000a:14063)}},
}

\bib{MR1901075}{article}{
      author={Givental, Alexander},
       title={Gromov-{W}itten invariants and quantization of quadratic
  {H}amiltonians},
        date={2001},
        ISSN={1609-3321},
     journal={Mosc. Math. J.},
      volume={1},
      number={4},
       pages={551\ndash 568, 645},
        note={Dedicated to the memory of I. G. Petrovskii on the occasion of
  his 100th anniversary},
      review={\MR{MR1901075 (2003j:53138)}},
}

\bib{MR1866444}{article}{
      author={Givental, Alexander},
       title={Semisimple {F}robenius structures at higher genus},
        date={2001},
        ISSN={1073-7928},
     journal={Internat. Math. Res. Notices},
      number={23},
       pages={1265\ndash 1286},
         url={http://dx.doi.org/10.1155/S1073792801000605},
      review={\MR{1866444 (2003b:53092)}},
}

\bib{MR2115767}{article}{
      author={Givental, Alexander},
       title={Symplectic geometry of {F}robenius structures},
        date={2004},
     journal={Aspects Math., E36},
       pages={91\ndash 112},
      eprint={math/0305409},
      review={\MR{2115767 (2005m:53172)}},
}

\bib{Gopakumar:1998ki}{article}{
      author={Gopakumar, Rajesh},
      author={Vafa, Cumrun},
       title={{On the gauge theory/geometry correspondence}},
        date={1999},
     journal={Adv. Theor. Math. Phys.},
      volume={3},
       pages={1415\ndash 1443},
      eprint={hep-th/9811131},
}

\bib{hhp:window}{article}{
      author={Herbst, Manfred},
      author={Hori, Kentaro},
      author={Page, David},
       title={{Phases Of N=2 Theories In 1+1 Dimensions With Boundary}},
        date={2008},
      eprint={0803.2045},
}

\bib{MR2700280}{book}{
      author={Horja, Richard~Paul},
       title={Hypergeometric functions and mirror symmetry in toric varieties},
   publisher={ProQuest LLC, Ann Arbor, MI},
        date={1999},
        ISBN={978-0599-62147-3},
  url={http://gateway.proquest.com/openurl?url_ver=Z39.88-2004&rft_val_fmt=info:ofi/fmt:kev:mtx:dissertation&res_dat=xri:pqdiss&rft_dat=xri:pqdiss:9958768},
        note={Thesis (Ph.D.)--Duke University},
      review={\MR{2700280}},
}

\bib{MR2553377}{article}{
      author={Iritani, Hiroshi},
       title={An integral structure in quantum cohomology and mirror symmetry
  for toric orbifolds},
        date={2009},
        ISSN={0001-8708},
     journal={Adv. Math.},
      volume={222},
      number={3},
       pages={1016\ndash 1079},
         url={http://dx.doi.org/10.1016/j.aim.2009.05.016},
      review={\MR{2553377 (2010j:53182)}},
}

\bib{MR2683208}{incollection}{
      author={Iritani, Hiroshi},
       title={Ruan's conjecture and integral structures in quantum cohomology},
        date={2010},
   booktitle={New developments in algebraic geometry, integrable systems and
  mirror symmetry ({RIMS}, {K}yoto, 2008)},
      series={Adv. Stud. Pure Math.},
      volume={59},
   publisher={Math. Soc. Japan},
     address={Tokyo},
       pages={111\ndash 166},
      review={\MR{2683208 (2011h:14081)}},
}

\bib{jk:bg}{article}{
      author={Jarvis, Tyler~J.},
      author={Kimura, Takashi},
       title={Orbifold quantum cohomology of the classifying space of a finite
  group},
        date={2002},
     journal={{\rm in} Orbifolds in mathematics and physics ({M}adison, {WI},
  2001), Contemp. Math.},
      volume={310},
       pages={123\ndash 134},
      review={\MR{MR1950944 (2004a:14056)}},
}

\bib{MR1816048}{article}{
      author={Kapovich, Michael},
      author={Millson, John~J.},
       title={Quantization of bending deformations of polygons in {$\bold E\sp
  3$}, hypergeometric integrals and the {G}assner representation},
        date={2001},
        ISSN={0008-4395},
     journal={Canad. Math. Bull.},
      volume={44},
      number={1},
       pages={36\ndash 60},
         url={http://dx.doi.org/10.4153/CMB-2001-006-3},
      review={\MR{1816048 (2001m:53159)}},
}

\bib{Katz:2001vm}{article}{
      author={Katz, Sheldon~H.},
      author={Liu, Chiu-Chu~Melissa},
       title={{Enumerative geometry of stable maps with Lagrangian boundary
  conditions and multiple covers of the disc}},
        date={2002},
     journal={Geom. Topol. Monogr.},
      volume={8},
       pages={1\ndash 47},
      eprint={math/0103074},
}

\bib{kezhou}{article}{
      author={Ke, Huazhong},
      author={Zhou, Jian},
       title={Quantum {M}c{K}ay correspondence via gauged linear sigma model,
  {\rm in preparation}},
}

\bib{krawitz2011landau}{article}{
      author={Krawitz, Marc},
      author={Shen, Yefeng},
       title={{L}andau--{G}inzburg/{C}alabi--{Y}au correspondence of all genera
  for elliptic orbifold $\mathbb{P}^1$},
        date={2011},
      eprint={1106.6270},
}

\bib{Krichever:1992qe}{article}{
      author={Krichever, Igor~M.},
       title={{The tau function of the universal Whitham hierarchy, matrix
  models and topological field theories}},
        date={1994},
     journal={Commun. Pure Appl. Math.},
      volume={47},
       pages={437},
      eprint={hep-th/9205110},
}

\bib{lauric}{article}{
      author={Lauricella, G.},
       title={Sulle funzioni ipergeometriche a piu variabili},
    language={Italian},
        date={1893},
        ISSN={0009-725X},
     journal={Rendiconti del Circolo Matematico di Palermo},
      volume={7},
       pages={111\ndash 158},
         url={http://dx.doi.org/10.1007/BF03012437},
}

\bib{lee2004frobenius}{article}{
      author={Lee, YP},
      author={Pandharipande, R},
       title={Frobenius manifolds, {G}romov--{W}itten theory, and {V}irasoro
  constraints part ii, in preparation, available at
  \texttt{https://www.math.utah.edu/$\sim$yplee/research/Part2.ps}},
        date={2004},
}

\bib{Lerche:2001cw}{article}{
      author={Lerche, W.},
      author={Mayr, P.},
       title={{On $N = 1$ mirror symmetry for open type II strings}},
        date={2001},
      eprint={hep-th/0111113},
}

\bib{Li:2004uf}{article}{
      author={Li, Jun},
      author={Liu, Chiu-Chu~Melissa},
      author={Liu, Kefeng},
      author={Zhou, Jian},
       title={{A Mathematical Theory of the Topological Vertex}},
        date={2009},
     journal={Geom. Topol.},
      volume={13},
       pages={527\ndash 621},
      eprint={math/0408426},
}

\bib{Li:2001sg}{article}{
      author={Li, Jun},
      author={Song, Yun~S.},
       title={{Open string instantons and relative stable morphisms}},
        date={2002},
     journal={Adv.Theor.Math.Phys.},
      volume={5},
       pages={67\ndash 91},
      eprint={hep-th/0103100},
}

\bib{Marino:2001re}{article}{
      author={Mari{\~n}o, Marcos},
      author={Vafa, Cumrun},
       title={Framed knots at large {$N$}},
        date={2002},
     journal={Contemp. Math.},
      volume={310},
       pages={185\ndash 204},
      eprint={hep-th/0108064},
      review={\MR{MR1950947 (2004c:57024)}},
}

\bib{moop}{article}{
      author={Maulik, D.},
      author={Oblomkov, A.},
      author={Okounkov, A.},
      author={Pandharipande, R.},
       title={Gromov-{W}itten/{D}onaldson-{T}homas correspondence for toric
  3-folds},
        date={2008},
     journal={Inventiones Mathematicae},
       pages={1\ndash 45},
      eprint={arXiv:0809.3976},
}

\bib{MR2962392}{article}{
      author={Mimachi, Katsuhisa},
      author={Sasaki, Takeshi},
       title={Irreducibility and reducibility of {L}auricella's system of
  differential equations {$E\sb {\rm D}$} and the {J}ordan-{P}ochhammer
  differential equation {$E\sb {\rm JP}$}},
        date={2012},
        ISSN={1340-6116},
     journal={Kyushu J. Math.},
      volume={66},
      number={1},
       pages={61\ndash 87},
         url={http://dx.doi.org/10.2206/kyushujm.66.61},
      review={\MR{2962392}},
}

\bib{MR2425184}{article}{
      author={Pandharipande, R.},
      author={Solomon, J.},
      author={Walcher, J.},
       title={Disk enumeration on the quintic 3-fold},
        date={2008},
        ISSN={0894-0347},
     journal={J. Amer. Math. Soc.},
      volume={21},
      number={4},
       pages={1169\ndash 1209},
         url={http://dx.doi.org/10.1090/S0894-0347-08-00597-3},
      review={\MR{2425184 (2009j:14075)}},
}

\bib{2012arXiv1210.2312R}{article}{
      author={Romano, Stefano},
       title={{Frobenius structures on double Hurwitz spaces}},
        date={2012},
      eprint={1210.2312},
}

\bib{phdthesis-romano}{book}{
      author={Romano, Stefano},
       title={Special {F}robenius structures on {H}urwitz spaces and
  applications},
      series={PhD Thesis in Mathematical Physics},
   publisher={SISSA, Trieste},
        date={2012},
}

\bib{r:lgoa}{article}{
      author={Ross, Dustin},
       title={Localization and gluing of orbifold amplitudes: The
  {G}romov-{W}itten orbifold vertex},
        date={2013},
     journal={{T}rans. of the {A}{M}{S}, {\rm in press}},
      eprint={math/1109.5995v3},
}

\bib{MR2234886}{incollection}{
      author={Ruan, Yongbin},
       title={The cohomology ring of crepant resolutions of orbifolds},
        date={2006},
   booktitle={Contemp. {M}ath.},
      volume={403},
   publisher={Amer. Math. Soc.},
     address={Providence, RI},
       pages={117\ndash 126},
      review={\MR{MR2234886 (2007e:14093)}},
}

\bib{MR723468}{article}{
      author={Saito, Kyoji},
       title={Period mapping associated to a primitive form},
        date={1983},
        ISSN={0034-5318},
     journal={Publ. Res. Inst. Math. Sci.},
      volume={19},
      number={3},
       pages={1231\ndash 1264},
         url={http://dx.doi.org/10.2977/prims/1195182028},
      review={\MR{723468 (85h:32034)}},
}

\bib{ed:window}{article}{
      author={Segal, Ed},
       title={Equivalence between {GIT} quotients of {L}andau-{G}inzburg
  {B}-models},
        date={2011},
        ISSN={0010-3616},
     journal={Comm. Math. Phys.},
      volume={304},
      number={2},
       pages={411\ndash 432},
         url={http://dx.doi.org/10.1007/s00220-011-1232-y},
      review={\MR{2795327 (2012f:81265)}},
}

\bib{MR1831820}{article}{
      author={Seidel, Paul},
      author={Thomas, Richard},
       title={Braid group actions on derived categories of coherent sheaves},
        date={2001},
        ISSN={0012-7094},
     journal={Duke Math. J.},
      volume={108},
      number={1},
       pages={37\ndash 108},
         url={http://dx.doi.org/10.1215/S0012-7094-01-10812-0},
      review={\MR{1831820 (2002e:14030)}},
}

\bib{Solomon:2006dx}{article}{
      author={Solomon, Jake},
       title={{Intersection theory on the moduli space of holomorphic curves
  with Lagrangian boundary conditions}},
        date={2006},
      eprint={math/0606429},
}

\bib{telemangiv}{article}{
      author={Teleman, Constantin},
       title={{The structure of $2D$ semi--simple field theories}},
        date={2012},
        ISSN={0020-9910},
     journal={Invent.~Math.},
      volume={188},
      number={3},
      eprint={0712.0160},
}

\bib{MR0340663}{article}{
      author={Toscano, Letterio},
       title={Sui polinomi ipergeometrici a pi\`u variabili del tipo
  {$F\sb{D}$} di {L}auricella},
        date={1973},
     journal={Matematiche (Catania)},
      volume={27},
       pages={219\ndash 250},
      review={\MR{0340663 (49 \#5415)}},
}

\bib{MR2578300}{article}{
      author={Tseng, Hsian-Hua},
       title={Orbifold quantum {R}iemann-{R}och, {L}efschetz and {S}erre},
        date={2010},
        ISSN={1465-3060},
     journal={Geom. Topol.},
      volume={14},
      number={1},
       pages={1\ndash 81},
         url={http://dx.doi.org/10.2140/gt.2010.14.1},
      review={\MR{2578300 (2011c:14147)}},
}

\bib{MR1930577}{book}{
      author={Vassiliev, V.~A.},
       title={Applied {P}icard-{L}efschetz theory},
      series={Mathematical Surveys and Monographs},
   publisher={American Mathematical Society},
     address={Providence, RI},
        date={2002},
      volume={97},
        ISBN={0-8218-2948-3},
      review={\MR{1930577 (2003k:32043)}},
}

\bib{MR1424469}{book}{
      author={Whittaker, E.~T.},
      author={Watson, G.~N.},
       title={A course of modern analysis},
      series={Cambridge Mathematical Library},
   publisher={Cambridge University Press},
     address={Cambridge},
        date={1996},
        ISBN={0-521-58807-3},
        note={An introduction to the general theory of infinite processes and
  of analytic functions; with an account of the principal transcendental
  functions, Reprint of the fourth (1927) edition},
      review={\MR{1424469 (97k:01072)}},
}

\bib{zhou2008crepant}{article}{
      author={Zhou, Jian},
       title={Crepant resolution conjecture in all genera for type {A}
  singularities},
        date={2008},
      eprint={0811.2023},
}

\bib{zz}{article}{
      author={Zong, Zhengyu},
       title={Equivariant gromov-witten theory of gkm orbifolds},
        date={2016},
      eprint={1604.07270},
}

\end{biblist}
\end{bibdiv}
\bibliographystyle{amsalpha}

\end{document}